\documentclass{amsart}
\usepackage{amsmath,amssymb,amsthm,amsfonts}
\usepackage{paralist}
\usepackage{booktabs}
\usepackage[usenames,dvipsnames]{color}
\usepackage{comment}
\usepackage{graphicx}
\usepackage{latexsym}
\usepackage[latin1]{inputenc}
\usepackage[shortlabels]{enumitem}
\usepackage{thmtools}
\usepackage{url}
\usepackage[all,cmtip]{xy}
\usepackage{tikz}
\tikzset{font={\fontsize{4pt}{12}\selectfont}}
\usepackage{hyperref} %hyperref wants to be loaded last

\usepackage[enableskew]{youngtab}
\usepackage{ytableau}

\newcommand{\supp}{\mathsf{supp}}

\newcommand{\f}{{\bf f}}
\newcommand{\ssf}{\mathsf{f}}

\newcommand{\pleasant}{pleasant}
\newcommand{\RT}{\mathsf{RT}}

\title{Hook formulas for skew
  shapes I. $q$-analogues and bijections}

\author[Alejandro Morales, Igor Pak, Greta Panova]{Alejandro H.~Morales$^\star$,
\ \ Igor Pak$^\star$, \ and \ \ Greta Panova$^\dagger$}

\thanks{\today}
\thanks{\thinspace ${\hspace{-.45ex}}^\star$Department of Mathematics,
UCLA, Los Angeles, CA~90095.
\hskip.06cm
Email:
\hskip.06cm
\texttt{\{ahmorales,\ts{pak}\}@math.ucla.edu}}

\thanks{\thinspace ${\hspace{-.45ex}}^\dagger$Department of Mathematics,
 UPenn, Philadelphia, PA~19104.
\hskip.06cm
Email:
\hskip.06cm
\texttt{panova@math.upenn.edu}}

\hypersetup{
  breaklinks, % more discreet colours for links
  pdftitle={Hook-length formulas for skew shapes}
}

\newcommand{\tr}{\mathrm{tr}}
\newcommand{\shape}{\mathrm{shape}}
\newcommand{\reading}{\mathrm{reading}}

\newcommand{\RSK}{\Psi}
\newcommand{\HG}{\Phi}
\newcommand{\GH}{\Omega}

\newcommand{\vf}[1]{{#1}^{\leftrightarrow}}
\newcommand{\hf}[1]{{#1}^{\updownarrow}}

\newcommand{\dd}{\mathsf{d}}
\newcommand{\PD}{\mathcal{P}}
\newcommand{\ED}{\mathcal{E}}

\newcommand{\PP}{\operatorname{PP}}
\newcommand{\RPP}{\operatorname{RPP}}
\newcommand{\SSYT}{\operatorname{SSYT}}
\newcommand{\SYT}{\operatorname{SYT}}
\newcommand{\EqSYT}{\operatorname{EqSYT}}
\newcommand{\TYYT}{\operatorname{DYT}}

\DeclareMathOperator{\area}{area}

\DeclareMathOperator{\shpeaks}{shpk}
\DeclareMathOperator{\expeaks}{expk}
\DeclareMathOperator{\tmaj}{tmaj}
\DeclareMathOperator{\maj}{maj}

\declaretheorem[numberwithin=section]{theorem}
\declaretheorem[numberlike=theorem]{lemma}
\declaretheorem[numberlike=theorem]{proposition}
\declaretheorem[numberlike=theorem]{corollary}

\declaretheorem[numberlike=theorem, style=definition]{definition}
\declaretheorem[numberlike=theorem, style=definition]{remark}

\declaretheorem[numberlike=theorem, style=definition]{example}

\numberwithin{equation}{section} % requires package amsthm

\def\emp{\varnothing}
\def\sq{\square}

\def\la{\lambda}
\def\ga{\gamma}

\def\vp{\varphi}

\def\cP{\mathcal P}

\def\ssu{\subset}

\def\<{\langle}
\def\>{\rangle}

\def\0{{\mathbf 0}}

\def\SS{{S}}

\def\.{\hskip.06cm}
\def\ts{\hskip.03cm}

\def\nin{\noindent}

\begin{document}

\begin{abstract}
The celebrated \emph{hook-length formula} gives a product formula
for the number of standard Young tableaux of a straight shape.
In 2014, Naruse announced a more general formula for
the number of standard Young tableaux of skew shapes as a positive
sum over \emph{excited diagrams} of products of hook-lengths.
We give an algebraic and a combinatorial proof of Naruse's formula,
by using \emph{factorial Schur functions} and a generalization of the
\emph{Hillman--Grassl correspondence}, respectively.

The main new results are two different $q$-analogues of Naruse's formula: for the
skew Schur functions, and for counting reverse plane partitions of skew
shapes.  We establish explicit bijections between these objects and
families of integer arrays with certain nonzero entries, which also
proves the second formula.
\end{abstract}


\ytableausetup{smalltableaux}

\maketitle
%\tableofcontents

%----------------------------------------------------------------
\section{Introduction} \label{sec:intro}
%----------------------------------------------------------------

\subsection{Foreword}
The classical {\em hook-length formula} (HLF) for the number of
\emph{standard Young tableaux} (SYT) of a Young diagram, is a beautiful result
in enumerative combinatorics that is both mysterious and extremely well studied.
In a way it is a perfect formula -- highly nontrivial, clean, concise
and generalizing several others (binomial coefficients, Catalan numbers, etc.)
The HLF was discovered by Frame, Robinson and Thrall~\cite{FRT} in~1954,
and by now it has numerous proofs: probabilistic, bijective, inductive,
analytic, geometric, etc.\ (see~$\S$\ref{ss:finrem-GNW}). Arguably,
each of these proofs does not really explain the HLF on a deeper level,
but rather tells a different story, leading to new generalizations and
interesting connections to other areas.  In this paper we prove a new
generalization of the HLF for skew shapes which presented an unusual and
interesting challenge; it has yet to be fully explained and understood.

For skew shapes, there is no product formula for the number
$f^{\lambda/\mu}$ of standard Young tableaux (cf.~Section~\ref{sec:compare}).
Most recently, in the context of equivariant Schubert calculus,
Naruse presented and outlined a proof in~\cite{Strobl}
of a remarkable generalization on the HLF, which we call the
\emph{Naruse hook-length formula} (NHLF).  This formula (see below),
writes $f^{\lambda/\mu}$ as a sum of ``hook products'' over the
\emph{excited diagrams}, defined as certain generalizations of skew shapes.
These excited diagrams were introduced by Ikeda and Naruse~\cite{IkNa09},
and in a slightly different form independently by Kreiman~\cite{VK,VK2}
and Knutson, Miller and Yong~\cite{KMY}. They are a combinatorial model for the
terms appearing in the formula for {\em Kostant
  polynomials} discovered independently by Andersen, Jantzen and
  Soergel~\cite[Appendix D]{AJS}, and Billey~\cite{Bil}
(see Remark~\ref{rem:BilleyF} and~$\S$\ref{ss:finrem-hist}).
These diagrams are the main combinatorial objects in this paper and have difficult
structure even in nice special cases (cf.~\cite{MPP2} and Ex.~\ref{ex:skewhook}).

\smallskip

The goals of this paper are twofold.  First, we give Naruse-style hook
formulas for the Schur function $s_{\lambda/\mu}(1,q,q^2,\ldots)$,
which is the generating function for \emph{semistandard Young tableaux}
(SSYT) of shape~$\lambda/\mu$,
and for the generating function for \emph{reverse plane partitions} (RPP)
of the same shape.  Both can be viewed as $q$-analogues of~NHLF.
In contrast with the case of straight shapes, here these two formulas are
quite different.  Even the summations are over different sets --
in the case of RPP we
sum over \emph{pleasant diagrams} which we introduce.
The proofs employ a combination of algebraic and bijective
arguments, using the factorial Schur functions and the
Hillman--Grassl correspondence, respectively.  While the
algebraic proof uses some powerful known results, the bijective
proof is very involved and occupies much of the paper.

Second, as a biproduct of our proofs we give the first purely
combinatorial (but non-bijective) proof of Naruse's formula.
We also obtain {\rm trace generating functions} for both SSYT
and RPP of skew shape, simultaneously generalizing classical
Stanley and Gansner formulas, and our $q$-analogues.
We also investigate combinatorics of excited and pleasant
diagrams and how they related to each other, which allow us
simplify the RPP case.

\smallskip

\subsection{Hook formulas for straight and skew shapes}
We assume here the reader is familiar with the basic definitions, which
are postponed until the next two sections.

The {\em standard Young tableaux} (SYT)  of straight and skew shapes are
central objects in enumerative and algebraic combinatorics.  The number
$f^{\lambda} = |\SYT(\la)|$ of standard Young tableaux of shape $\lambda$ has the
celebrated {\em hook-length formula} (HLF):

\begin{theorem}[HLF; Frame--Robinson--Thrall~\cite{FRT}]
Let $\lambda$ be a partition of~$n$.  We have:
\begin{equation} \label{eq:hlf}
f^{\lambda} \, = \, \frac{n!}{\prod_{u\in [\lambda]} h(u)}\,,
\end{equation}
where $h(u)=\lambda_i-i+\lambda'_j-j+1$ is the {\em hook-length} of the
square $u=(i,j)$.
\end{theorem}

Most recently, Naruse generalized~\eqref{eq:hlf} as follows.
For a skew shape $\lambda/\mu$, an {\em excited diagram} is a subset of the
Young diagram~$[\lambda]$ of size $|\mu|$, obtained from the Young diagram~$[\mu]$
by a sequence of {\em excited moves}:
\begin{center}
\includegraphics{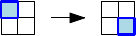}.
\end{center}
Such a move $(i,j) \to (i+1,j+1)$ is allowed only if cells $(i,j+1)$,
$(i+1,j)$ and $(i+1,j+1)$ are unoccupied (see the precise definition and
an example in~$\S$\ref{ss:excited-def}).
We use $\ED(\lambda/\mu)$ to denote
the set of excited diagrams of~$\lambda/\mu$.

\begin{theorem}[NHLF; Naruse \cite{Strobl}] \label{thm:IN}
Let $\lambda,\mu$ be partitions, such that $\mu \ssu \la$.  We have:
\begin{equation} \label{eq:Naruse}
f^{\lambda/\mu} \,  = \, |\la/\mu|! \, \sum_{D \in \ED(\lambda/\mu)}\,\.\.
 \prod_{u \in [\lambda]\setminus D} \frac{1}{ h(u)}\.\ts\..
\end{equation}
\end{theorem}

When $\mu = \emp$, there is a unique excited diagram $D=\emp$, and we obtain
the usual~HLF.

\smallskip

\subsection{Hook formulas for semistandard Young tableaux} \label{ss:intro-ssyt}
Recall that (a specialization of) a skew Schur function is the generating function for the
semistandard Young tableaux of shape~$\lambda/\mu$:
$$
s_{\lambda/\mu}(1,q,q^2,\ldots) \, = \, \sum_{\pi \in \SSYT(\lambda/\mu)} q^{|\pi|} \,.
$$
When $\mu=\emp$, Stanley found the following beautiful hook formula.

\begin{theorem}[Stanley \cite{St71}] \label{thm:HG}
\begin{equation} \label{eq:HG}
s_{\lambda}(1,q,q^2,\ldots) \, = \, q^{b(\lambda)}  \. \prod_{u\in [\lambda]} \. \frac{1}{1-q^{h(u)}}\,,
\end{equation}
where $b(\lambda)=\sum_i (i-1)\lambda_i$.
\end{theorem}

This formula can be viewed as $q$-analogue of the HLF.  In fact, one can
derive the HLF~\eqref{eq:hlf} from~\eqref{eq:HG} by Stanley's theory of
{\em $P$-partitions} \cite[Prop. 7.19.11]{EC2} or by a geometric argument~\cite[Lemma 1]{P1}.
Here we give the following natural analogue of NHLF~\eqref{eq:HG}.

\begin{theorem} \label{thm:skewSSYT}
We have:
\begin{equation} \label{eq:skewschur}
s_{\lambda/\mu}(1,q,q^2,\ldots) \, = \, \sum_{S\in \ED(\lambda/\mu)}
\. \.\.\prod_{(i,j) \in [\lambda]\setminus S}\frac{q^{\lambda'_j-i}}{1-q^{h(i,j)}}\ts.
\end{equation}
\end{theorem}

By analogy with the straight shape, Theorem~\ref{thm:skewSSYT} implies~NHLF,
see Proposition~\ref{cor:getnaruse}.
We prove Theorem~\ref{thm:skewSSYT} in Section~\ref{sec:algproof}
by using algebraic tools.

% Further, we prove that the restriction of the Hilman--Grassl
\smallskip

\subsection{Hook formulas for reverse plane partitions and SSYT via bijections} \label{ss:intro-rpp}
In the case of straight shapes, the enumeration of
RPP can be obtained from SSYT, by subtracting $(i-1)$ from the
entries in the $i$-th row.  In other words, we have:
\begin{equation} \label{eq:RPP-prod}
\sum_{\pi \in \RPP(\lambda)} q^{|\pi|}
\, = \, \prod_{u\in [\lambda]} \. \frac{1}{1-q^{h(u)}}\,.
\end{equation}
Note that the above relation does not hold for skew shapes, since entries
on the $i$-th row of a skew SSYT do not have to be at least~$(i-1)$.

Formula~\eqref{eq:RPP-prod} has a classical combinatorial proof
by the {\em Hillman--Grassl correspondence}~\cite{HG},
which gives a bijection $\HG$ between RPP ranked by
the size and nonnegative arrays of shape~$\la$ ranked by the hook
weight.

In the case of skew shapes, we study both the enumeration of skew RPP
and of skew SSYT via the map $\HG$. First,
we view RPP of skew shape $\la/\mu$  as a special case of RPP of
shape~$\la$ with zeros in $\mu$. We obtain the following
generalization of formula~\eqref{eq:RPP-prod}. This result is natural from enumerative point of view, but is
unusual in the literature (cf.~Section~\ref{sec:compare} and~$\S$\ref{ss:finrem-Lam}),
and is completely independent of Theorem~\ref{thm:skewSSYT}.

\begin{theorem} \label{thm:skewRPP}
We have:
\begin{equation} \label{eq:skew-RPP}
\sum_{\pi \in \RPP(\lambda/\mu)} q^{|\pi|} \, = \, \sum_{S \in
 \PD(\lambda/\mu)} \. \.\prod_{u\in S}\frac{q^{h(u)}}{1-q^{h(u)}}\,,
\end{equation}
where $\PD(\lambda/\mu)$ is the set of pleasant diagrams $($see {\rm Definition~\ref{def:agog} }$)$.
\end{theorem}

The theorem employs a new family of combinatorial objects called
{\em pleasant diagrams}.  These diagrams can be defined as
subsets of complements of excited diagrams (see Theorem~\ref{conj:charpleasant}),
and are technically useful.  This allows us to write
the RHS of~\eqref{eq:skew-RPP} completely in terms of excited diagrams
(see Corollary~\ref{cor:skewRPP}).  Note also that as corollary of
Theorem~\ref{thm:skewRPP}, we obtain a combinatorial proof of
NHLF (see~$\S$\ref{ss:pleasant-skewRPP-excited}).

\smallskip

Second, we look at the restriction of~$\HG$ to SSYT of skew shape $\la/\mu$. The major technical result of this part is Theorem~\ref{thm:bij},
which states that such a restriction gives a \textbf{bijection}
between SSYT of shape $\lambda/\mu$ and arrays of
nonnegative integers of shape $\lambda$ with zeroes in the excited diagram
and certain nonzero cells ({\em excited arrays}, see Definition~\ref{def:hookarray}).
% and their set is denoted by $\EA(\lambda/\mu)$.
In other words, we fully characterize the preimage of the SSYT of shape $\lambda/\mu$
under the map~$\HG$.  This and the properties of~$\HG$ allows us to obtain a number
of generalizations of Theorem~\ref{thm:skewSSYT} (see below).

The proof of Theorem~\ref{thm:bij} goes through several steps of interpretations using
careful analysis of longest decreasing subsequences in these arrays and a detailed study of structure
of the resulting tableaux under the RSK.  We built on top of the celebrated \emph{Greene's
theorem} and several of Gansner's results.

\subsection{Further extensions} \label{ss:intro-trace}
One of the most celebrated formula in enumerative combinatorics is
\emph{MacMahon's formula} for enumeration
of \emph{plane partitions}, which can be viewed as a limit case
of {\em Stanley's trace formula} (see~\cite{St71,StPP}):
$$
\sum_{\pi \in \PP} q^{|\pi|} \, = \, \prod_{n=1}^\infty \.
\frac{1}{(1- q^{n})^n}\,,  \qquad
\sum_{\pi \in \PP(m^\ell)} \. q^{|\pi|}\ts t^{\tr(\pi)} \, = \, \prod_{i=1}^m \. \prod_{j=1}^\ell \,
\frac{1}{1-t\ts q^{i+j-1}}\,.
$$
Here $\tr(\pi)$ refers to the {\em trace} of the plane partition.

These results were further generalized by Gansner~\cite{G} by using the properties
of the Hillman--Grassl correspondence combined with that of the RSK correspondence
(cf.~\cite{G2}).

\begin{theorem}[Gansner~\cite{G}] \label{thm:trace-RPP}
We have:
\begin{equation} \label{eq:traceeq}
\sum_{\pi \in \RPP(\lambda)} \. q^{|\pi|}\ts t^{\tr(\pi)} \, = \, \prod_{u\in
  \square^{\lambda}}\.  \frac{1}{1-tq^{h(u)}} \. \prod_{u\in
  [\lambda]\setminus \square^{\lambda}} \. \frac{1}{1-q^{h(u)}}\,,
\end{equation}
where $\square^{\lambda}$ is the {\em Durfee square} of the Young diagram of~$\lambda$.
\end{theorem}

For SSYT and RPP of skew shapes, our analysis of the
Hillman--Grassl correspondence gives the following simultaneous
generalizations of Gansner's theorem and our Theorems~\ref{thm:skewSSYT} and~\ref{thm:skewRPP}.

\begin{theorem} \label{thm:RPP-trace}
We have:
\begin{equation}
\label{eq:RPP-trace}
\sum_{\pi \in \RPP(\lambda/\mu)} \. q^{|\pi|}\ts t^{\tr(\pi)} \, = \, \sum_{S \in
  \PD(\lambda/\mu)} \, \prod_{u\in S \cap \square^\lambda} \.
  \frac{t\ts q^{h(u)}}{1-t\ts q^{h(u)}}\, \prod_{u\in S \setminus \square^\lambda} \. \frac{q^{h(u)}}{1-q^{h(u)}}\,.
\end{equation}
\end{theorem}

As with the~\eqref{eq:RPP-prod}, the RHS of~\eqref{eq:RPP-trace}
can be stated completely in terms of excited diagrams
(see Corollary~\ref{cor:trace-RPP-excited}).

\begin{theorem} \label{thm:SSYT-trace}
We have:
\begin{equation}
\label{eq:SSYT-trace}
\sum_{\pi \in \SSYT(\lambda/\mu)} \. q^{|\pi|}\ts t^{\tr(\pi)} \, = \, \sum_{S\in
  \ED(\lambda/\mu)} q^{a(S)}\ts t^{c(S)}\, \prod_{u \in \overline{S} \cap \square^{\lambda}} \.
\frac{1}{1-t\ts q^{h(u)}} \, \prod_{u\in \overline{S} \setminus
  \square^{\lambda}} \. \frac{1}{1-q^{h(u)}}\,,
\end{equation}
where $\overline{S}=[\lambda]\setminus S$,
$a(S)=\sum_{(i,j)\in [\lambda]\setminus S} (\lambda'_j-i)$ and
$c(S)=|\supp(A_S)\cap \square^{\lambda}|$ is the size of the support of
the excited array $A_S$ inside the Durfee square $\square^{\lambda}$ of $\lambda$.
\end{theorem}

Let us emphasize that the proof Theorem~\ref{thm:SSYT-trace} requires
both the algebraic proof of Theorem~\ref{thm:skewSSYT} and the analysis
of the Hillman--Grassl correspondence.

In Section~\ref{sec:bounded_parts} we also consider the enumeration of
skew SSYTs with bounded size of the entries. For straight shapes the
number is given by the hook-content formula. It is natural to expect
an extension of this result through the NHLF. Using the established
bijections and their properties we derive compact positive formulas
for $s_{\lambda/\mu}(1,q,\ldots,q^M)$ as sums over excited diagrams,
in the cases when the skew shape is a border strip. For general skew
shapes it is not clear how to extend the classical hook-content
formula, this is discussed in Section~\ref{ss:finrem-GNW}.

\smallskip

\subsection{Comparison with other formulas.}

In Section~\ref{sec:compare} we provide a comprehensive overview of
the other formulas for $f^{\lambda/\mu}$ that are either already
present in the literature or could be deduced. We show that the NHLF is not a restatement of any of them, and in particular demonstrate how it differs in the number of summands and the terms themselves.

The classical formulas are the Jacobi--Trudi identity, which has negative terms, and the expansion of $f^{\lambda/\mu}$ via the Littlewood--Richardson rule as a sum over $f^{\nu}$ for $\nu \vdash n$.  Another formula is the Okounkov--Olshanski identity summing particular products over SSYTs of shape $\mu$. While it looks similar to the NHLF, it has more terms and the products are not over hook-lengths.

We outline another approach to formulas for $f^{\lambda/\mu}$. We
observe that the original proof of Naruse of the NHLF in~\cite{Strobl}
comes from a particular specialization of the formal variables in the
evaluation of equivariant Schubert structure constants (generalized
Littlewood--Richardson coefficients) corresponding to Grassmannian
permutations. Ikeda--Naruse and Naruse respectively give a formula for their evaluation
in~\cite{IkNa09} via the excited diagrams on one-hand and an iteration
of a Chevalley formula on the other hand, which gives the
correspondence with skew standard Young tableaux. Our algebraic proof
of Theorem~\ref{thm:skewSSYT} follows this approach.

Now, there are other expressions for these equivariant Schubert structure constants,
which via the above specialization would give enumerative formulas for $f^{\lambda/\mu}$.
First, the Knutson-Tao puzzles \cite{KT} give an enumerative formula as a sum
over puzzles of a product of weights corresponding to them. It is also different from the sum over excited diagrams, as shown in examples. Yet
another rule for the evaluation of these specific structure constants
is given by Thomas and Yong in~\cite{TY}, as a sum over certain
edge-labeled skew SYTs of products of weights (corresponding to the
edge label's paths under jeu-de-taquin). An example in
Section~\ref{sec:compare} illustrates that the terms in the formula are different from the terms in the NHLF.

\smallskip

\subsection{Paper outline.}
The rest of the paper is organized as follows. We begin with notation,
basic definitions and background results (Section~\ref{sec:notation}).
The definition of excited diagrams is given in Section~\ref{sec:excited_diagrams},
together with the original formula of Naruse and corollaries of the $q$-analogue. It also contains
the enumerative properties of excited diagrams and the correspondence
with flagged tableaux.  In Section~\ref{sec:algproof}, we give an
algebraic proof of the main Theorem~\ref{thm:skewSSYT}.
Section~\ref{sec:HGSSYT} described the Hillman--Grassl correspondence,
with various properties and an equivalent formulation using the
RSK correspondence in Corollary~\ref{cor:HGvsRSK}.

Section~\ref{sec:HGRPP} defines pleasant diagrams and proves
Theorem~\ref{thm:skewRPP} using the Hillman--Grassl correspondence,
and as a corollary gives a purely combinatorial proof of NHLF
(Theorem~\ref{thm:IN}). Then, in Section~\ref{sec:HGSSYT}, we show that the
Hillman--Grassl map is a bijection between skew SSYT of shape
$\lambda/\mu$ and certain integer arrays whose support is in
the complement of an excited diagram. Section~\ref{sec:bounded_parts} derives the formulas for $s_{\lambda/\mu}(1,q,\ldots,q^M)$ in the cases of border strips. Section~\ref{sec:compare}
compares NHLF and other formulas for $f^{\la/\mu}$.
We conclude with final remarks and open problems in Section~\ref{sec:finrem}.

%Section~\ref{sec:enum_strips_SSYT} considers the special case
%when $\lambda/\mu$ is a thick strip shape,
%which give the connection with Euler and Catalan numbers.

%In Section~\ref{sec:enum_strips_RPP},
%we consider the pleasant diagrams of the thick strip shapes, establishing
%connection with Schr\"oder numbers. We also state conjectures
%on certain determinantal formulas.

%----------------------------------------------------------------
\bigskip\section{Notation and definitions} \label{sec:notation}
%----------------------------------------------------------------

\subsection{Young diagrams} \label{ss:not-yd}
Let $\lambda=(\lambda_1,\ldots,\lambda_r),
\mu=(\mu_1,\ldots,\mu_s)$ denote integer partitions of
length $\ell(\lambda)=r$ and $\ell(\mu)=s$. The {\em size} of the partition
is denoted by $|\lambda|$ and $\lambda'$
denotes the {\em conjugate partition} of $\lambda$. We use $[\lambda]$ to
denote the Young diagram of the partition $\lambda$. The \emph{hook length}
$h_{ij} = \la_i - i +\la_j' -j +1$ of a square $u=(i,j)\in [\la]$ is
the number of squares directly to the right and directly below~$u$
in~$[\la]$.  The {\em Durfee square} $\square^{\lambda}$ is the largest
square inside~$[\la]$; it is always of the form
$\{(i,j), 1\le i,j \le k\}$.

A {\em skew shape} is denoted by $\lambda/\mu$. For an integer $k$,
$1-\ell(\lambda)\leq k \leq \lambda_1-1$, let $\dd_k$ be the diagonal
$\{(i,j) \in \lambda/\mu \mid i-j = k\}$,
where $\mu_k=0$ if $k> \ell(\mu)$. For an integer $t$, $1\leq t
\leq \ell(\lambda)-1$ let $\dd_t(\mu)$ denote the diagonal $\dd_{\mu_t-t}$
where $\mu_t=0$ if $\ell(\mu)<t \leq \ell(\lambda)$.

Given the skew shape $\lambda/\mu$, let $P_{\lambda/\mu}$ be the poset
of cells $(i,j)$ of $[\lambda/\mu]$ partially ordered by component. This poset is {\em
  naturally labelled}, unless otherwise stated.

%Finally, let $\delta_n =(n-1,n-2,\ldots,2,1)$ denotes the staircase shape.

\subsection{Young tableaux} \label{ss:not-yt}
A {\em reverse plane partition} of skew shape
$\lambda/\mu$ is an array $\pi=(\pi_{ij})$ of nonnegative integers of shape $\lambda/\mu$
that is weakly increasing in rows and columns. We denote the set of
such plane partitions by $\RPP(\lambda/\mu)$.
 A {\em semistandard Young tableau} of shape $\lambda/\mu$ is a RPP
of shape $\lambda/\mu$ that is
strictly increasing in columns. We denote the set of such tableaux by
$\SSYT(\lambda/\mu)$. A {\em standard Young
  tableau} (SYT) of shape $\lambda/\mu$ is an array $T$ of shape
$\lambda/\mu$ with the numbers $1,\ldots,n$, where
$n=|\lambda/\mu|$, each 
appearing once, strictly increasing in rows and columns.  For example,
there are five SYT of shape $(32/1)$:
\[
\ytableausetup{smalltableaux}\ytableaushort{\none12,34}\ \quad
\ytableaushort{\none13,24}\ \quad  \ytableaushort{\none14,23}\ \quad
\ytableaushort{\none23,14}\ \quad \ytableaushort{\none24,13}
\]
 The {\em
  size} $|P|$ of an RPP $P$ or  $|T|$ of a tableau $T$ is the
sum of its entries. A {\em descent} of a  SYT $T$ is an index $i$
such that $i+1$ appears in a row below $i$. The {\em major index}
$\tmaj(T)$ is the sum $\sum i$ over all the descents of~$T$.

\subsection{Symmetric functions} \label{ss:not-sym}
Let $s_{\lambda/\mu}({\bf x})$ denote the  {\em skew Schur function}
of shape $\lambda/\mu$  in variables
${\bf x} = (x_0,x_1,x_2,\ldots)$. In particular,
\[
s_{\lambda/\mu}({\bf x}) \. = \. \sum_{T\in \SSYT(\lambda/\mu)} {\bf x}^T\., \ \  \qquad
s_{\lambda/\mu}(1,q,q^2,\ldots) \. = \. \sum_{T\in \SSYT(\lambda/\mu)} q^{|T|}
\.,
\]
where \ts ${\bf x}^T = x_0^{\#0s \text{ in } (T)}\ts x_1^{\#1s \text{ in } (T)}\ldots$ \ts.
The Jacobi--Trudi identity (see e.g.~\cite[\S 7.16]{EC2}) states that
\begin{equation} \label{eq:JTid}
s_{\lambda/\mu}({\bf x}) \, = \, \det\bigl[h_{\lambda_i-\mu_j -i+j}({\bf x})\bigr]_{i,j=1}^{n},
\end{equation}
where $h_k({\bf x}) = \sum_{i_1\leq i_2 \leq \cdots \leq i_k}
x_{i_1}x_{i_2}\cdots x_{i_k}$ is the $k$-th {\em complete symmetric
  function}. Recall also two specializations of $h_k({\bf x})$:
$$h_k(1^n) \. = \ts \binom{n+k-1}{k} \quad \text{and} \quad
h_k(1,q,q^2,\ldots) \, = \, \prod_{i=1}^k \.
\frac{1}{1-q^i}
$$
(see e.g.~\cite[Prop. 7.8.3]{EC2}), and the specialization of
$s_{\lambda/\mu}({\bf x})$ from Stanley's theory of {\em $(P,\omega)$-partitions} (see \cite[Thm.~3.15.7 and Prop.~7.19.11]{EC2}):
\begin{equation} \label{eq1:skewPpartitions}
s_{\lambda/\mu}(1,q,q^2,\ldots) \, = \, \frac{\sum_{T}
  q^{\tmaj(T)}}{(1-q)(1-q^2)\cdots (1-q^n)}\,,
\end{equation}
where the sum in the numerator of the RHS is over $T$ in $SYT(\lambda/\mu)$, $n=|\lambda/\mu|$ and
$\tmaj(T)$ is as
defined in Section~\ref{ss:not-yt}.

\subsection{Permutations}  We write permutations of $\{1,2,\ldots,n\}$
in {\em one-line notation}: $w=(w_1w_2\ldots w_n)$ where $w_i$ is the
image of $i$. A {\em descent} of $w$ is an index $i$ such that
$w_i>w_{i+1}$. The {\em major index} $\maj(w)$ is the sum $\sum i$ of
all the descents $i$ of $w$.

\subsection{Bijections} \label{ss:not-bij}
To avoid ambiguity, we use the word \emph{bijection} solely as
a way to say that map $\phi: X\to Y$ is one-to-one and onto.  We use the word
\emph{correspondence} to refer to an algorithm defining~$\phi$.
Thus, for example, the Hillman--Grassl correspondence $\Psi$
defines a bijection between certain sets of tableaux and arrays.

%----------------------------------------------------------------
\bigskip\section{Excited diagrams}\label{sec:excited_diagrams}
%----------------------------------------------------------------

\subsection{Definition and examples} \label{ss:excited-def}
Let $\lambda/\mu$ be a skew partition and $D$ be a subset of the Young
diagram of $\lambda$. A cell $u=(i,j) \in D$ is called {\em active} if
  $(i+1,j)$, $(i,j+1)$ and $(i+1,j+1)$ are all in
$[\lambda]\setminus D$.  Let $u$ be an
active cell of $D$, define $\alpha_u(D)$ to be the set obtained by
replacing $(i,j)$ in $D$ by $(i+1,j+1)$. We call this replacement an {\em excited move}. An {\em excited diagram} of
$\lambda/\mu$ is a subdiagram of $\lambda$ obtained from the Young
diagram of $\mu$
after a sequence of excited moves on active cells. Let
$\mathcal{E}(\lambda/\mu)$ be the set of
excited diagrams of~$\lambda/\mu$.

% For example,
% Figure~\ref{fig1}(b) shows the three excited
% diagrams of $(2^31/1^2)$.

%The following example illustrates Naruse's hook-length formula (NHLF) in a special casse.

\begin{example}\label{ex:excited-def}
There are three excited diagrams for the shape $(2^31/1^2)$, see
Figure~\ref{fig1}. The hook-lengths of the cells of these diagrams
are $\{5,4\}, \{5,1\}$ and $\{2,1\}$ respectively and these are the
excluded hook-lengths. The NHLF states in this case:
\[
f^{(2^31/1^2)} \, = \, 5!\left(\frac{1}{3\cdot 3\cdot 2\cdot 1 \cdot 1} + \frac{1}{4\cdot
    3\cdot 3\cdot 2 \cdot 1} + \frac{1}{5\cdot 4\cdot 3\cdot 3\cdot 1}\right) \, = \, 9\ts.
\]

\begin{figure}[hbt]
\begin{center}
\includegraphics{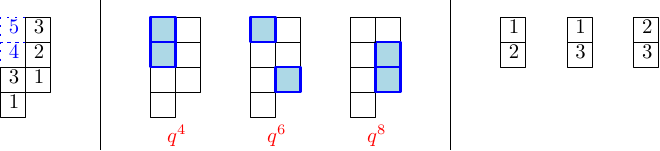}
\caption{
The hook-lengths of the skew shape $\lambda/\mu=(2^31/1^2)$, three
excited diagrams for $(2^31/1^2)$ and the corresponding flagged tableaux
in $\mathcal{F}(\mu,(3,3))$.}
\label{fig1}
\end{center}
\end{figure}

\nin
For the $q$-analogue, let
$$
a(D) \, := \, \sum_{(i,j) \in [\lambda]\setminus D} \. (\lambda'_j-i)
$$
be the sum of exponents of $q$ in the numerator of the RHS of~\eqref{eq:skewschur}.
We have $a(D_1) = 4$, $a(D_2)=6$ and $a(D_3)=8$,
where $D_1,D_2,D_3\in \ED(2^31/1^2)$ are the three excited diagrams
in the figure.

Now Theorem~\ref{thm:skewSSYT} gives
\begin{multline*}
s_{2^31/1^2}(1,q,q^2,\ldots) \, = \, \frac{q^4}{(1-q^3)^2(1-q^2)(1-q)^2} \. +\\
+\. \frac{q^6}{(1-q^4)(1-q^3)^2(1-q^2)(1-q)} \. + \. \frac{q^8}{(1-q^5)(1-q^4)(1-q^3)^2(1-q)}\..
\end{multline*}
Compare this with the expression \eqref{eq1:skewPpartitions} for the
same specialization
\[
s_{2^31/1^2}(1,q,q^2,\ldots) \, = \frac{q^6+2q^4+2q^3+q^5+2q^2+q}{(1-q)(1-q^2)(1-q^3)(1-q^4)(1-q^5)}.
\]
\end{example}

\smallskip

\begin{example} \label{ex:skewhook}
For the hook shape $(k,1^{d-1})$ we have that $f^{(k,1^{d-1})} =
\binom{k+d-2}{k-1}$. By symmetry, for the skew shape $\lambda/\mu$ with $\lambda = (k^d)$ and
$\mu=((k-1)^{d-1})$, we also have $f^{\lambda/\mu} = f^{(k,1^{d-1})}$. The
complements of excited diagrams of this shape are in bijection with
lattice paths $\gamma$ from the cell labeled $(d,1)$ to the cell $(1,k)$. Thus $|\ED(\lambda/\mu)|
= \binom{k+d-2}{k-1}$. Here is an example with $k=d=3$:
\begin{center}
\includegraphics{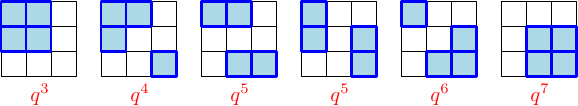}.
\end{center}
Moreover, since $h(i,j) = i+j-1$ for $(i,j)\in
[\lambda]$ then the NHLF, switching the LHS and RHS, states  in this case:
\begin{equation} \label{eq:specialcaseeq1}
\sum_{\gamma:\, (d,1) \to (1,k)} \prod_{(i,j) \in \gamma}
\frac{1}{i+j-1} = \binom{k+d-2}{k-1}\ts ,
\end{equation}
where $\gamma:\, (d,1) \to (1,k)$ means that $\gamma$ is a  NE lattice paths between the given cells.
Next we apply our first $q$-analogue to this shape. First, we have
that $\ts s_{k1^{d-1}} = s_{\lambda/\mu}$
\cite[Prop.~7.10.4]{EC2}. Next, by \cite[Cor. 7.21.3]{EC2} the principal specialization
of the Schur function $\ts s_{k1^{d-1}}$ equals
\[
s_{k1^{d-1}}(1,q,q^2,\ldots) \, = \, q^{\binom{d}{2}} \ts
\prod_{i=1}^{k+d-1} \. \frac{1}{1-q^i} \begin{bmatrix} k+d-2\\ k-1 \end{bmatrix}_q,
\]
where $\begin{bmatrix} n\\ k \end{bmatrix}_q$ is a $q$-binomial
coefficient. Second if the excited diagram $D \in \ED(\lambda/\mu)$
corresponds to path~$\gamma$ then one can show that $a(D) =
\binom{n}{2}+\area(\gamma)$ where $\area(\gamma)$ is the number of cells in the $d\times k$
rectangle South East of the path $\gamma$.
Putting this all together then Theorem~\ref{thm:skewSSYT} for shape $\lambda/\mu$ gives
\begin{equation} \label{eq:specialcaseeq2}
\left(\prod_{i=1}^{k+d-1}
  (1-q^i)\right) \sum_{\gamma: (d,1)  \to (1,k)} q^{\area(\gamma)}\prod_{(i,j)
  \in \ga} \frac{1}{1-q^{i+j-1}} \,=\, \begin{bmatrix} k+d-2\\ k-1 \end{bmatrix}_q\,.
\end{equation}
In \cite{MPP3}, we show that \eqref{eq:specialcaseeq1} and \eqref{eq:specialcaseeq2} are special cases
of the Racah and $q$-Racah formulas in~\cite{BGR}.
% add connection to Borodin identity cite paper
\end{example}

\medskip

\subsection{NHLF from its $q$-analogue.}
Next, before proving Theorem~\ref{thm:skewSSYT}, we first show that it is a {\em $q$-analogue} of \eqref{eq:Naruse}.
This argument is standard; we outline it for reader's convenience.

\begin{proposition} \label{cor:getnaruse}
Theorem~\ref{thm:skewSSYT} implies the NHLF \eqref{eq:Naruse}.
\end{proposition}

\begin{proof}
Multiplying
\eqref{eq1:skewPpartitions} by \ts $(1-q) \cdots (1-q^n)$ \ts and using
Theorem~\ref{thm:skewSSYT}, gives
\begin{equation} \label{eq2:skewPpartitions}
\sum_{T \in \SYT(\lambda/\mu)}
  q^{\tmaj(T)}  \,= \, \prod_{i=1}^n (1-q^i) \. \sum_{D\in \ED(\lambda/\mu)} \.
\prod_{(i,j)\in [\lambda]\setminus D} \frac{q^{\lambda'_j-i}}{1-q^{h(i,j)}}\..
\end{equation}
Since all excited diagrams $D\in \ED(\lambda/\mu)$ have size $|\mu|$
then by taking the limit $q\to 1$ in \eqref{eq2:skewPpartitions}, we obtain
the NHLF~\eqref{eq:Naruse}.
\end{proof}

Theorem~\ref{thm:skewRPP} is a different $q$-analogue of NHLF, as explained in Section~\ref{sec:HGRPP}.

%\begin{figure}
%\begin{center}
%\includegraphics{excited_ex1b}
%\end{center}
%\caption{The three excited diagrams for $\lambda/\mu=(2^31/1^2)$ and
 %their associated statistics.}
%\label{fig1b}
%\end{figure}

\medskip \subsection{Flagged tableaux}\label{ss:excited-flagged}
Excited diagrams of $\lambda/\mu$ are also
equivalent to certain {\em flagged tableaux} of shape $\mu$
(see Proposition~\ref{prop:excited2flaggedtab} and \cite[\S 6]{VK}) and thus the number of
excited diagrams is given by a determinant (see Corollary~\ref{cor:GV}), a
polynomial in the parts of $\lambda$ and~$\mu$.

In this section we relate excited diagrams with {\em flagged
  tableaux}. The relation is based on a map by Kreiman~\cite[\S 6]{VK}
  (see also~\cite[\S 5]{KMY}).

We start by stating an important property of excited diagrams that
follows immediately from their construction. Given a set $D\subseteq
[\lambda]$ we say that $(i,j), (i+m,j+m) \in D\cap \dd_k$ for $m>0$ are {\em consecutive}
if there is no other element in $D$ on diagonal $\dd_k$ between them.

\begin{definition}[Interlacing property] \label{def:interlacing}
Let $D \subset [\lambda]$. If $(i,j)$ and $(i+m,j+m)$ are two
consecutive elements in $D \cap \dd_k$ then $D$ contains an element in
each diagonal $\dd_{k-1}$ and $\dd_{k+1}$ between columns $j$ and $j+m$.
\. Note that the excited diagrams in $\ED(\lambda/\mu)$ satisfy this property
by construction.
\end{definition}

Fix a sequence $\f=(\ssf_1,\ssf_2,\ldots,\ssf_{\ell(\mu)})$
of weakly increasing nonnegative integers.  Define $\mathcal{F}(\mu,\f)$ to be the set of
$T \in \SSYT(\mu)$, such that all entries $1\leq T_{ij}\leq \ssf_i$.
Such tableaux are called {\em flagged SSYT} and they were first
studied by Lascoux and Sch\"utzenberger~\cite{LS} and Wachs~\cite{W85}.  By the
Lindstr\"om--Gessel--Viennot lemma on non-intersecting paths (see e.g.~\cite[Thm.~7.16.1]{EC2}),
the size of $\mathcal{F}(\mu,\f)$ is given by a determinant:

\begin{proposition}[Gessel--Viennot \cite{GV}, Wachs \cite{W85}] \label{prop:flagged-det}
In the notation above, we have:
\[
|\mathcal{F}(\mu,\f)| \, = \,
\det\left[h_{\mu_i-i+j}(1^{\ssf_i})\right]_{i,j=1}^{\ell(\mu)}
\, = \, \det\left[
  \binom{\ssf_i+\mu_i-i+j-1}{\mu_i-i+j}\right]_{i,j=1}^{\ell(\mu)},
\]
where $h_k(x_1,x_2,\ldots)$ denotes the complete symmetric function.
\end{proposition}

Given a skew shape $\lambda/\mu$, each row $i$ of $\mu$ is between
the rows $k_{i-1} < i \leq k_{i}$ of two corners of $\mu$. When a corner of $\mu$ is in row $k$, let $\ssf'_k$ be the last row of diagonal $\dd_{\mu_k-k}$ in $\lambda$. Lastly, let
$\f^{(\lambda/\mu)}$ be the vector\footnote{In \cite{KMY}, the vector $\f^{\lambda/\mu}$ is called a {\em flagging}.} $(\ssf_1,\ssf_2,\ldots,\ssf_{\ell(\mu)})$, $\ssf_i=\ssf'_{k_i}$ where $k_i$ is the row of the corner of $\mu$ at
or immediately after row~$i$ (see Figure~\ref{fig:flagged_tab}). Let
$\mathcal{F}(\lambda/\mu):=\mathcal{F}(\mu,\f^{(\lambda/\mu)})$.

Let $T_{\mu}$ be the tableaux of shape $\mu$ with entries $i$ in row
$i$. Note that $T_{\mu} \in \mathcal{F}(\lambda/\mu)$. We define an
analogue of an excited move for flagged tableaux. A cell $(x,y)$ of
$T$ in $\mathcal{F}(\lambda/\mu)$ is {\em active} if increasing
$T_{x,y}$ by $1$ results in a flag SSYT tableau $T'$ in
$\mathcal{F}(\lambda/\mu)$. We call this map $T\mapsto T'$ a {\em flagged
  move} and denote by $\alpha'_{x,y}(T)=T'$.

Next we show
that excited diagrams in $\ED(\lambda/\mu)$ are in bijection
with flagged tableaux in $\mathcal{F}(\lambda/\mu)$.

Given $D \in \ED(\lambda/\mu)$, we define $\varphi(D):=T$ as follows:
Each cell  $(x,y)$ of $[\mu]$
corresponds to a cell $(i_x,j_y)$ of $D$. We let $T$ be the tableau of
shape $\mu$ with $T_{x,y}=i_x$.  An example is given in
Figure~\ref{fig:flagged_tab}.

\begin{figure}[hbt]
\begin{center}
\includegraphics[scale=1.1]{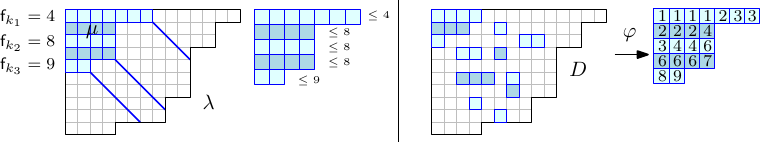}
\caption{Given a skew shape $\lambda/\mu$, for each corner $k$ of $\mu$ we
record the last row $\ssf_k$ of $\lambda$ from diagonal $\dd_{\mu_k-k}$. These
row numbers give the bound for the flagged tableaux of shape $\mu$ in $\mathcal{F}(\mu, \f^{(\lambda/\mu)})$.}
\label{fig:flagged_tab}
\end{center}
\end{figure}

\begin{proposition}
\label{prop:excited2flaggedtab}
We have $|\ED(\lambda/\mu)|=|\mathcal{F}(\lambda/\mu)|$ and the
map  $\varphi$ is a bijection between these two sets.
\end{proposition}

\begin{proof}
We need to prove that $\varphi$ is a
well defined map from $\ED(\lambda/\mu)$ to $\mathcal{F}(\mu,
\f^{(\lambda/\mu)})$.
First, let us show that $T=\varphi(D)$ is a SSYT by induction on the number of excited
moves of $D$. First, note that $\varphi([\mu])=T_{\mu}$ which is SSYT. Next, assume that for $D\in
\ED(\lambda/\mu)$, $T=\varphi(D)$ is a SSYT and
$D'=\alpha_{(i_x,j_y)}(D)$ for some active cell $(i_x,j_y)$ of $D$
corresponding to $(x,y)$ in $[\mu]$. Then $T'=\varphi(D')$ is
obtained from $T$ by adding $1$ to entry $T_{x,y}=i_x$ and leaving the
rest of entries unchanged. When $(x+1,y)
\in [\mu]$, since $(i_x+1,j_y)$ is not in $D$ then the cell of the
diagram corresponding to $(x+1,y)$ is in a row $>i_x+1$, therefore
$T'_{x,y}=i_x+1<T_{x+1,y}=T'_{x+1,y}$. Similarly, if $(x,y+1) \in [\mu]$, since
$(i_x,j_x+1)$ is not in $D$ then the cell of the diagram corresponding to
$(x,y+1)$ is in a row $>i_x$, therefore $T'_{x,y}=i_x+1\leq
T_{x,y+1}=T'_{x,y+1}$. Thus, $T' \in \SSYT(\la/\mu)$.

Next, let us show that $T$ is a flagged tableau in $\mathcal{F}(\mu,
\f^{(\lambda/\mu)})$. Given an excited diagram $D$, if cell
$(i_x,j_y)$ of $D$ is the cell corresponding to $(x,y)$ in $[\mu]$ then the row $i_x$ is
at most $\ssf_{k_x}$: the last row of diagonal $\dd_{\mu_{k_x}-k_x}$ where $k_x$ is the row of
the corner of $\mu$ on or immediately after row $x$. Note that the numbers $\ssf_{k_x}$ are weakly increasing as $x$ grows, since the diagonals move to the left and so intersect $\la$ at lower rows. Thus $T_{x,y}\leq
\ssf_{k_x}$, which proves the claim.

Finally, we prove that  $\varphi$ is a bijection by building its inverse.
Given $T \in \mathcal{F}(\mu, \f^{(\lambda/\mu)})$, let $D=\vartheta(T)$ be the set
$D=\{ (T_{x,y},y+T_{x,y}) \mid (x,y) \in [\mu]\}$.
Let us show  $\vartheta$ is a well defined map from $\mathcal{F}(\mu,
\f^{(\lambda/\mu)})$ to $\ED(\lambda/\mu)$. By definition of the flags $\f^{(\lambda/\mu)}$,
observe that $D$ is a subset of $[\lambda]$. We prove that $D$ is in $\ED(\lambda/\mu)$
by induction on the number of flagged moves $\alpha'_{x,y}(\cdot)$.  First, observe that
$\vartheta(T_{\mu})=[\mu]$ which is in $\ED(\lambda/\mu)$. Assume that
for $T \in \mathcal{F}(\lambda/\mu)$, $D=\vartheta(T)$ is in
$\ED(\lambda/\mu)$ and $T'=\alpha'_{x,y}(T)$ for some active cell
$(x,y)$ of $T$. Note that replacing $T_{x,y}$ by $T_{x,y}+1$
results in a flagged tableaux $T'$ in $\mathcal{F}(\lambda/\mu)$ is
equivalent to $(i_x,i_y)$ being an active cell of~$D$. Since
$\vartheta(T')=\alpha_{i_x,i_y}(D)$ and the latter is an excited
diagram, the result follows. By construction, we conclude that
$\vartheta = \varphi^{-1}$, as desired.  \end{proof}

By Proposition~\ref{prop:flagged-det}, we immediately have the
following corollary.

 \begin{corollary} \label{cor:GV}
\[
|\ED(\lambda/\mu)| = \det\left[ \binom{\ssf^{(\lambda/\mu)}_i+\mu_i-i+j-1}{\ssf^{(\lambda/\mu)}_i-1}\right]_{i,j=1}^{\ell(\mu)}.
\]
\end{corollary}

Let $\mathcal{K}(\lambda/\mu)$ be  the set of $T \in \SSYT(\mu)$ such that all entries
$t=T_{i,j}$ satisfy the inequalities
$t \leq \ell(\lambda)$ and $T_{i,j} + c(i,j) \leq \lambda_t$.

\begin{proposition}[Kreiman~\cite{VK}] \label{prop:excited2kreiman}
We have $|\ED(\lambda/\mu)|=|\mathcal{K}(\lambda/\mu)|$ and the map
$\varphi$ is a bijection between these two sets.
\end{proposition}

Since the correspondences $\varphi$ from
Propositions~\ref{prop:excited2flaggedtab} and
\ref{prop:excited2kreiman} are the same then both sets of tableaux
are equal.

\begin{corollary} \label{cor:excited-kreiman}
We have $\mathcal{F}(\lambda/\mu)=\mathcal{K}(\lambda/\mu)$.
\end{corollary}

\begin{remark}{\rm
To clarify the unusual situation in this section, here we have
three equinumerous sets $\mathcal{K}(\lambda/\mu)$, $\mathcal{F}(\lambda/\mu)$ and
$\mathcal{E}(\lambda/\mu)$, all of which were previously defined in the literature.
The first two are in fact \emph{the same} sets, but defined somewhat differently;
essentially, the set of inequalities
in the definition of $\mathcal{K}(\lambda/\mu)$ has redundancies.
Since our goal is to prove Corollary~\ref{cor:GV}, we find
it easier and more instructive to use Kreiman's map~$\vp$ with a new
analysis (see below), to prove directly that
$|\mathcal{E}(\lambda/\mu)| = |\mathcal{F}(\lambda/\mu)|$.
An alternative approach would be to prove the equality
of sets $\mathcal{F}(\lambda/\mu)=\mathcal{K}(\lambda/\mu)$ first
(Corollary~\ref{cor:excited-kreiman}), which reduces the problem
to Kreiman's result (Proposition~\ref{prop:excited2kreiman}).
}\end{remark}

%----------------------------------------------------------------
\bigskip\section{Algebraic proof of Theorem~\ref{thm:skewSSYT}} \label{sec:algproof}
%----------------------------------------------------------------

% In this section we prove algebraically Theorem~\ref{thm:skewSSYT}.

\subsection{Preliminary results}
A skew shape
$\lambda/\mu$ with $\mu \subseteq \lambda \subseteq d \times (n-d)$ is
in correspondence with a pair of {\em Grassmannian permutations} $w\preceq v$
of $n$ both with descent at position $d$ and where $\preceq$ is the
strong Bruhat order. Recall that a permutation
$v=v_1v_2\cdots v_n$ is Grassmannian if it has a unique descent.
The permutation $v$ is obtained from
the diagram $\lambda$ by writing the numbers $1,\ldots,n$ along the
unit segments of the boundary of $\lambda$ starting at the bottom left
corner and ending at the top right of the enclosing $d \times (n-d)$
rectangle. The permutation $v$ is obtained by first reading the $d$
numbers on the vertical segments and then the $(n-d)$ numbers on the horizontal
segments. The permutation $w$ is obtained by the same procedure on
partition $\mu$ (see Figure~\ref{fig:skew_grass}).

\begin{figure}[hbt]
\begin{center}
\includegraphics[scale=0.8]{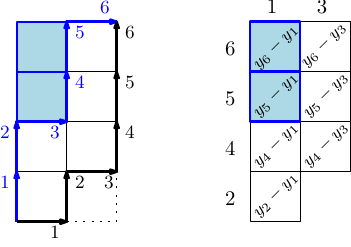}
\caption{The skew shape $2221/11$ corresponds to the Grassmannian
  permutations $v=245613$ and $w=124536$.}
\label{fig:skew_grass}
\end{center}
\end{figure}

Note that \. $\bigl(v(1),\ldots,v(n) \bigr) \ts = \ts
\bigl(\lambda_d+1, \lambda_{d-1} +2, \ldots \lambda_{1}+d, j_1,\ldots,j_{n-d}\bigr)$ \. and
\begin{align}
v(d+1-i)  \. = \. \lambda_i+d+1-i \tag{$\ast$}, \label{eq:vequal}
\end{align}
 where $\{j_1,\ldots,j_{n-d}\} = [n]\setminus \{\lambda_d+1, \lambda_{d-1} +2, \ldots, \lambda_{1}+d\}$ arranged in increasing order.  The numbers written up to the vertical segment on row $i$
are  $1,\ldots,\lambda_i+d-i$, of which $d-i$ are on the first
vertical segments, and the other $\lambda_i$ are on the first horizontal segments.
This gives
\begin{equation} \label{eq:setequal}
\bigl\{ v(1),\ldots,v(d-i),v(d+1),v(d+2),\ldots,v(d+\lambda_i)\bigr\}
\, = \, \bigl\{1,\ldots,\lambda_i+d-i\bigr\}\..\tag{$\ast\ast$}
\end{equation}

Let $[X_w]$ be the equivariant Schubert class corresponding to a permutation $w$
and let $[X_w]|_v$ be the multivariate polynomial with variables $y_1,\ldots,y_n$  corresponding to the image of the
class under a certain homomorphism $\iota_v$. We use a result from Ikeda and~Naruse~\cite{IkNa09} and Kreiman
\cite{VK} that follows from a formula by Billey for $[X_w]|_v$
when $v,w$ are Grassmannian permutations.

\begin{theorem}[Ikeda--Naruse \cite{IkNa09}, Kreiman \cite{VK};
  Billey \cite{Bil}]\label{thm:ikna1}
Let $w \preceq v$ be Grassmannian permutations whose unique descent is at
position $d$ with corresponding partitions $\mu \subseteq
\lambda\subseteq d\times (n-d)$. Then
$$
[X_w]\ts\bigl|_v \,\. = \.
\sum_{D \in \ED(\lambda/\mu)} \prod_{(i,j) \in D} \bigl(y_{v(d+j)} - y_{v(d-i+1)} \bigr).
$$
\end{theorem}

\begin{remark} \label{rem:BilleyF}
For general permutations $w\preceq v$ the polynomial
$[X_w]\ts\bigl|_v$ is a  {\em Kostant polynomial} $\sigma_w(v)$, see
\cite{KK,Bil,Tym}. Billey's
formula \cite[Appendix D.3]{AJS}\,\,\cite[Eq. (4.5)]{Bil}  expresses the latter as certain sums over reduced
subwords of $w$ from a fixed reduced word of $v$. Since in our context
 $w$ and $v$ are Grassmannian, the reduced subwords are
related only by commutations and no braid relations
(cf.~\cite{Ste}). This property allows the
authors in \cite[Thm. 1]{IkNa09} to find a bijection between the reduced subwords and excited
diagrams. The author in \cite[Prop. 2.2]{VK} uses the different method of {\em Gr\"{o}bner
  degenerations} to prove the result.
%This is
%a polynomial identity in~$\varepsilon_i$. As explained in~\cite{IkNa09,Strobl},
%they can be replaced by independent variables~$y_i$. So we have an honest polynomial identity:
\end{remark}

The {\em factorial Schur functions} (see e.g.~\cite{MS}) are defined as
$$s_{\mu}^{(d)}({\bf x}\ts |\ts {\bf a})\, := \, \frac{ \det \bigl[ (x_j-a_1)\cdots(x_j-a_{\mu_i+d-i}) \bigr]_{i,j=1}^d }{\prod_{1\leq i< j \leq d} \, (x_i - x_j)},$$
where $x=(x_1,x_2,\ldots,x_d)$ and $a=(a_1,a_2,\ldots)$ is a sequence
of parameters.

\begin{theorem}[Theorem 2 in \cite{IkNa09}, attributed to Knutson-Tao \cite{KT},
  Lakshmibai--Raghavan--Sankaran] \label{thm:ikna2}
$$[X_w]\ts \bigl|_v \.=
\.(-1)^{\ell(w)}s_{\mu}^{(d)}\bigl(y_{v(1)},\ldots,y_{v(d)} \ts |\ts y_1,\ldots,y_{n-1}\bigr).$$
\end{theorem}

\begin{corollary} \label{cor:IkNa}
Let $w\preceq v$ be Grassmannian permutations whose unique descent is at
position $d$ with corresponding partitions $\mu \subseteq \lambda
\subseteq d \times (n-d)$. Then
\begin{equation}
\label{factorial_excited2}
s_{\mu}^{(d)}\bigl(y_{v(1)},\ldots,y_{v(d)} \ts \bigl| \ts y_1,\ldots,y_{n-1}\bigr)
\. = \. \sum_{D \in \ED(\lambda/\mu)} \. \prod_{(i,j) \in D}
\bigl(y_{v(d-i+1)}-y_{v(d+j)}\bigr).
\end{equation}
\end{corollary}

\begin{proof}
Combining Theorem~\ref{thm:ikna1} and Theorem~\ref{thm:ikna2} we get
\begin{align}\label{factorial_excited}
(-1)^{\ell(w)}s_{\mu}^{(d)}(y_{v(1)},\ldots,y_{v(d)}\ts \bigl|\ts y_1,\ldots,y_{n-1})
\. = \.
\sum_{D \in \ED(\lambda/\mu)} \. \prod_{(i,j) \in D} \bigl(y_{v(d+j)} - y_{v(d-i+1)}\bigr).
\end{align}
Note that $\ell(w) = |\mu|$ and $\ell(v) = |\lambda|$, so we can
remove the $(-1)^{\ell(w)}$ on the left of~\eqref{factorial_excited}
by negating all linear terms on the right and get the desired result.
\end{proof}

\medskip \subsection{Proof of Theorem~\ref{thm:skewSSYT}}
First we use Corollary~\ref{cor:IkNa} to get an identity of rational
functions in $y=(y_1,y_2,\ldots,y_n)$ (Lemma~\ref{lemma:thm1}). Then
we evaluate this identity at  $y_p=q^{p-1}$ and use some identities of
symmetric functions to prove the theorem. Let
\[
H_{i,r}({\bf y}) \, := \, \begin{cases} \displaystyle \prod_{p=\mu_r+d+1-r}^{\lambda_i+d-i} \bigl(y_{\lambda_i+d+1-i}
    -y_p\bigr)^{-1}  &\text{ if $\mu_r+d-r \leq  \lambda_i+d-i$}\ts, \\
 {\hskip4.4cm 0}  & \, \text{otherwise\ts.} \end{cases}\ts.
\]
\begin{lemma} \label{lemma:thm1}
 \begin{equation} \label{thm1:lemm1}
\det\bigl[H_{i,j}({\bf y})\bigr]_{i,r=1}^d  \.= \. \sum_{D\in \ED(\lambda/\mu)} \prod_{(i,j) \in [\lambda]\setminus D}
\frac{1}{y_{v(d+1-i)}-y_{v(d+j)}}.
\end{equation}
\end{lemma}

\begin{proof}
Start with \eqref{factorial_excited2} and divide both sides by
\begin{equation} \label{eq:denom}
\prod_{(i,j) \in [\lambda] } \bigl(y_{v(d+1-i)} - y_{v(d+j)}\bigl) \,\, = \,
\prod_{i=1}^d \. \prod_{j=1}^{\lambda_i} \. \bigl(y_{v(d+1-i)} - y_{v(d+j)} \bigr)\.,
\end{equation}
to obtain
\begin{equation} \label{thm:eq1}
 \frac{ s_{\mu}^{(d)}(y_{v(1)},\ldots,y_{v(d)}|y_1,\ldots,y_{n-1}) }{
   \prod_{(i,j) \in [\lambda] } (y_{v(d+1-i)} - y_{v(d+j)}) } \, = \,
   \sum_{D \in \ED(\lambda/\mu)} \prod_{(i,j) \in [\lambda] \setminus D}
 \frac{1}{ y_{v(d+1-i)} - y_{v(d+j)} }.
\end{equation}

Denote the LHS of \eqref{thm:eq1} by $S_{\lambda,\mu}({\bf y})$. By the determinantal formula for
factorial Schur functions and by \eqref{eq:denom} we have
\begin{align*}
S_{\lambda,\mu}({\bf y})& \, = \, \frac{ \det \left[
   \prod_{p=1}^{\mu_r+d-r} (y_{v(d+1-i)} - y_p) \right]_{i,r=1}^d }
{ \prod_{i=1}^d  \prod_{k=i+1}^d (y_{v(d+1-i)} -y_{v(d+1-k)} ) }\cdot \frac{1}{\prod_{i=1}^d \prod_{j=1}^{\lambda_i} (y_{v(d+1-i)} - y_{v(d+j)} ) }\\
&= \, \det \left[ \frac{\prod_{p=1}^{\mu_r+d-r} (y_{v(d+1-i)} - y_p)
  }{\prod_{k=i+1}^d (y_{v(d+1-i)} -y_{v(d+1-k)} )
    \prod_{j=1}^{\lambda_i} (y_{v(d+1-i)} - y_{v(d+j)} )
  }\right]_{i,r=1}^d\..
\end{align*}
Using \eqref{eq:setequal} in the denominator of the matrix entry, we obtain:
\begin{align}
 S_{\lambda,\mu}({\bf y})& \. = \. \det \left[ \prod_{p=1}^{\mu_r+d-r}
     (y_{v(d+1-i)} - y_p)
     \prod_{p=1}^{\lambda_i+d-i} (y_{v(d+1-i)} - y_{p})^{-1}
   \right]_{i,r=1}^d. \label{eq:deteq2}
\end{align}
By~\eqref{eq:vequal}, we have $v(d+1-i) = \lambda_i + d+1-i$.  Therefore,
the matrix entry on the RHS of~\eqref{eq:deteq2} simplifies
to~$H_{i,r}({\bf y})$.
\begin{equation} \label{eq:deteq2b}
 S_{\lambda,\mu}({\bf y}) \. = \. \det[H_{i,r}({\bf y})]_{i,r=1}^d\..
\end{equation}
Combining \eqref{eq:deteq2b}   with \eqref{thm:eq1}  we obtain \eqref{thm1:lemm1}
as desired.
\end{proof}

\smallskip

Next, we evaluate $y_p = q^{p-1}$ for $p=1,\ldots,n$ in
\eqref{thm1:lemm1}. Since
\begin{equation}
(y_{v(d+1-i)} - y_{v(d+j)}) \. \Big|_{y_p=q^p}  \. = \. q^{\lambda_i+d+1-i} -q^{d-\lambda_j'+j} \. = \.
-q^{d-\lambda_j'+j} (1-q^{h(i,j)})\.,
\end{equation}
we obtain
\begin{equation} \label{eq:deteq3}
\det\bigl[H_{i,r}(1,q,q^2,\ldots,q^{n-1})\bigr]_{i,r=1}^d \, = \, (-1)^{|\lambda|-|\mu|}\sum_{D\in
  \ED(\lambda/\mu)} \. \prod_{(i,j)\in [\lambda]\setminus D} \. \frac{q^{-d+\lambda'_j-j}}{1-q^{h(i,j)}}\..
\end{equation}
We now simplify the matrix entry
$H_{i,r}(1,q,q^2,\ldots,q^{n-1})$. For $\nu=(\nu_1,\ldots,\nu_d)$, let
\[
g(\nu) \. := \. \sum_{i=1}^d \binom{\nu_i+d+1-i}{2}\..
\]
We then have:

\begin{proposition} \label{prop:thm1}
\[
H_{i,r}(1,q,q^2,\ldots,q^{n-1}) \, = \, q^{-\binom{\la_i+d+1-i}{2}+\binom{\mu_r+d+1-r}{2}} \ts h_{\lambda_i -i - \mu_r+r}(1,q,q^2,\ldots)\.,
\]
where $h_k({\bf x})$ denotes the $k$-th complete symmetric function.
\end{proposition}

\begin{proof}  We have:
\begin{align*}
H_{i,r}(1,q,q^2,\ldots,q^{n-1}) &\, = \,\prod_{p=\mu_r+d+1-r}^{\lambda_i+d-i} \.
\frac{1}{q^{\lambda_i+d+1-i} - q^p} \\
&\, = \, (-1)^{\lambda_i-i -\mu_r+r} \. q^{-\binom{\la_i+d+1-i}{2}+\binom{\mu_r+d+1-r}{2}} \.
 \prod_{p=1}^{\lambda_i-i - \mu_r+r}\. \frac{1}{1-q^p}\\
&\, = \, (-1)^{\lambda_i-i -\mu_r+r} \. q^{-\binom{\la_i+d+1-i}{2}+\binom{\mu_r+d+1-r}{2}}\. h_{\lambda_i -i - \mu_r+r}(1,q,q^2,\ldots)\ts,
\end{align*}
where the last identity follows by the principal specialization of the
complete symmetric function.
\end{proof}

Using Proposition~\ref{prop:thm1}, the LHS of \eqref{eq:deteq3}  becomes
\begin{equation}
\aligned
\det\bigl[ H_{i,r}(1,q,\ldots,q^{n-1})\bigr]_{i,r=1}^d & \, = \, (-1)^{|\lambda|-|\mu|} q^{-g(\lambda) + g(\mu)}
\det\left[ h_{\lambda_i -i - \mu_r+r}(1,q,q^2,\ldots) \right]_{i,r=1}^d\\
&\, = \, (-1)^{|\lambda|-|\mu|} q^{-g(\lambda) + g(\mu)}
s_{\lambda/\mu}(1,q,q^2,\ldots)\., \label{eq:JTdet}
\endaligned
\end{equation}
where the last equality follows by the Jacobi--Trudi identity for skew
Schur functions \eqref{eq:JTid}.  From here, rearranging powers of~$q$ and cancelling
signs, equation~\eqref{eq:deteq3} becomes
\begin{equation} \label{eq:deteq4}
s_{\lambda/\mu}(1,q,q^2,\ldots) \, = \, q^{g(\lambda)-g(\mu) }\sum_{D \in \ED(\lambda/\mu)} \prod_{(i,j) \in [\lambda] \setminus D} \frac{q^{-d+\lambda_j'-j} }{1-q^{h(i,j)} }\,.
\end{equation}
It remains to match the powers of $q$ in~\eqref{eq:deteq4} and~\eqref{eq:skewschur}.

\begin{proposition} \label{prop:sameq}
For an excited diagram $D\in \ED(\lambda/\mu)$ we
  have:
\[
g(\lambda)-g(\mu) + \sum_{(i,j) \in [\lambda]\setminus D}
(-d+\lambda'_j-j) \ = \, \sum_{(i,j)\in [\lambda]\setminus D} (\lambda'_j-i)\..
\]
\end{proposition}

\begin{proof}
Note that \ts
 $g(\lambda)= \sum_{i=1}^d \sum_{p=1}^{\la_i+d-i} p = \sum_{i=1}^{d} i(d-i) + d|\lambda| + \sum_{(i,j) \in [\lambda] } c(i,j) $, \ts where $c(i,j)=j-i$. Therefore,
%\begin{align}\label{q_power}
$$
g(\lambda)-g(\mu) - \sum_{(i,j) \in [\lambda]\setminus D} d \ = \, g(\lambda)- g(\mu) -d(|\lambda|-|D|)
\, = \, \sum_{(i,j) \in [\lambda]} c(i,j) -\sum_{(i,j) \in [\mu]} c(i,j)\..
$$
%\end{align}
Finally, notice that the cells of any excited diagram $D$ have the
same multiset of content values, since every excited move is along a diagonal and the content of the moved cell $j-i$ remains constant. Thus the power of $q$ for each term becomes
$$
\sum_{(i,j) \in [\lambda]\setminus [\mu] } c(i,j) + \sum_{(i,j) \in [\lambda]\setminus D} (\lambda'_j-j) = \sum_{(i,j) \in [\lambda] \setminus D} \left(c(i,j) +\lambda'_j-j\right) \. = \.
\sum_{(i,j) \in [\lambda] \setminus D} \lambda'_j-i\.,
$$
as desired.
\end{proof}

Using Proposition~\ref{prop:sameq} on the RHS of \eqref{eq:deteq4}
yields \eqref{eq:skewschur} finishing the proof of Theorem~\ref{thm:skewSSYT}.

%----------------------------------------------------------------
\bigskip\section{The Hillman--Grassl  and the RSK correspondences} \label{sec:HGSSYT}
%----------------------------------------------------------------

\subsection{The Hillman--Grassl correspondence}
%
% Given a reverse plane
% partition $\pi$ of shape $\lambda$, by adding $i-1$ to $\pi_{ij}$ we obtain a SSYT of shape $\lambda$. Conversely, given a
% SSYT $T=(T_{ij})$ of shape $\lambda$, since $T$ is strictly increasing
% in the columns then $T_{ij}\geq i-1$. So by subtracting $i-1$ to each
% $T_{ij}$ in row $i$ we obtain a RPP of shape $\lambda$. Thus
%
% \begin{equation} \label{eq:SSYT2RPP}
% \sum_{T \in \SSYT(\lambda)} q^{|T|} = q^{b(\lambda)} \sum_{\pi \in
%  \RPP(\lambda)} q^{|\pi|},
% \end{equation}
% where $b(\lambda)=\sum_i (i-1)\lambda_i$.

% Thus~\eqref{eq:HG} is
% equivalent to
% \begin{equation} \label{eq:skew-RPP}
% \sum_{\pi \in \RPP(\lambda)} q^{|\pi|} = \frac{1}{\prod_{u\in[ \lambda]} (1-q^{h(u)})}.
% \end{equation}
%
% \begin{remark} \label{rem:skew_diff_rpp_ssyt}
% Note that \eqref{eq:SSYT2RPP} does not hold for SSYT and RPP of skew shape since entries on the $i$-th
% row of SSYT do not have to be at least $i-1$ (e.g. $\ytableausetup{smalltableaux}
% \ytableaushort{
% \none 0,01}$).
% \end{remark}

Recall the \ts {\em Hillman--Grassl correspondence}
which defines a map between RPP $\pi$ of
shape $\lambda$ and arrays $A$ of nonnegative integers of shape
$\lambda$ such that $|\pi| =\sum_{u\in [\lambda]} A_u h(u)$.
Let  $\mathcal{A}(\lambda)$ be the set of such arrays. The
{\em weight} $\omega(A)$ of $A$ is the sum $\omega(A) := \sum_{u\in
  \lambda} A_u h(u)$. We review this construction
and some of its properties (see~\cite[\S 7.22]{EC2} and~\cite[\S 4.2]{SSG}).
We denote by $\HG$ the Hillman--Grassl map $\HG:\pi \mapsto A$.

\begin{definition}[Hillman--Grassl map $\Phi$]
Given a reverse plane partition $\pi$ of shape $\lambda$, let $A$ be an
array of zeroes of shape $\lambda$. Next we find a path $\mathsf{p}$ of North
and East steps in $\pi$ as follows:
%, subract $1$ from each cell of $P$ and
%obtain $\pi'$ that is still a RPP. If $P$ starts in column $c_s$ and end
%in row $r_f$, then we add $1$ to cell $(r_f,c_s)$ of $A$. The length
%of $P$, i.e. the amount we subtract to $\pi$, is the hook-length of
%cell $(c_s,r_f)$. The path $P$ is built as follows:
\begin{compactitem}
\item[(i)] Start $\mathsf{p}$ with the most South-Western nonzero entry in
  $\pi$. Let $c_s$ be the column of such an entry.
\item[(ii)] If $\mathsf{p}$ has reached $(i,j)$ and $\pi_{i,j} = \pi_{i-1,j}>0$ then
  $\mathsf{p}$ moves North to $(i-1,j)$, otherwise if $0<\pi_{i,j} <
  \pi_{i-1,j}$ then $\mathsf{p}$ moves East to $(i+1,j)$.
\item[(iii)] The path $\mathsf{p}$ terminates when the previous move is not possible
  in a cell at row $r_f$.
\end{compactitem}
Let $\pi'$ be obtained from $\pi$ by subtracting $1$ from every entry
in $\mathsf{p}$. Note that $\pi'$ is still a RPP. In the
array $A$ we add $1$ in position $A_{c_s,r_f}$ and obtain array $A'$.
We iterate these three steps until we reach a plane partition of
zeroes. We map $\pi$ to the final array $A$.
\end{definition}

\begin{theorem}[\cite{HG}]
The map $\HG:\RPP(\lambda) \to \mathcal{A}(\lambda)$ is a bijection.
\end{theorem}

Note that if $A=\HG(\pi)$ then $|\pi| = \omega(A)$ so as a corollary
we obtain~\eqref{eq:RPP-prod}.  Let us now
describe the inverse $\GH:A\mapsto \pi$ of the Hillman--Grassl map.

\begin{definition}[Inverse Hillman--Grassl map $\HG^{-1}$]
Given an array $A$ of nonnegative integers of shape $\lambda$, let
$\pi$ be the RPP of shape $\lambda$ of all zeroes. Next, we order the
nonzero entries of $A$, counting multiplicities, with the order $(i,j)<(i',j')$ if $j>j'$ or $j=j'$ and $i<i'$
(i.e. $(i,j)$ is right of $(i',j')$ or higher in the same
column). Next, for each entry $(r_s,c_j)$ of $A$ in this order
$(i_1,j_1),\ldots, (i_m,j_m)$ we build a reverse path $\mathsf{q}$ of South and
West steps in $\pi$ starting at row $r_s$ and ending in column $c_f$
as follows:
%add 1 from each cell of $Q$ and obtain $\pi'$
%that is still a RPP.  Finally, we subtract $1$ to cell $(i,j)$ of
%$A$. The reverse path $Q$ associated to entry $(r_s,c_f)$ is built as follows:
\begin{compactitem}
\item[(i)] Start $\mathsf{q}$ with the most Eastern entry of $\pi$ in row
  $r_s$.
\item[(ii)] If $\mathsf{q}$ has reached $(i,j)$ and $\pi_{i,j}=\pi_{i+1,j}$
  then $\mathsf{q}$ moves South to $(i-1,j)$, otherwise $\mathsf{q}$ moves West to
  $(i+1,j)$.
\item[(iii)] Path $\mathsf{q}$ ends when it reaches the Southern entry of $\pi$
  in column $c_f$.
\end{compactitem}
Step (iii) is actually attained (see e.g.~\cite[Lemma~4.2.4]{SSG}).
Let $\pi'$ be obtained from $\pi$ by adding $1$ from every entry in~$\mathsf{q}$.
Note that $\pi'$ is still a RPP. In the array $A$ we subtract $1$
in position $A_{c_f,r_s}$ and obtain array $A'$. We iterate this
process following the order of the nonzero entries of $A$ until we
reach an array of zeroes. We map $A$ to the final RPP $\pi$.
Note that $\omega(A)=|\pi|$.
\end{definition}

\begin{theorem}[\cite{HG}]
We have $\GH = \HG^{-1}$.
\end{theorem}

By abuse of notation, if $\pi$ is a skew RPP of shape $\lambda/\mu$,
we define $\HG(\pi)$ to be $\HG(\hat{\pi})$ where $\hat{\pi}$ is
the RPP of shape $\lambda$ with zeroes in $\mu$ and agreeing with
$\pi$ in $\lambda/\mu$:
\begin{center}
\includegraphics{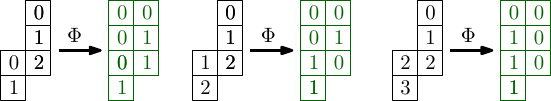}
\end{center}
Recall that unlike for straight shapes,  the enumeration of SSYT and RPP
of skew shape are not equivalent. Therefore, the image
$\HG(\SSYT(\lambda/\mu))$ is a strict subset of
$\HG(\RPP(\lambda/\mu))$. In Section~\ref{sec:HGSSYT} we characterize
the SSYT case in terms of excited diagrams, and in Section~\ref{sec:HGRPP} we
characterize the RPP case in terms of new diagrams called {\em pleasant
diagrams}. Both characterizations require a few properties of $\HG$
that we review next.

\medskip
\subsection{The Hillman--Grassl correspondence and Greene's theorem}

In this section we review key properties of the
Hillman--Grassl correspondence related to the {\em RSK
  correspondence} \cite[\S 7.11]{EC2}. We denote the RSK correspondence by \ts $\RSK:M\mapsto
(P,Q)$, where $M$ is a matrix with nonnegative integer entries and
$I(\RSK(M)):=P$, $R(\RSK(M)):=Q$ are SSYT of the same shape called the {\em insertion} and {\em
  recording} tableau, respectively.

Given a reverse plane partition $\pi$ and an integer $k$ with
$1-\ell(\lambda)\leq k \leq \lambda_1-1$, a $k$-diagonal is the sequence
of entries $(\pi_{ij})$ with $i-j=k$. Each $k$-diagonal of $\pi$ is
nonincreasing and so we denote it by a partition $\nu^{(k)}$. The {\em
  $k$-trace} of $\pi$ denoted by $\tr_k(\pi)$ is the sum of the parts
of $\nu^{(k)}$. Note that the $0$-trace of $\pi$ is the standard trace
$\tr(\pi)=\sum_i \pi_{i,i}$.

Given the Young diagram of $\lambda$ and an integer $k$ with
$1-\ell(\lambda) \leq k \leq \lambda_1-1$, let $\square_k^\lambda$ be the largest $i
\times (i+k)$ rectangle that fits inside the Young diagram starting at
$(1,1)$. For $k=0$, the rectangle $\square_0^{\lambda}=\sq^\la$ is the (usual) Durfee
square of~$\lambda$.
Given an array $A$ of shape $\lambda$, let $A_k$ be the subarray of
 $A$ consisting of the cells inside $\square^\lambda_k$ and $|A_k|$ be the sum of its entries.
Also, given a rectangular array~$B$, let
$\hf{B}$ and $\vf{B}$ denote the arrays $B$ flipped vertically and horizontally,
respectively. Here vertical flip means that the bottom row become the top row,
and horizontal means that the rightmost column becomes the leftmost column.

In the construction $\HG^{-1}$, entry~$1$ in position $(i,j)$ adds $1$ to
the $k$-trace if and only if $(i,j) \in \square^\lambda_k$. This
observation implies the following result.

\begin{proposition}[Gansner, Thm.~3.2 in~\cite{G}] \label{prop:HGtrace}
Let $A=\HG(\pi)$ then for $k$ with
$1-\ell(\lambda)\leq k \leq \lambda_1-1$ we have
\[
\tr_k(\pi) = |A_k|.
\]
\end{proposition}

As a corollary, when $k=0$, Proposition~\ref{prop:HGtrace} gives
Gansner's formula~\eqref{eq:traceeq} for the generating series
for $\RPP(\lambda)$ by size and trace.  Indeed, the generating function
for the arrays is a product over cells $(i,j)\in [\la]$ of terms
which contain~$t$ in the numerator if only if $(i,j) \in \square^\lambda$.  We refer
to~\cite{G} for the details.

\smallskip

Let us note that not only is the $k$-trace determined by
Proposition~\ref{prop:HGtrace} but also the parts of
$\nu^{(k)}$. This next result states that the partition $\nu^{(k)}$ and its
conjugate are determined by nondecreasing and nonincreasing chains
in the rectangle~$A_k$.

Given an $m\times n$ array $M=(m_{ij})$ of nonnegative integers, an {\em
  ascending chain} of length $s$ of $M$ is a sequence
  $\mathfrak{c}:=\left((i_1,j_1),(i_2,j_2),\ldots,(i_s,j_s)\right)$ where  $m\geq i_1
\geq  \cdots \geq
i_s \geq 1$ and $1\leq j_1 \leq  \cdots \leq  j_s \leq n$  where
$(i,j)$ appears in $\mathfrak{c}$ at most $m_{ij}$ times. A
{\em descending chain} of length $s$ is a sequence $\mathfrak{d}:=
\left((i_1,j_1),(i_2,j_2),\ldots,(i_s,j_s)\right)$ where $1\leq i_1 < \cdots <
i_s \leq m$ and $1\leq j_1 < \cdots < j_s \leq n$ where $(i,j)$
appears in $\mathfrak{d}$ only if $m_{ij}\neq 0$.

Let $ac_1(M)$ and $dc_1(M)$ be
the length of the longest ascending and descending chains in $M$
respectively. In general for $t\geq 1$, let $ac_t(M)$
be the maximum combined length of $t$ ascending chains where the
combined number of times $(i,j)$ appears is $m_{ij}$. We define
$dc_t(M)$ analogously for descending chains.

\begin{theorem}[Part~(i) by Hillman--Grassl~\cite{HG}, part~(ii) by
  Gansner~\cite{G}]  \label{thm:HGdiag}
Let $\pi \in \RPP(\lambda)$ and let $1-\ell(\lambda)\leq k \leq \lambda_1-1$.  Denote by $\nu=\nu^{(k)}$
the partition whose parts are the entries on the $k$-diagonal of~$\pi$, and let $A=\HG(\pi)$.  Then, for all $t\geq 1$ we
have:
\begin{itemize}
\item[{\rm (i)}]  $ac_t(A_k) = \nu_1+\nu_2+\cdots+\nu_t$,
\item[{\rm (ii)}] $dc_t(A_k)=\nu'_1+\nu'_2+\cdots + \nu'_t$.
\end{itemize}
\end{theorem}

\begin{remark} {\rm
This result is the analogue of \emph{Greene's theorem} for the RSK
correspondence~$\RSK$, see e.g.~\cite[Thm.~A.1.1.1]{EC2}. In fact, we have the following explicit connection with RSK.
%the partition $\nu^{(k)}$ from each $k$-diagonal gives the shape of the tableaux of $\RSK(\hf{A_k})$ and $\RSK(\vf{A_k})$.

}
\end{remark}

\begin{corollary} \label{cor:HGvsRSK}
Let $\pi$ be in $\RPP(\lambda)$, $A=\HG(\pi)$, and let~$k$ be an
integer $1-\ell(\lambda)\leq k \leq \lambda_1-1$.  Denote by
$\nu^{(k)}$ is the partition obtained from the $k$-diagonal of~$\pi$.
Then the shape of the tableaux in $\RSK(\hf{A_k})$
and $\RSK(\vf{A_k})$ is equal to~$\nu^{(k)}$.
\end{corollary}

\begin{example}
Let $\la= (4,4,3,1)$ and $\pi\in \RPP(\la)$ be as below.  Then we have:
$$\pi=
\begin{ytableau}
*(green) 0& *(yellow)1&3&4\\
1& *(green) 3&*(yellow) 5&6\\
3&6& *(green) 7\\
3
\end{ytableau} \qquad
A=\HG(\pi) = \begin{ytableau}
0&2&1&1\\ 1&1&1&2 \\ 2&1&1\\
0
\end{ytableau}
$$
Note that $\nu^{(0)}=(7,3)$ and
indeed $\ell(\nu^{(0)})=2=dc_1(A_0)$.  For example, take
$\mathfrak{d}=\{(2,2),(3,3)\}$.  Similarly, $\nu^{(1)}=(5,1)$,
$\ell(\nu^{(1)})=2=dc_1(A_1)$.  Applying RSK to~$\vf{A_1}$
and~$\vf{A_0}$ we get tableaux of shape $\nu^{(1)}$ and
$\nu^{(0)}$, respectively:
$$
I(\RSK(\vf{A_1})) = I(\RSK \begin{ytableau}
1&2&0\\ 1&1&1
\end{ytableau}) = \begin{ytableau}
1&1&2&2&3\\
2
\end{ytableau}\,,\qquad \ \
I(\RSK(\vf{A_0})) = I(\RSK \begin{ytableau}
1&2&0\\ 1&1&1 \\ 1&1&2
\end{ytableau}) =
\begin{ytableau}
1& 1& 1& 2& 2&3& 3\\
2&2&3
\end{ytableau}.
$$
\end{example}

% example
% description of inverse HG

%----------------------------------------------------------------
\bigskip
\section{Hillman--Grassl map on skew RPP} \label{sec:HGRPP}
%----------------------------------------------------------------

\nin
In this section we show that the Hillman--Grassl map is a
bijection between RPP of skew shape and arrays of nonnegative
integers with support on certain diagrams related to excited diagrams.

\subsection{Pleasant diagrams} \label{sec:pleasant}

We identify any diagram $S$ (set of boxes in $[\lambda]$) with its corresponding $0$-$1$ indicator array, i.e. array of shape $\lambda$ and support $S$.

\begin{definition}[Pleasant diagrams] \label{def:agog}
 A diagram $S\subset [\lambda]$ is a {\em \pleasant~diagram} of $\lambda/\mu$ if for all
 integers $k$ with $1-\ell(\lambda)\leq k \leq \lambda_1-1$, the subarray
 $S_k := S\cap \square^{\lambda}_k$ has no descending chain bigger than
 the length $s_k$ of the diagonal $\dd_k$ of $\lambda/\mu$, i.e. for every $k$ we have $dc_1(S_k) \leq s_k$. We denote the set of
 \pleasant~diagrams of $\lambda/\mu$ by $\PD(\lambda/\mu)$.
\end{definition}

%\begin{remark}
%If we identify $S\subset [\lambda]$ with the $0$-$1$
% array of shape $\lambda$ with support $S$, then the condition for
% pleasant diagrams is equivalent to  $dc_1(S_k) \leq s_k$ for all $k$ with $1-\ell(\lambda)\leq k \leq \lambda_1-1$.
%\end{remark}

\begin{example} \label{ex:pleasant12}
The skew shape $(2^2/1)$ has $12$ \pleasant~diagrams of which two are
complements of excited diagrams (the first in each row):
\begin{center}
\includegraphics{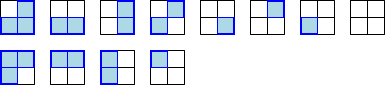}.
\end{center}
These are diagrams $S$ of $[2^2]$ where $S \cap \square^{\lambda}_{-1}, S\cap \square^{\lambda}_0$ and
$S\cap \square^\lambda_1$ have no descending chain bigger than
$s_{k}=|\dd_k|=1$ for $k$ in $\{-1,0,1\}$.
\end{example}

\begin{theorem} \label{thm:rpp_pleasant}
A RPP $\pi$ of shape $\lambda$ has support in a skew shape $\lambda/\mu$ if and only if
the support of $\HG(\pi)$ is a pleasant diagram in
$\PD(\lambda/\mu)$. In particular
\begin{equation} \label{eq:rpppleasant}
\sum_{\pi \in \RPP(\lambda/\mu)} q^{|\pi|} = \sum_{S \in
  \PD(\lambda/\mu)} \left[ \prod_{u\in S} \frac{q^{h(u)}}{1-q^{h(u)}} \right].
\end{equation}
\end{theorem}

\begin{proof}
By Theorem~\ref{thm:HGdiag}, a RPP $\pi$ of shape $\lambda$ has
support in the skew shape $\lambda/\mu$ if and only if $A=\HG(\pi)$
satisfies
\[
dc_1(A_k)=\nu'_1 \leq s_k,
\]
for $1-\ell(\lambda) \leq k \leq
\lambda_1-1$, where $\nu=\nu^{(k)}$. In other words, $\pi$ has support
in the skew shape $\lambda/\mu$ if and
only if the support $S\subseteq [\lambda]$ of $A$ is in
$\PD(\lambda/\mu)$. Thus, the Hillman--Grassl map is a
bijection between $\RPP(\lambda/\mu)$ and arrays of nonnegative
integers of shape $\lambda$ with support in a \pleasant~diagram $S \in
\PD(\lambda/\mu)$. This proves the first
claim. Equation~\eqref{eq:rpppleasant} follows since $|\pi|=\omega(\HG(\pi))$.
\end{proof}

\begin{remark}The proof of Theorem~\ref{thm:rpp_pleasant} gives an alternative description for
pleasant diagrams $\PD(\lambda/\mu)$ as the supports of $0$-$1$ arrays $A$ of shape
$\lambda$ such that $\HG^{-1}(A)$ is in $\RPP(\lambda/\mu)$.
\end{remark}

We also give a generalization of the trace generating function~\eqref{eq:traceeq}
for these reverse plane partitions.

\begin{proof}[Proof of Theorem~\ref{thm:RPP-trace}]
Given a pleasant diagram $S\in \PD(\lambda/\mu)$, let $\mathcal{B}_S$
be the collection of arrays of shape $\lambda$ with support in $S$. Given a RPP
$\pi$, let $A=\HG(\pi)$.
By Theorem~\ref{thm:rpp_pleasant} $\pi$ has shape $\lambda/\mu$
if and only if $A$ has support in a pleasant diagram $S$ in
$\PD(\lambda/\mu)$. Thus
\begin{equation}\label{eq1:tracerpp}
\sum_{\pi \in \RPP(\lambda/\mu)} \. q^{|\pi|}\ts t^{\tr(\pi)} \, \, = \, \sum_{S\in
  \PD(\lambda/\mu)} \, \sum_{\pi \in \HG^{-1}(\mathcal{B}_S)} \.
q^{|\pi|}\ts t^{\tr(\pi)}\,,
\end{equation}
where for each $S\in \PD(\lambda/\mu)$ we have
\begin{equation} \label{eq2:tracerpp}
\sum_{\pi \in \HG^{-1}(\mathcal{B}_S)} \. q^{|\pi|} \, = \,
\prod_{u\in S} \frac{q^{h(u)}}{1-q^{h(u)}}\,.
\end{equation}
Next, by Proposition~\ref{prop:HGtrace} for
$k=0$, the trace $\tr(\pi)$ equals $|A_0|$, the sum of the entries of
$A$ in the Durfee square $\square^{\lambda}$ of $\lambda$. Therefore,
we refine \eqref{eq2:tracerpp} to keep track of the trace of the RPP
and obtain
\begin{equation} \label{eq3:tracerpp}
\sum_{\pi \in \HG^{-1}(\mathcal{B}_S)} \. q^{|\pi|}\ts t^{\tr(\pi)}
\, = \, \prod_{u\in S\cap \square^\lambda}
  \frac{t\ts q^{h(u)}}{1-t\ts q^{h(u)}}\, \prod_{u\in S \setminus
    \square^\lambda} \. \frac{q^{h(u)}}{1-q^{h(u)}}\,.
\end{equation}
Combining \eqref{eq1:tracerpp} and \eqref{eq3:tracerpp} gives the
desired result.
\end{proof}

\medskip
\subsection{Combinatorial proof of NHLF \eqref{eq:Naruse}: relation between pleasant and excited diagrams. }
Theorem~\ref{thm:skewSSYT} relates SSYT of skew shape with excited
diagrams and Theorem~\ref{thm:rpp_pleasant}  relates RPP of skew shape with
pleasant diagrams. Since SSYT are RPP then we expect a relation
between pleasant and excited diagrams of a fixed skew shape
$\lambda/\mu$. The first main result of this subsection characterizes the
pleasant diagrams of maximal size in terms of excited diagrams. The
second main result characterizes all pleasant diagrams.

The key towards these results is  a more graphical characterization of
pleasant diagrams as described in the proof of
Lemma~\ref{lem:pleasant_is_subset_of_excited}. It makes the
relationship with excited diagrams more apparent and also allows for a more intuitive description for both kinds of diagrams.

\begin{theorem} \label{thm:when_excited_pleasant}
A \pleasant~diagram $S\in \PD(\lambda/\mu)$ has size $|S|\leq
|\lambda/\mu|$ and has maximal size $|S| = |\lambda/\mu|$ if
and only if the complement $[\lambda]\setminus S$ is an excited
diagram in $\ED(\lambda/\mu)$.
\end{theorem}

By combining this theorem with Theorem~\ref{thm:rpp_pleasant} we
derive again the NHLF. In contrast with the
derivation of this formula in
Proposition~\ref{cor:getnaruse} (our first proof of the NHLF), this derivation is entirely combinatorial.
\begin{proof}[Second proof of the NHLF \eqref{eq:Naruse}]
By Stanley's theory of {\em $P$-partitions},
\cite[Thm.~3.15.7]{EC2}
\begin{equation} \label{eq1:RPPskewPpartitions}
\sum_{\pi \in \RPP(\lambda/\mu)} q^{|\pi|} \, = \, \frac{\sum_{w \in \mathcal{L}(P_{\lambda/\mu})}
 \, q^{\maj(w)}}{\prod_{i=1}^n (1-q^i)}\,,
\end{equation}
where $n=|\lambda/\mu|$ and the sum in the numerator of the RHS is over linear extensions
$w$ of the poset $P_{\lambda/\mu}$ with a {\em natural labelling}. Multiplying
\eqref{eq1:RPPskewPpartitions} by $(1-q)\cdots (1-q^n)$, and using
Theorem~\ref{thm:skewRPP} gives
\begin{equation} \label{eq2:RPPskewPpartitions}
\sum_{w \in \mathcal{L}(P_{\lambda/\mu})}   q^{\maj(w)} \, = \,
\left( \prod_{i=1}^n (1-q^i) \right) \sum_{S\in \PD(\lambda/\mu)} \.
\prod_{u\in S} \. \frac{q^{h(u)}}{1-q^{h(u)}}\..
\end{equation}
By Theorem~\ref{thm:when_excited_pleasant}, pleasant diagrams $S \in
\PD(\lambda/\mu)$ have size $|S|\leq n$, with the equality here exactly
when $\overline{S}\in \ED(\lambda/\mu)$.  Thus, letting \ts $q \to 1$ \ts
in~\eqref{eq2:RPPskewPpartitions} gives $f^{\lambda/\mu}$ on the LHS.
On the RHS, we obtain the sum of products
$$\prod_{u\in \overline{S}} \. \frac{1}{h(u)}\,,
$$
over all excited
diagrams $S \in \ED(\lambda/\mu)$. This implies the NHLF~\eqref{eq:Naruse}.
\end{proof}

\begin{lemma}  \label{lem:pleasant_is_subset_of_excited}
Let \ts $S\in \PD(\lambda/\mu)$. Then there is an excited diagram $D \in
\ED(\lambda/\mu)$, such that \ts $S \subseteq
[\lambda]\setminus D$.
\end{lemma}

\begin{proof}
Given a pleasant diagram $S$, we use Viennot's {\em shadow lines}
construction \cite{XV77} to obtain a family of nonintersecting paths on
$[\lambda]$. That is, we imagine there is a light source at the $(1,1)$ corner of $[\lambda]$ and the
elements  of $S$ cast shadows along the axes with vertical and horizontal segments. The boundary of the resulting shadow
forms the first shadow line $L_1$. If lines $L_1,L_2,\ldots,L_{i-1}$ have
already been defined we
define $L_i$ inductively as follows: remove the elements of $S$ contained in any of the $i-1$
lines and set $L_i$ to be the shadow line of the remaining elements of $S$. We iterate
this until no element of $S$ remains in the shadows. Let
$L_1,L_2,\ldots,L_m$ be the shadow lines produced. Note that these
lines form $m$ nonintersecting paths in $[\lambda]$ that go from
bottom south-west cells of columns to rightmost north-east cells of rows of the diagram.

By construction, the {\em peaks} (i.e. top corners) of the shadow lines $L_i$ are
elements of $S$ while other cells of $L_i$ might be in $[\lambda]\setminus
S$.

Next, we augment $S$ to obtain $S^*$ by adding all the cells of lines
$L_1,\ldots,L_m$ that are not in $S$. Note that $S^*$ is also a pleasant diagram in $\PD(\lambda/\mu)$ since
the added cells of the lines $L_1,\ldots,L_m$ do not yield longer
decreasing chains than those in $S$. Moreover, no two cells from a decreasing chain can be part of the same shadow line, and there is at least one decreasing subsequence with cells in all lines, as can be constructed by induction.  In particular, the number of shadow lines
intersecting each diagonal $\dd_k$ (i.e. intersecting the rectangle $\square^{\lambda}_k$) is at most $s_k$. Denote this
number by $s'_k$.

Next, we claim that $S^*$ is the complement of an excited diagram $D^*
\in \ED(\lambda/\nu)$ for some partition $\nu$. To see this we do
 moves on the noncrossing paths (shadow lines) that are analogous to reverse excited moves, as follows. If
the lines contain $(i,j), (i+1,j), (i,j+1)$ but not $(i+1,j+1)$, then
notice that the first three boxes lie on one path $L_t$. In this path we replace $(i,j)$  with $(i+1,j+1)$ to obtain path~$L'_t$. We do the same
replacement in $S^*$:
\begin{center}
\includegraphics{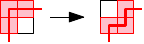}
\end{center}
Following Kreiman \cite[\S 5]{VK} we call this move a {\em reverse
  ladder move}. By doing reverse ladder moves
iteratively on $S^*$ we obtain the complement of some Young diagram
$[\nu]\subseteq [\lambda]$.

\begin{figure}[hbt]
\begin{center}
\includegraphics{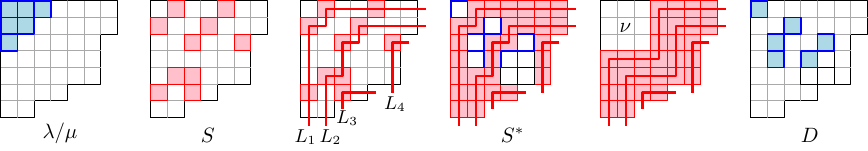}
\caption{Example of the construction in Lemma~\ref{lem:pleasant_is_subset_of_excited}. From left to
  right: a shape $\lambda/\mu$, a
  pleasant diagram $S\in \PD(\lambda/\mu)$, the shadow lines
  associated to $S$, the augmented pleasant diagram $S^*$ that is a
  complement of an excited diagram $D^*$ in $\ED(\lambda/\nu)$ for some
  $\nu$, $\mu \subseteq \nu \subseteq \lambda$. In general, $D^*$
  contains all $D\in \ED(\lambda/\mu)$ with $S\subseteq
  [\lambda]\setminus D$.}
\label{fig:shadow}
\end{center}
\end{figure}

Next, we show that $\mu \subseteq \nu$. Reverse ladder moves do not change the number $s'_k$ of
shadow lines intersecting each diagonal, thus $s'_k$ is also the
length of the diagonal $\dd_k$ of $\lambda/\nu$. Since $s'_k \leq
s_k$, the length of the diagonal $\dd_k$ of $\lambda/\mu$, then $\mu
\subseteq \nu$ as desired.

Finally, we have \ts $D^* = [\lambda] \setminus S^*$ is in $\ED(\lambda/\nu)$, since the reverse ladder move is the reverse excited move on the corresponding diagram. Since $D^*$ is obtained my moving the cells of $[\nu]$ we can consider the cells of $D^*$ which correspond to the cells of $[\mu]\subseteq [\nu]$, denote the collection of these cells as $D$. Then $D \in \ED(\lambda/\mu)$, and we have:
$$S \subseteq S^* = [\lambda] \setminus D^* \subseteq [\lambda] \setminus D$$
 and the statement follows.
 % See Figure~\ref{fig:shadow} for an example of these constructions.
\end{proof}
% ??? By restricting $D^*$ to the cells in $\nu$ corresponding to $\mu$ we
%obtain an excited diagram $D$ in $\ED(\lambda/\mu)$. Then since
%$S\subseteq S^*$, $S^* = [\lambda]\setminus D^*$ and $D\subseteq D^*$
%then $S\subseteq [\lambda]\setminus D$ as desired.

We prove Theorem~\ref{thm:when_excited_pleasant} via three lemmas.

\begin{lemma} \label{lemma:excitedispleasant}
For all \ts $D\in \ED(\lambda/\mu)$, we have \ts $[\lambda]\setminus D \in \PD(\lambda/\mu)$.
\end{lemma}

\begin{proof}

Let $D_0 = \mu$, i.e.~the excited diagram which corresponds to the
original skew shape $\lambda/\mu$. Following the shadow line
construction from the proof of
Lemma~\ref{lem:pleasant_is_subset_of_excited}, we construct the shadow
lines for the diagram  $P_0=[\lambda/\mu]$. These lines trace out the
{\em rim-hook tableaux}: let $L_1$ be the outer boundary of $[\mu]$ inside $[\lambda]$, then $L_2$ is the outer boundary of what remains after removing $L_1$, etc. If the skew shape becomes disconnected then there are separate lines for each connected segment.

\begin{center}
\includegraphics{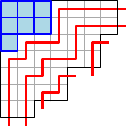}
\end{center}

Since a diagonal of length $\ell$ has exactly $\ell$ shadow lines
crossing it, we have for each rectangle $D_k$ there are exactly $s_k$
lines $L_i$ crossing $\dd_k$ and hence also crossing $D_k$. An excited move corresponds to a ladder move on some line (see the proof of Lemma~\ref{lem:pleasant_is_subset_of_excited}), which makes an inner corner of a line into an outer corner. These moves cannot affect the endpoints of a line, so if a line does not cross a rectangle $D_k$ initially then it will never cross it after any number of excited moves. Moreover, any diagonal $\dd_k$ will be crossed by the same set of lines formed originally in $P_0$. Hence the complement of any excited diagram is a collection of shadow lines, which were obtained from the original ones by ladder moves. Then the number of shadow lines crossing $D_k$ is always $s_k$. Finally, since no decreasing sequence can have more than one box on a given shadow line (i.e., a SW to NE lattice path), we have the longest decreasing subsequence in $D_k$ will have length at most $s_k$ -- the number of shadow lines there. Therefore, the excited diagram satisfies Definition~\ref{def:agog}. \end{proof}

By Lemma~\ref{lemma:excitedispleasant}, the complements of excited diagrams in
$\ED(\lambda/\mu)$ give pleasant diagrams of size
$|\lambda/\mu|$. Next, we show that there are no pleasant diagrams of
larger size.

\begin{lemma} \label{lemma:boundsizepleasant}
For all \ts $S\in \PD(\lambda/\mu)$, we have \ts $|S|\leq |\lambda/\mu|$.
\end{lemma}

\begin{proof}
For each diagonal $\dd_k$ of $\lambda/\mu$, any elements of $S \cap
\dd_k$ form a descending chain in $S_k$. Thus, by definition of
pleasant diagrams $|S \cap \dd_k|\leq s_k$ where
$s_k=\left|[\lambda/\mu]\cap \dd_k\right|$ is the length of diagonal $\dd_k$
in~$\lambda/\mu$. Therefore,
\begin{align*}
|S| &\, = \, \sum_{k=1-\ell(\lambda)}^{\lambda_1-1} \. |S\cap \dd_k| \, \, \leq \,  \sum_{k=1-\ell(\lambda)}^{\lambda_1-1} \. s_k \,  = \, |\lambda/\mu|\.,
\end{align*}
as desired.
\end{proof}

The next result shows that the only pleasant diagrams of size
$|\lambda/\mu|$ are complements of excited diagrams.

\begin{lemma} \label{lemma:bigpleasantisexcited}
For all \ts $S\in \PD(\lambda/\mu)$ \ts with \ts $|S| = |\lambda/\mu|$, we have \ts $[\lambda]\setminus S \in \ED(\lambda/\mu)$.
\end{lemma}

\begin{proof}
By the argument in the proof of Lemma~\ref{lemma:boundsizepleasant},
if  $S \in \PD(\lambda/\mu)$ has size $|S| = |\lambda/\mu|$ then
for each integer $k$ with $1-\ell(\lambda)\leq k \leq \lambda_1$ we have $|S\cap \dd_k| = |\dd_k|=s_k$.

Suppose $\overline{S} = [\lambda]\setminus S$ is not an excited
diagram. This means that there are two cells $a=(i,j),b=(i+m,j+m) \in \overline{S}$ on
some diagonal $\dd_k$ with no other cell of $\dd_k$ in $\overline{S}$ between them, that violate the interlacing property
(Definition~\ref{def:interlacing}). This means that
there are no other cells in $\overline{S}$ between cells
$a$ and $b$ in either diagonal $\dd_{k+1}$ or diagonal
$\dd_{k-1}$. Without loss of generality assume that this occurs in
diagonal $\dd_{k-1}$. This means that all the $m$ cells in $\dd_{k-1}$
between cells $a$ and $b$ are in $S$. Let $\mathfrak{d}$ be the
descending chain in $S$ of all the $s_k$ cells in $S\cap \dd_k$
including the $m-1$ cells in $\dd_k$ between $a$ and $b$. Let
$\mathfrak{d}'$ be the descending chain consisting of the cells in
$S\cap \dd_k$ before cell $a$, followed by the $m$ cells in $S \cap
\dd_{k-1}$ between cell $a$ and $b$, and the cells in $S\cap \dd_k$
after cell $b$ (see Figure~\ref{fig:plesant_is_excited}). However
$|\mathfrak{d}'| = s_k+1$ which contradicts the requirement that all
descending chains in $S\cap \square^{\lambda}_k$ have length $\leq s_k$.
\end{proof}

\begin{figure}[hbt]
\begin{center}
\includegraphics{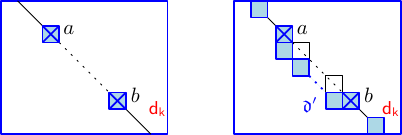}
\caption{Two consecutive cells $a$ and $b$ in $\overline{S}$
  that violate the interlacing property of excited diagrams.}
\label{fig:plesant_is_excited}
\end{center}
\end{figure}

\begin{proof}[Proof of Theorem~\ref{thm:when_excited_pleasant}]
The result follows by combining Lemmas~\ref{lemma:excitedispleasant} and
\ref{lemma:bigpleasantisexcited}.
\end{proof}

\begin{theorem} \label{conj:charpleasant}
A diagram $S\subset [\lambda]$ is a pleasant diagram in
$\PD(\lambda/\mu)$ if and only if $S \subseteq
[\lambda]\backslash D$ for some excited diagram $D\in \ED(\lambda/\mu)$.
\end{theorem}

We need a new lemma to prove this result.

\begin{lemma} \label{lem:charpleasantlem1}
Given an excited diagram $D$ in $\ED(\lambda/\mu)$ then $S \subseteq
[\lambda]\backslash D$ is a pleasant diagram in $\PD(\lambda/\mu)$.
\end{lemma}

\begin{proof}
Theorem~\ref{thm:when_excited_pleasant} characterizes maximal pleasant
diagrams in  $\PD(\lambda/\mu)$ as complements of excited diagrams in $\ED(\lambda/\mu)$. Since subsets of pleasant
diagrams are also pleasant diagrams, then all subsets $S$ of $[\lambda]\setminus D$ for $D\in \ED(\lambda/\mu)$ are pleasant diagrams.
\end{proof}

\begin{proof}[Proof of Theorem~\ref{conj:charpleasant}]
The theorem follows from Lemma~\ref{lem:charpleasantlem1} and Lemma~\ref{lem:pleasant_is_subset_of_excited}.
\end{proof}

\medskip \subsection{Enumeration of pleasant diagrams}

Next we give two formulas for the number of pleasant
diagrams of $\lambda/\mu$ as sums of excited diagrams. Both formulas are corollaries of the proof of
Lemma~\ref{lem:pleasant_is_subset_of_excited}. Given a pleasant
diagram $S$, let $\shpeaks(D)$ be the number of peaks of the shadow lines $L_1,\ldots,L_m$ obtained
from the pleasant diagram $[\lambda]\setminus {D}$.

\begin{proposition}\label{prop:pleasant-excited-sum}
\[
|\PD(\lambda/\mu)| = \sum_{\nu, \mu \subseteq \nu \subseteq \lambda} \sum_{D\in
  \ED(\lambda/\nu)} 2^{|\lambda/\nu|-\shpeaks(D)}.
\]
\end{proposition}

\begin{example}
The skew shape $(2^2/1)$ has $12$ pleasant diagrams (see
Example~\ref{ex:pleasant12}). The possible $\nu$ containing $\mu=(1)$
are $(1), (1^2), (2), (2,1), (2,2)$ and their corresponding excited
diagrams with peaks (in pink) are the following:
\begin{center}
\includegraphics{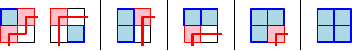}.
\end{center}
We can see that $12=2^1+2^2+2^1+2^1+2^0+2^0$.
\end{example}

\begin{proof}[Proof of Proposition~\ref{prop:pleasant-excited-sum}]
As in the proof of Lemma~\ref{lem:pleasant_is_subset_of_excited},
from the shadow lines $L_1,L_2,\ldots,L_m$ of a pleasant diagram $S\in
\PD(\lambda/\mu)$ we obtain an excited diagram $D^* \in
\ED(\lambda/\nu)$ for $\mu\subseteq \nu$ such that $S\subseteq
[\lambda]\setminus D^*$. The peaks of these lines are elements in~$S$,
and these peaks uniquely determine the lines. The other cells in the
lines, $|\lambda/\nu|-\shpeaks(D^*)$ many, may or may not be in~$S$.

Therefore, we obtain a surjection
\[
\varrho_1:  \. \PD(\lambda/\mu)\to \bigcup_{\nu, \mu\subseteq \nu\subseteq
  \lambda} \ED(\lambda/\nu)\.,  \qquad \varrho_1: \. S\mapsto D^*\.,
\]
 such that \ts $|\varrho_1^{-1}(D^*)|=2^{|\lambda/\nu|-\shpeaks(D^*)}$.
 This implies the result (see Figure~\ref{fig:exrhos}).
\end{proof}

\begin{figure}[hbt]
\begin{center}
\includegraphics{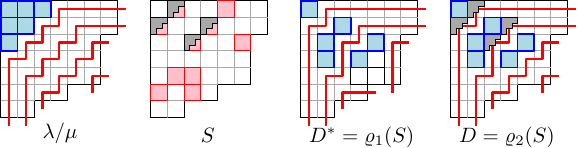}
\end{center}
\caption{Example of the maps $\varrho_1$ and $\varrho_2$ on a pleasant
  diagram $S$.}
\label{fig:exrhos}
\end{figure}

For the second formula we need to define a similar peak statistic $\expeaks(D)$ for each excited diagram
$D\in\ED(\lambda/\mu)$.  For an excited diagram $D$ we associate a subset of
$[\lambda]\setminus D$ called {\em excited peaks} and denote it by
$\Lambda(D)$ in the following way. For $[\mu] \in
\ED(\lambda/\mu)$ the set of excited peaks is
$\Lambda([\mu])=\varnothing$. If $D$
is an excited diagram with active cell $u=(i,j)$ then the excited peaks of
$\alpha_u(D)$ are
\[
\Lambda(\alpha_u(D)) = \left(\Lambda(D)- \{(i,j+1),(i+1,j)\}\right) \cup \{u\}.
\]
That is, the excited peaks of $\alpha_u(D)$ are obtained from those of
$D$ by adding $(i,j)$ and removing $(i,j+1)$ and $(i+1,j)$ if any of
the two are
in $\Lambda(D)$:
\begin{center}
\includegraphics{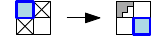}.
\end{center}
Finally, let $\expeaks(D):=|\Lambda(D)|$ be
the number of excited peaks of $D$.

\begin{theorem} \label{thm:num_pleasant}
For a skew shape $\lambda/\mu$ we have
\[
|\PD(\lambda/\mu)| = \sum_{D\in \ED(\lambda/\mu)} 2^{|\lambda/\mu|-\expeaks(D)},
\]
where $\expeaks(D)$ is the number of excited peaks of the excited diagram $D$.
\end{theorem}

We prove Theorem~\ref{thm:num_pleasant} via the following Lemma. Given
a set $\mathcal{S}$, let $2^{\mathcal{S}}$ denote the subsets of $\mathcal{S}$.

\begin{lemma} \label{lem:num_pleasant}
We have $\PD(\lambda/\mu) = \bigcup_{D\in \ED(\lambda/\mu)} \Lambda(D) \times
2^{[\lambda]\setminus (D\cup \Lambda(D))}$.
\end{lemma}

\begin{proof}
As in the proof of Lemma~\ref{lem:pleasant_is_subset_of_excited}, from the shadow lines $L_1,L_2,\ldots,L_m$ of a pleasant diagram $S\in
\PD(\lambda/\mu)$ we obtain an excited diagram $D^* \in
\ED(\lambda/\nu)$ for $\mu\subseteq \nu$ such that $S\subseteq
[\lambda]\setminus D^*$. If we restrict $D^*$ to the cells coming from
$[\mu]$ we obtain an excited diagram $D\in \ED(\lambda/\mu)$.
Setting $\varrho_2(S)= D$ defines a new surjection \.
$\varrho_2:\ts \PD(\lambda/\mu)\to \ED(\lambda/\mu)$ (see Figure~\ref{fig:exrhos}).
It remains to prove that
$$
\varrho_2^{-1}(D) \, = \, \Lambda(D)\times 2^{[\lambda]\setminus (D\cup \Lambda(D))}\..
$$
First, the excited peaks are
peaks of the shadow lines $L'_1,L'_2,\ldots,L'_k$  of
$[\lambda]\setminus D$ obtained by a {\em ladder move}:
\vspace{-0.125in}
\begin{center}
\includegraphics{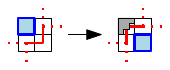}
\end{center}
\vspace{-0.1in}
Thus the peaks of the shadow lines $\{L'_i\}$ are either excited peaks or
original peaks of the shadow lines of $[\lambda/\mu]$. Second, note
that the
excited peaks $\Lambda(D)$ determine uniquely the excited diagram
$D$. Thus the non-excited peaks of the shadow lines and the other
cells of the lines $\{L'_i\}$, those in $[\lambda]\setminus (D\cup \Lambda(D))$, may
or may not be in $S$. This proves the claim for $\varrho_2^{-1}(D)$.
\end{proof}

\begin{proof}[Proof of Theorem~\ref{thm:num_pleasant}]
By Lemma~\ref{lem:num_pleasant} and since $|[\lambda]\setminus (D\cup
\Lambda(D))|=|\lambda/\mu|-\expeaks(D)$ then
\begin{align*}
|\PD(\lambda/\mu)| = \sum_{D\in \ED(\lambda/\mu)} 2^{|\lambda/\mu|-\expeaks(D)},
\end{align*}
as desired.
\end{proof}

\begin{example}
The skew shape $(2^2/1)$ has $12$ pleasant diagrams (see Example~\ref{ex:pleasant12}) and
$2$ excited diagrams, with sets of excited peaks $\varnothing$ and $\{(1,1)\}$,
respectively. Indeed, we have $|\PD(2^2/1) = 2^3+2^2=12$. A more
complicated example is shown in Figure~\ref{fig:excited_peaks}.
The number of pleasant diagrams in this case is
$|\PD(4^3/2)|=2^{10}+2\cdot 2^9 + 3\cdot 2^8 =2816$.
\end{example}

\begin{figure}[hbt]
\begin{center}
\includegraphics[scale=0.8]{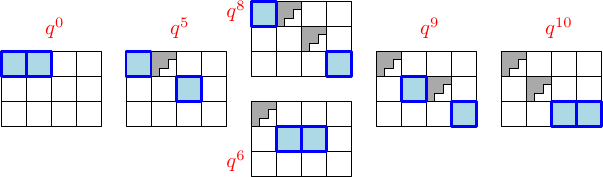}
\end{center}
\caption{The six excited diagrams $D$ for $\lambda/\mu=(4^3/2)$,
their corresponding excited peaks (in gray), and weights $a'(D)$,
defined as sums of hook-lengths of these peaks.}
\label{fig:excited_peaks}
\end{figure}

\medskip \subsection{Excited diagrams and skew RPP}\label{ss:pleasant-skewRPP-excited}

In Section~\ref{sec:pleasant} we expressed the generating function of skew RPP
using pleasant diagrams. In this section we use Lemma~\ref{lem:num_pleasant}  to give
an expression for this generating series in terms of excited diagrams.

\begin{corollary} \label{cor:skewRPP}  We have:
\[
\sum_{\pi \in \RPP(\lambda/\mu)} q^{|\pi|} \, = \, \sum_{D\in
  \ED(\lambda/\mu)} q^{a'(D)} \prod_{u\in [\lambda]\setminus D}\frac{1}{1-q^{h(u)}}\.,
\]
where \ts $a'(D):=\sum_{u \in \Lambda(D)} h(u)$.
\end{corollary}

\begin{example}
The shape $\lambda/\mu=(4^3/2)$ has six excited diagrams. See
Figure~\ref{fig:excited_peaks} for the corresponding statistic
$a'(D)$ of each of these diagrams.
\end{example}

\begin{example}
Following Example~\ref{ex:skewhook}, take the \emph{inverted hook shape} \ts
$(k^d/(k-1)^{d-1}$ and apply Corollary~\ref{cor:skewRPP}. Using Stanley's theory of
$P$-partitions, we obtain:
\begin{equation} \label{eq:skewhookRPP}
  \prod_{i=1}^{k+d-1} \frac{1}{1-q^{i}}\,
{\Bigg [}\sum_{S\in \binom{[k+d-2]}{k-1}}  q^{\maj(S)} {\Bigg ]}  \, = \, \sum_{\ga:\, (d,1) \to (1,k)} q^{a'(\gamma)} \, \prod_{(i,j)
  \in \ga} \frac{1}{1-q^{i+j-1}}\,,
\end{equation}
where
$$
\maj(S) \. = \. \sum_{i\not\in S, i+1 \in S} (i+1) \qquad \text{and} \qquad
a'(\gamma)\. = \. \sum_{(i,j) \text{ peak of } \gamma} \. (i+j-1)\,.$$
The
$q$-analogue of the binomial coefficients in the RHS of
\eqref{eq:skewhookRPP} appears to be new.
\end{example}

\begin{proof}[Proof of Corollary~\ref{cor:skewRPP}]
By Theorem~\ref{thm:rpp_pleasant}, we have:
\[
\sum_{\pi \in \RPP(\lambda/\mu)} q^{|\pi|} \, = \, \sum_{S \in
  \PD(\lambda/\mu)} \. \prod_{u \in S} \. \frac{q^{h(u)}}{1-q^{h(u)}}\..
\]
Using Lemma~\ref{lem:num_pleasant} and the surjection $\vartheta_2$ in
its proof, we can rewrite the RHS above as
a sum over excited diagrams.  We have:
\begin{align*}
\sum_{S \in
  \PD(\lambda/\mu)} \. \prod_{u \in S} \.
  \frac{q^{h(u)}}{1-q^{h(u)}} &= \sum_{D\in \ED(\lambda/\mu)} \. \sum_{S \in \varrho_2^{-1}(D)} \.
\prod_{u \in S} \.   \frac{q^{h(u)}}{1-q^{h(u)}} \\
&= \, \sum_{D\in \ED(\lambda/\mu)}  \. \prod_{u\in \Lambda(D)} \.   q^{a'(D)} \,
\prod_{u\in [\lambda]\setminus D} \. \frac{1}{1-q^{h(u)}}\.,
\end{align*}
as desired.
\end{proof}

This result also implies the NHLF~\eqref{eq:Naruse}.

\begin{proof}[Third proof of the NHLF~\eqref{eq:Naruse}]
By Stanley's theory of {$P$-partitions},
\cite[Thm.~3.15.7]{EC2} we obtain
\eqref{eq1:RPPskewPpartitions}. Multiplying this equation by
$\prod_{i=1}^n (1-q^i)$ where $n=|\lambda/\mu|$ and using
Corollary~\ref{cor:skewRPP} gives
\[
\sum_{w \in \mathcal{L}(P_{\lambda/\mu})}   q^{\maj(w)} =
\prod_{i=1}^n (1-q^i) \sum_{D\in \ED(\lambda/\mu)}
q^{a'(D)}\. \prod_{u\in [\lambda]\setminus D}\frac{1}{1-q^{h(u)}}\.,
\]
Taking the limit $q\to 1$ in the equation above gives the NHLF~\eqref{eq:Naruse}.
\end{proof}

\begin{corollary}\label{cor:trace-RPP-excited}
We have:
\[
\sum_{\pi \in \RPP(\lambda/\mu)} \. q^{|\pi|}\ts t^{\tr(\pi)} \, = \,
\sum_{D\in
  \ED(\lambda/\mu)} q^{a'(D)}\ts t^{c(D)}\, \prod_{u \in \overline{D} \cap \square^{\lambda}}
\frac{1}{1-t\ts q^{h(u)}} \, \prod_{u\in \overline{D} \setminus
  \square^{\lambda}} \frac{1}{1-q^{h(u)}}\,,
\]
where \ts $\overline{D}=[\lambda]\setminus D$,
\ts $a'(D)=\sum_{u \in \Lambda(D)} h(u)$ \ts and \ts
$c(D)=|\Lambda(D)\cap \square^{\lambda}|$.
\end{corollary}

\begin{proof}
The proof follows verbatim to those of
Theorems~\ref{thm:RPP-trace},~\ref{thm:SSYT-trace} and
Corollary~\ref{cor:skewRPP}.  The details are straightforward.
\end{proof}

%----------------------------------------------------------------
\bigskip
\section{Hillman--Grassl map on skew SSYT}  \label{sec:HGSSYT}
%----------------------------------------------------------------

\nin
In this section we show that the Hillman--Grassl map is a
bijection between SSYT of skew shape and certain arrays of nonnegative
integers with support in the complement of excited diagrams and
some forced nonzero entries.
First, we
describe these arrays and state the main result.
Note that in contrast with the previous section,
the argument is not entirely bijective and requires
Theorem~\ref{thm:skewSSYT} (see also $\S$\ref{ss:compare-nhlf}).

\medskip
\subsection{Excited arrays} \label{sec:excited_arrays}

We fix the skew shape $\lambda/\mu$. Recall that for $1\leq t \leq
\ell(\lambda)-1$, $\dd_t(\mu)$ denotes the diagonal $\{(i,j) \in
\lambda/\mu \mid i-j =
\mu_t-t\}$, where $\mu_t=0$ if $\ell(\mu)<t\leq \ell(\lambda)$. Thus each row of $\mu$
is in correspondence with a diagonal $\dd_t(\mu)$.
%: $\delta_k$ contains the
%diagonal of the last element of the $k$th row of $\mu$ .
See
Figure~\ref{fig:skew_diag}: Left.

Let $A_{\mu}$ be the array of shape $\lambda$ with ones in each
diagonal $\dd_t(\mu)$ and zeros elsewhere. For $[\mu] \in \ED(\lambda/\mu)$, each active cell $u=(i,j)$ of
$[\mu]$ satisfies $(A_{\mu})_{i+1,j}=0$ and $(A_{\mu})_{i+1,j+1}=1$.

For each active cell $u$ of $[\mu]$, $\alpha_u(D_\mu)$ gives another
excited diagram in $\ED(\lambda/\mu)$. We do an analogous action:
\begin{equation} \label{eq:excited-array-move}
\beta_u:\,\,\raisebox{-6pt}{\includegraphics{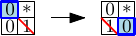}}
\end{equation}
 on $A_{\mu}$ to obtain a $0$-$1$ array associated to
 $\alpha_u(D_\mu)$. Concretely if $A$ is a $0$-$1$ array of shape
 $\lambda$ and $u=(i,j)$ is a cell such that $A_{i+1,j}=0$ and
 $A_{i+1,j+1}=1$ then $\beta_u(A)$ is the $0$-$1$ array $B$ of shape
 $\lambda$ with $B_{i+1,j+1}=0$, $B_{i+1,j}=1$ and $B_v=A_v$ for $v
 \neq \{(i+1,j),(i+1,j+1)\}$. Next, we define {\em excited arrays} by
 repeatedly applying $\beta_u(\cdot)$ on active cells $u$ starting from $A_{\mu}$.

\begin{definition}[excited arrays] \label{def:hookarray}
For an excited diagram $D$ in $\ED(\lambda/\mu)$ obtained
from $[\mu]$ by a sequence of excited moves $D=\alpha_{u_k} \circ
\alpha_{u_{k-1}} \circ \cdots \circ \alpha_{u_1}(\mu)$, then we let
$A_D = \beta_{u_k} \circ \beta_{u_{k-1}} \circ \cdots \circ
\beta_{u_1}(A_{\mu})$ provided the operations $\beta_u$ are well defined.
So each excited diagram $D$ is associated to a
$0$-$1$ array $A_D$ (see Figure~\ref{fig:skew_diag}).
\end{definition}

Next we show that the procedure for obtaining the arrays $A_D$ is well
defined; meaning that  at each
stage, the conditions to apply $\beta_u(\cdot)$ are met.

\begin{proposition} \label{prop:welldefined}
Let $A_D$ be the excited array of $D\in \ED(\lambda/\mu)$
and $u=(i,j)$ be an active cell of~$D$.  Then
$(A_D)_{i+1,j+1}=1$ and $(A_D)_{i+1,j}=0$.
\end{proposition}

\begin{proof}
We prove this by induction on the number of excited moves. If
$D=[\mu]$ and $u\in [\mu]$ is an active cell then $u=(t,\mu_t)$
is the last cell of a row of $\mu$ with $\mu_{t+1}<\mu_t$. This
implies that
$(t+1,\mu_t+1) \in \dd_t(\mu)$ and $(t+1,\mu_t)\not\in \dd_{t+1}(\mu)$ and
so $(A_\mu)_{t+1,\mu_t+1}=1$ and $(A_{\mu})_{t+1,\mu_t}=0$.

Assume the result holds for $D \in \ED(\lambda/\mu)$. If
$D'=\alpha_{(i,j)}(D)$ then $A_{D'}=\beta_{(i,j)}(A_D)$ is well
defined since $(A_D)_{i+1,j+1}=1$ and $(A_D)_{i+1,j}=0$. Let $v=(i',j')$
be an active cell of $D'$. If $v'=(i',j')$ is also an active cell
of~$D$, then the excited move $\beta_u(\cdot)$ did not alter the values
at $(i'+1,j'+1)$ and $(i'+1,j')$. In this case
$(A_{D'})_{i'+1,j'+1}=(A_D)_{i+1,j+1}=1$ and
$(A_{D'})_{i'+1,j'}=(A_{D})_{i'+1,j'}=0$. If $v'$ is not an active
square of $D$ then $u$ is one of $\{(i',j+1),(i'+1,j'),(i'-1,j'-1)\}$
(note that $u\neq (i'+1,j'+1)$ since the corresponding flagged tableau
would not be semistandard). In each of these three cases we see
that $(A_{D'})_{i'+1,j'+1}=1$ and $(A_{D'})_{i'+1,j'}=0$:
\begin{center}
\includegraphics[width=0.9\textwidth]{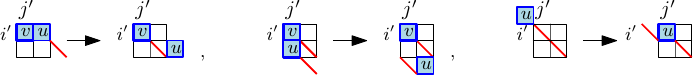}
\end{center}
This completes the proof.
\end{proof}

The support of excited arrays can be divided into {\em broken diagonals}

\begin{definition}[Broken diagonals] \label{rem1:broken_diag}
To each excited diagram $D\in
\ED(\lambda/\mu)$ we associate {\em broken diagonals} that come from
$\dd_t(\mu)$ for $1\leq t \leq \ell(\lambda)-1$, that are described as follows.
The diagram $[\mu] \in \ED(\lambda/\mu)$ is associated to
$\dd_1(\mu),\ldots,\dd_{\ell(\lambda)-1}(\mu)$. Then iteratively, if $D$
is an excited diagram with broken diagonals
$\dd_1(D),\ldots,\dd_{\ell-1}(D)$ and $D'=\alpha_{(i,j)}(D)$ then $(i+1,j+1)$
is in some $\dd_t(D)$. We let $\dd_r(D')=\dd_r(D)$ if $r\neq
t$ and $\dd_t(D')=\dd_t(D) \setminus \{(i+1,j+1)\} \cup
\{(i+1,j)\}$ (See Figure~\ref{fig:skew_diag}). Note that the broken diagonals $\dd_t(D)$ give precisely the
support of the excited arrays $A_D$.
\end{definition}

\begin{remark} \label{rem2:broken_diag}
 Each broken diagonal $\dd_t(D)$ is a sequence of diagonal
  segments from $\dd_t(\mu)$ broken by horizontal segments coming from row
  $\mu_t$. We call these segments {\em excited segments}. In particular if $(a,b)\in \dd_t(D)$ with $a,b>1$ then
  either $(a-1,b-1) \in \dd_t(D)$ or $(a-1,b-1)\in D$.
\end{remark}

\begin{remark} \label{rem:minT}
Let $T_0$ be the {\em minimal} SSYT of shape $\lambda/\mu$, i.e.~the tableau whose
with $i$-th column $(0,1,\ldots,\lambda'_i-\mu'_i)$.
We then have $\HG(T_0)=A_{\mu}$.
\end{remark}

\begin{figure}
\includegraphics[scale=1]{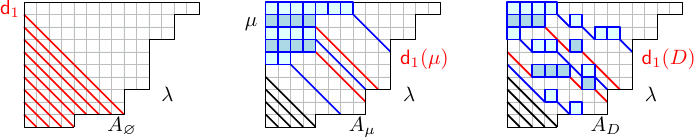}
\caption{The diagonals $\dd_1(\mu),\ldots,
  \dd_{\ell(\lambda)-1}(\mu)$, the support of $A_{\mu}$
  represented by diagonals, and the support of array $A_D$
  associated to an excited diagram~$D$.}
 \label{fig:skew_diag}
\end{figure}

%\begin{example}
%See the second row in Figure~\ref{fig1b}
%for excited arrays (in green) associated to the excited diagrams in
%$\mathcal{E}_{2221}(11)$.
%\end{example}

\begin{definition} \label{def:set_excited_arrays}
For $D\in \ED(\lambda/\mu)$, let $\mathcal{A}^*_D$ be the set of arrays $A$ of nonnegative integers
of shape $\lambda$  with support contained in
$[\lambda]\setminus D$, and nonzero entries $A_{i,j}>0$ if $(A_D)_{i,j}=1$,
where $A_D$ is $0$-$1$ excited array corresponding to $D$.
\end{definition}

We are now ready to state the main result of this section.

\begin{theorem} \label{thm:bij}
The $($restricted$)$ Hillman--Grassl map $\HG$ is a bijection:
$$\HG\.{}: \ \SSYT(\lambda/\mu) \, \longrightarrow  \bigcup_{D \in \ED(\lambda/\mu)} \. \mathcal{A}^*_D\..
$$
\end{theorem}

We postpone the proof until later in this section.  Let us first present
the applications of this result.
Note first that since $\HG(\cdot)$ is weight preserving, Theorem~\ref{thm:bij}
implies an alternative description of the statistic $a(D)=\sum_{u
  \in \overline{D}} (\lambda'_j-i)$ from \eqref{eq:skewschur} in terms of sums of
hook-lengths of the support of $A_D$ (i.e. the weight
$\omega(A_D)$).

\begin{corollary}
For a skew shape $\lambda/\mu$, we have:
% the generating function
% $s_{\lambda/\mu}(1,q,q^2,\ldots)  = \sum_{T \in \SSYT(\lambda/\mu)}
% q^{|T|}$ equals
\[
s_{\lambda/\mu}(1,q,q^2,\ldots) \, = \, \sum_{D\in \ED(\lambda/\mu)}
q^{\omega(A_D)} \. \prod_{u \in [\lambda]\setminus D} \.
\frac{1}{1-q^{h(u)}}\,.
\]
In particular for all $D\in \ED(\lambda/\mu)$ we have $a(D)= \omega(A_D)$.
\end{corollary}

\begin{example}
For $\lambda/\mu = (4^32/31)$, we have $|\ED(4^32/31)|=7$, see Figure~\ref{fig2}. By the corollary,
\begin{multline*}
 s_{4^32/31}(1,q,q^2,\ldots) \, = \, \frac{q^8}{[5]^2[4][3]^2[2]^3[1]^2} \. + \.
                                 \frac{q^9}{[6][5]^2[3]^2[2]^3[1]^2}  \. + \.
                                 \frac{q^9}{[5]^2[4]^2[3]^2[2]^2[1]^2} \. + \\
 + \frac{q^{10}}{[6][5]^2[4][3]^2[2]^2[1]^2}  \. + \.
  \frac{q^{11}}{[6][5]^2[4]^2[3][2]^2[1]^2} \. + \.
  \frac{q^{12}}{[6]^2[5]^2[4][3][2]^2[1]^2} \. + \. \frac{q^{13}}{[7][6]^2[5][4][3][2]^2[1]^2}\,,
\end{multline*}
where here and only here we use $[m]:=1-q^m$.
\end{example}

\begin{figure}
\begin{center}
\includegraphics[scale=0.8]{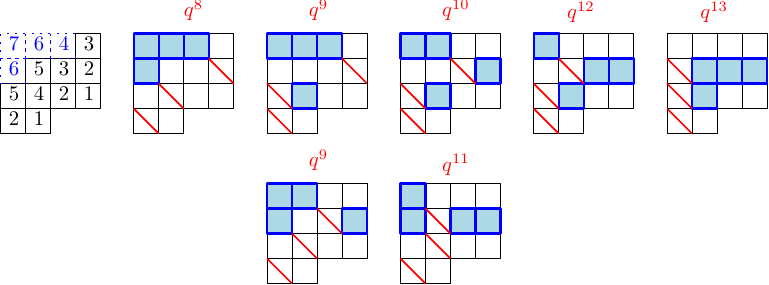}
\end{center}
\caption{The excited diagrams $D$ for $(4^32/31)$, their
  respective excited-arrays $A_D$ (the broken diagonals correspond
  to the $1$s in $A_D$) and weights $q^{\omega(A_D)}=q^{a(D)}$ where
  $\omega(A_D)$ is the sum of hook-lengths of the support of $A_D$ and
  $a(D)=\sum_{u\in \overline{D}} (\lambda'_j-i)$.}
\label{fig2}
\end{figure}

Since by Theorem~\ref{thm:skewRPP} we understand the image of the
Hillman--Grassl map on SSYT of skew shape then we are able to
give a generalization of the trace generating function
\eqref{eq:traceeq} for these SSYT.

\begin{proof}[Proof of Theorem~\ref{thm:SSYT-trace}]
By Theorem~\ref{thm:bij} a tableau $T$ has shape $\lambda/\mu$ if and only if
$A:=\HG(T)$ is in $\mathcal{A}_D^*$ for some excited diagram $D \in
\ED(\lambda/\mu)$. Thus,
\begin{equation} \label{eq1:traceSSYT}
\sum_{T \in \SSYT(\lambda/\mu)} \. q^{|T|}\ts t^{\tr(T)} \, = \, \sum_{D \in
  \ED(\lambda/\mu)} \, \sum_{T \in \HG^{-1}(\mathcal{A}_D^*)} \. q^{|T|}\ts t^{\tr(T)}\,,
\end{equation}
where for each $D\in \ED(\lambda/\mu)$ we have:
\begin{equation} \label{eq2:traceSSYT}
\sum_{T \in \HG^{-1}(\mathcal{A}_D^*)} \. q^{|T|} \, = \,
q^{\omega(A_D)} \. \prod_{u\in \overline{D}} \. \frac{1}{1-q^{h(u)}}\,.
\end{equation}
Next, by Proposition~\ref{prop:HGtrace} for
$k=0$, the trace $\tr(\pi)$ equals $|A_0|$, the sum of the entries of
$A$ in the Durfee square $\square^{\lambda}$ of $\lambda$. Therefore,
we refine \eqref{eq2:traceSSYT} to keep track of the trace of the SSYT
and obtain
\begin{equation} \label{eq3:traceSSYT}
\sum_{T \in \HG^{-1}(\mathcal{A}_D^*)} \. q^{|T|}\ts t^{\tr(T)}
\, = \, q^{\omega(A_D)}\ts t^{c(D)}\, \prod_{u \in \overline{D} \cap \square^{\lambda}}
\frac{1}{1-t\ts q^{h(u)}} \, \prod_{u\in \overline{D} \setminus
  \square^{\lambda}} \frac{1}{1-q^{h(u)}},
\end{equation}
where $c(D)=|\supp(A_D)\cap \square^{\lambda}|$ and $\omega(A_D)=a(D)$.
Combining \eqref{eq1:traceSSYT} and \eqref{eq3:traceSSYT} gives the
result.
\end{proof}

%\begin{example}
%For $\lambda/\mu = 4441/2$, there are six excited diagrams (see
%Figure~\ref{fig2}), and
%\[
%\sum_{T \in \SSYT(4441/2)} q^{|T|} =
%\frac{q^{10}[7][5] + {q^{11}}[7][3] +  {q^{12}}[7][1]+{q^{13}}[4][3] + {q^{14}}[4][1] +
%q^{15}[2][1]}{[7][6][5]^2[4]^2[3]^3[2]^2[1]^2}
%\]
%where $[m]=1-q^m$.
%\end{example}

%\begin{figure}
%\begin{center}
%\includegraphics{excited_ex2}
%\end{center}
%\caption{The excited diagrams for $\lambda/\mu=4441/2$.}
%\label{fig2}
%\end{figure}

\noindent {\bf Proof of Theorem~\ref{thm:bij}:}  First we use Theorem~\ref{thm:rpp_pleasant} to show that $\HG^{-1}(\bigcup_{U\in
  \ED(\lambda/\mu)} \mathcal{A}^*_D)$ consists of RPP of skew
shape $\lambda/\mu$ (Lemma~\ref{lemma:skewsupp}). Then we show that these RPP are
also column-strict (Lemma~\ref{lemma:colstrict}). These two results
and the fact that $\HG^{-1}$ is injective imply that
\[
\HG^{-1}: \bigcup_{U\in
  \ED(\lambda/\mu)} \mathcal{A}^*_D  \hookrightarrow \SSYT(\lambda/\mu)\..
\]
In addition, since $\HG$ is weight preserving, we have:
\begin{equation} \label{eq:qdiff}
s_{\lambda/\mu}(1,q,q^2,\ldots) - F(q)
\in \mathbb{N}[[q]]\.,
\end{equation}
where
$$F(q) \, := \, \sum_{D \in \ED(\lambda/\mu)} q^{\omega(A_D)} \prod_{u\in
  [\lambda]\setminus D} \frac{1}{1-q^{h(u)}}\,.$$

 By Theorem~\ref{thm:skewSSYT}
and the equality $a(D)=\omega(A_D)$ (Proposition~\ref{prop:samestats}),
it follows that the difference in~\eqref{eq:qdiff} is zero. Therefore, the restricted map $\HG$ is a
bijection between tableaux in $\SSYT(\lambda/\mu)$ and arrays in $\bigcup_{U\in
  \ED(\lambda/\mu)} \mathcal{A}^*_D$, as desired. \qed

\medskip \subsection{$\HG^{-1}(\mathcal{A}^*_D)$ are RPP of skew
  shape} \label{sec:excited_is_pleasant}

Given an excited diagram $D \in
\ED(\lambda/\mu)$, let $\mathcal{A}_D$ be the set of arrays of
nonnegative integers of shape $\lambda$ with support in
$[\lambda]\setminus D$. Note that the set of excited arrays $\mathcal{A}^*_D$
from Definition~\ref{def:set_excited_arrays} is contained in $\mathcal{A}_D$.
We show that the RPP in $\HG^{-1}(\mathcal{A}_D)$ have support
contained in $\lambda/\mu$ and therefore so do the RPP in $\HG^{-1}(\mathcal{A}^*_D)$.

\begin{lemma} \label{lemma:skewsupp}
For each excited diagram $D\in
\ED(\lambda/\mu)$, the reverse plane partitions in
$\HG^{-1}(\mathcal{A}^*_D)$ have support contained in $\lambda/\mu$.
\end{lemma}

\begin{proof}
By Lemma~\ref{lemma:excitedispleasant}, the support of arrays in
$\mathcal{A}_D$ are pleasant diagrams in $\PD(\lambda/\mu)$. So by
Theorem~\ref{thm:rpp_pleasant} it follows that
$\HG^{-1}(\mathcal{A}_D) \subseteq \RPP(\lambda/\mu)$. Since
$\mathcal{A}^*_D \subseteq \mathcal{A}_D$, the result follows.
\end{proof}

\medskip
\subsection{$\HG^{-1}(\mathcal{A}^*_D)$ are column strict skew RPP}

\begin{lemma} \label{lemma:colstrict}
For each excited diagram $D\in
\ED(\lambda/\mu)$, the reverse plane partitions in
$\HG^{-1}(\mathcal{A}^*_D)$ are column strict skew RPPs of shape $\lambda/\mu$.
\end{lemma}

Let $\pi$ be the reverse plane partition
$\HG^{-1}(A)$ for $A\in \mathcal{A}^*_D$ and $D\in
\ED(\lambda/\mu)$. By Lemma~\ref{lemma:skewsupp}, we know that $\pi$ has support
in the skew shape $\lambda/\mu$. We show that $\pi$ has strictly increasing
columns by  comparing any two adjacent entries from the same column of $\pi$ . Consider the two
adjacent diagonals  of $\pi$ to which the corresponding entries belong
and let $\nu^1$ and $\nu^2$ be the partitions obtained by reading
these diagonals bottom to top. There are two cases depending
on whether the diagonals end in the same column or in the same row of
$\lambda/\mu$;
\begin{compactitem}
\item[\textbf{Case 1:}] If the diagonals end in the same column,
 then  it suffices to show that $\nu^2_i < \nu^1_i$ for all $i$.
\item[\textbf{Case 2:}] If the diagonals end in the same row,
then  it suffices to show that $\nu^2_{i+1} < \nu^1_i$ for all $i$.
\end{compactitem}
\begin{center}
\includegraphics{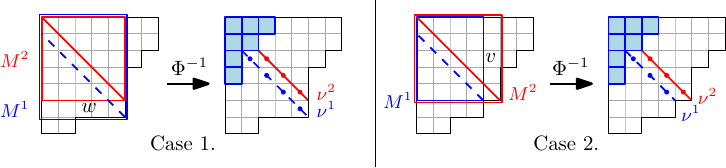}
\end{center}

Before we treat these cases we prove the following Lemma needed for both.

\begin{lemma} \label{lem:firstcolP}
Let $M$ be a rectangular array coming from $A\in \mathcal{A}^*_D$ with
NW corner $(1,1)$.
Then the first column of $P=I(\RSK(\hf{M}))$ is \ts $(1,\ldots,h)$,
% $$\begin{ytableau} 1&2&\none&
%  \none[\cdots] &\none & h \end{ytableau}\ ,$$
where~$h$ is the height of~$P$.
\end{lemma}

\begin{proof}
We will use the symmetry of the RSK correspondence. Recall that $\RSK(N)=(P,Q)$
for some rectangular array $N$ then $\RSK(N^{T})=(Q,P)$ so that~$P$
is the recording tableaux by doing the RSK on $N$ row by row, bottom
to top. Thus the first column of $P$ gives the row numbers of
$N$ where the height of the insertion tableaux increased by
one.

Let $R$ be the rectangular shape of $M$. By Greene's theorem,
$h$ is equal to the length of the longest decreasing
subsequence in~$M$. By Lemma~\ref{lemma:excitedispleasant}, $h$ is
at most the length of longest diagonal of~$R/\mu$.

Note that $M$ contains a broken diagonal of
length at least $h-1$ since either the longest diagonal of length $h$ in
$R/\mu$ ends in a vertical step of $\mu$, in which case $M$ has a
broken diagonal of the same length, or the longest diagonal ends in a
horizontal step of $\mu$ in which case $M$ has a broken diagonal of
length $h-1$.

Let $\dd$ be such a broken diagonal. Since a broken diagonal is a
decreasing subsequence that spans consecutive rows, then $\dd$ spans
the lower $h-1$ rows of $M$. This guarantees that the first column of $P$
is $1,2,\ldots,h-1,c$, where $c\geq h$ is the row where we first get a
decreasing subsequence of length $h$.

Assume there is a longest decreasing subsequence $\mathfrak{d}$ of
length $h$ whose first cell $x=(i_1,j_1)$ is in a row $c=i_1>h$
(counting rows bottom to top), and take both $i_1$ and $j_1$ to be
minimal.

Either $x$ is inside or outside of $[\mu]$. If $x$ is outside then there is a
diagonal that ends in row $i_1-1$ to the left of $x$, which results in
a broken diagonal of length $i_1-1 \geq h$ in the excited
diagram. Hence, there is a decreasing subsequence of length $h$ starting
at a lower row than the row $i_1$ of  $x$, leading to a  contradiction. See Figure~\ref{fig:pf1stcol}
: Left.

When $x$ is inside of $[\mu]$ then there is an excited cell below $x$ in the same diagonal.
There must be a broken diagonal $\dd'$  that
reaches at least row $i_1-1$ below or to the left of $x$. At row $i_1$, the sequence
$\mathfrak{d}$ is above $\dd'$ and the last entry of
$\mathfrak{d}$ is below $\dd'$, as otherwise $\mathfrak{d}$ would
be shorter than $\dd'$. Thus the sequence $\mathfrak{d}$
and the broken diagonal $\dd'$ cross. Consider the first crossing
tracing top down. Let $a$ be the last cell of $\mathfrak{d}$ before
this crossing and let $b$ be the cell of $\mathfrak{d}$ on or after the
crossing. Note that below $a$ in the same column there is either a
nonzero from $\dd'$ or a zero from the excited horizontal
segment of $\dd'$. In either case, $a$ is higher than the lowest cell of
$\dd'$ to the left of $b$. Define $\mathfrak{d}'$ to be the
sequence consisting of the segment of $\dd'$ from row $i_1-1$ up until
the crossing followed by the segment of $\mathfrak{d}$ from cell $b$
onwards (see Figure~\ref{fig:pf1stcol}: Right). Note that $\mathfrak{d}'$ is a decreasing sequence of
$R$ that starts at row $i_1-1$ and column $\leq j_1$ and has length $h$
since $\mathfrak{d}$ includes a nonzero element from the row below the
row $a$.  This contradicts the
minimality of $x$.

\begin{figure}
\begin{center}
\includegraphics{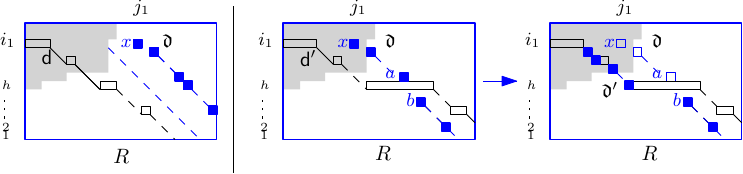}
\caption{Two cases to consider in the proof of
  Lemma~\ref{lem:firstcolP} depending on whether cell $x=(i_1,j_1)$ is
  outside or inside of $[\mu]$.}
\label{fig:pf1stcol}
\end{center}
\end{figure}

In summary, we conclude that $c=h$, and the first column of $P$ is $(1,\ldots,h)$, as desired.
This finishes the proof.  \end{proof}

\smallskip

\nin
{\bf Column strictness in Case 1.} \.
By Corollary~\ref{cor:HGvsRSK} we have $\nu^1=\shape(P^1)$ and
$\nu^2=\shape(P^2)$ where $P^1=I(\RSK(M^1))$, $P^2=I(\RSK(M^2))$, and the rectangular array $M^1=\vf{A_{t}}$ is
obtained from the rectangular array  $M^2=\vf{A_{t+1}}$ by adding a row $w$ at the
end. Thus $\nu^1$ is the shape of the insertion tableau obtained by row
inserting $w$ (from left to right) in the insertion tableau $I(\RSK(M^2))$ of shape~$\nu^2$.

\begin{proposition}
In Case 1 we have $\nu^2_i < \nu^1_i$ for $1\leq i \leq \min\{\ell(\nu^1),\ell(\nu^2)\}$.
\end{proposition}

\begin{proof}

Let $P=I(\RSK(M^1))$ and $Q=R(\RSK(M^1))$ (in this case, $M^1$ is being read
top to bottom,  left to right, i.e. row by row starting from the right,
from the original array $A_{t}$ before the flip). Let $m$ be the height of $M^1$. The strict
inequality is equivalent to the fact that the insertion of the last
row in $P$ results in an extension of every row, i.e. every row of the
recording tableau $Q$ has at least one entry equal to~$m$.
We will prove the last statement. Note that by the symmetry of the
RSK correspondence, we have $Q$ is the insertion tableaux
for $A_t$ when read column by column from right to left.

\smallskip

\nin
{\bf Claim:} \ts \emph{Let $h$ be the height of $Q$, i.e. the longest decreasing
subsequence of $M^1$. Then row $i$ of $Q$ contains at least one entry
from each of $m,\ldots,m-h+i$. }

\smallskip

Note that $h$ is equal to the length of the longest broken diagonal or one more than that.
We prove the claim by induction on the number of columns in $M^1$,
i.e. in $A_t$. Let $M^1 = [u^1, u^2, \ldots , u^r]$, where $u^i$ is
its $i$-th column. In terms of the excited array $A_t$, we have  that $u^r$ is the first column of $A_t$ and $u^1$ -- the last. Suppose that the claim is true for $A_t$ restricted to its first $r-1$ columns, which is still an excited array by definition, and let $u=u^1$ be its last column. Let $Q$ be the insertion tableaux of $[u^2,\ldots,u^r]$ read column by column, i.e. $Q = u^2 \leftarrow u^3 \leftarrow \cdots$, where $\leftarrow$ indicates the insertion of the corresponding sequence. Then let $Q'$ be the insertion tableaux corresponding to $A_t$, so $Q' = u \leftarrow  \reading(Q)$ by Knuth equivalence. The reading word is obtained from $Q$ by reading it row by row from the bottom to the top, each row read left to right.

First, suppose that $u$ does not increase the length of the longest decreasing subsequence, so $Q'$ has also height $h$. Let $a \in [m-j+i,m]$ be a number present in row $i$ of $Q$. When it is inserted in $Q'$ it will first be added to row 1, where there could be other entries equal to $a$ already present. The first such entry $a$ will be bumped by something $\leq a-1$ coming from inserting row $i-1$ of $Q$ into $Q'$. This had to happen in $Q$ since $a$ reached row $i$. From then on the same numbers will bump each other as in the original insertion which created $Q$. Thus an entry equal to $a$ will reach row $i$ after the $i-1$ bumps. Since the height of $Q'$ is unchanged, the claim holds as it pertains only to the original entries $a$ from $Q$ which again occupy the corresponding rows.

Next, suppose that $u$ increases the length of the longest decreasing
subsequence to $h+1$.  Then the longest broken diagonal in
$A_t$ has length at least~$h$.  Also, column~$u$ must have an
element equal to $m$, i.e.~a nonzero entry in $A_t$'s lower right
corner. Moreover, we claim that the longest decreasing subsequence has
to occupy the consecutive rows of $A_t$ from $m-h$ to $m$, and thus the
longest decreasing  subsequences in $u,\reading(Q)$  are
$m,m-1,\ldots,m-h$. This is shown within the proof of
Lemma~\ref{lem:firstcolP}. From there on, in $u \leftarrow
\reading(Q)$ we have element $m$ from $u$ bumped by something $ \leq
m-1$ from the last row of $Q$. Afterwards, the bumps happen similarly
to the previous case and the numbers from $Q$ reach their corresponding
rows, so the $m$ from $u$ reaches eventually one row below, i.e. row
$h+1$. The entry $m-h$ from the longest decreasing sequence is
inserted from the first row of $Q$   and  is, therefore, in row 1 of~$Q'$, so by iteration~$Q'$ has the desired structure. This ends the proof of the claim and thus the Proposition.
 \end{proof}

\smallskip

\nin
{\bf Column strictness in Case~2.} \.
By Corollary~\ref{cor:HGvsRSK} we have  $\nu^1=\shape(P^1)$ where $P^1=I(\RSK({M^1}))$ and
$\nu^2=\shape(P^2)$ where $P^2=I(\RSK({M^2}))$ and the rectangular array $M^2=\hf{A_{t+1}}$ is
obtained from the rectangular array $M^1=\hf{A_t}$ by adding a column $v$ at the
end (we read column by column SW to NE). Thus $\nu^2$ is the shape of the insertion tableau obtained by row
inserting $v$ (from top to bottom) in the insertion tableau $I(\RSK(M^1))$ of shape $\nu^1$.

\begin{proposition} \label{prop:case2}
For Case 2 we have $\nu^2_{i+1} < \nu^1_{i}$ for $1\leq i \leq \min\{\ell(\nu^1),\ell(\nu^2)-1\}$.
\end{proposition}

We prove a stronger statement that requires some notation. Let $P$ be the insertion tableau of shape $\nu$ where $M=\hf{B}$ for
some rectangular array $B$ of $A \in \mathcal{A}^*_D$  with NW corner
$(1,1)$. for a positive integer $k$, let
$P_i(k)$ be the number of entries in row $i$ of $P$ which are
$\leq k$.

\begin{lemma} \label{lemma:insertionSW}
With $P$ and $P_i(k)$ as defined above, for $k>1$ we have
\begin{compactitem}
\item[(i)] If $P_i(k)>0$ then $P_i(k) < P_{i-1}(k-1)$,
\item[(ii)] If $k$ is in row $i$ of $P$ where $k>i$ then $P_i(k-1)>0$.
\end{compactitem}
\end{lemma}

\begin{proof}[Proof of Proposition~\ref{prop:case2}]
We first show that Lemma~\ref{lemma:insertionSW} implies that the
{\em insertion path} of the RSK map of $M$ moves {\em strictly} to the left.
To see this, let~$P$ be
the resulting tableaux obtained at some stage of the insertion when
$j$ is inserted in row $1$ and bumps $j_1>j$ to row $2$.
Then $j$ is inserted at position
$P_1(j_1-1)$ in row 1 and $j_1$ is inserted at position
$P_2(j_1)>0$ in row 2. By  Condition~(i),
$P_2(j_1)<P_1(j_1-1)$. Iterating this argument as elements get bumped
in lower rows implies the claim.

Next, note that a bumped element at position $\nu^1_2+1$ from row 1 of $P^1$ cannot be
added to row 2 as otherwise the insertion path would move strictly
down, violating Condition (i). Thus the only elements from row 1 of $P^1$ that can be added to
row 2 in $P^2$ are those in
positions $>\nu^1_2+1$. And so there are no more than $\nu^1_1-\nu^1_2-1$  such elements
implying that $\nu^2_2\leq \nu^1_1-1$. Iterating
this argument in the other rows implies the result.
\end{proof}

\begin{proof}[Proof of Lemma~\ref{lemma:insertionSW}]
Note that Condition~(ii) for $k=i+1$ follows by
Lemma~\ref{lem:firstcolP}. Note that the statement of the lemma holds for any step of the insertion, since it applies for $P$ as a recording tableaux. Since $P_i(k)$ are increasing for $k$ with
$i$ fixed then Condition~(ii) holds. We claim that after each single insertion
of~$\RSK$, Condition~(i) still holds. We prove this when
inserting an element $j$ in row~$r$. Iterating this argument as elements
get bumped in lower rows implies the claim.

Assume $P$ verifies Condition~(i) and we insert $j$ in row
$r$ of $P$ to obtain a tableaux $P'$ of shape $\nu'$. By
Lemma~\ref{lem:firstcolP} we have $j\geq r$.
If $j$ is added to
the end of the row then  Condition~(i) still
holds for $P'$ since $P'_r(j)>P_r(j)$.
If $j$ bumps $j_1$ in row $r$ then $j_1>j$ and
\begin{equation} \label{eq:newNu1}
P'_r(j) \. = \. P'_r(j_1-1)= P_r(j_1-1)+1\., \qquad
P'_r(j_1)\. = \. P_r(j_1)
\end{equation}
and all
other $P'_r(i)=P_r(i)$ remain the same.

Next, we insert $j_1$ in row $r+1$ of~$P$. Regardless of whether $j_1$
is added to the end of the row or bumps another element to row~$r+2$,
we have
%\begin{} \label{eq:newNu2}
$$
P'_{r+1}(j) \. = \. P_{r+1}(j)\., \qquad P'_{r+1}(j_1) \. = \. P_{r+1}(j_1)+1\.,
$$%\end{align}
and all other $P'_{r+1}(i)=P_{r+1}(i)$ remain the same. Since $P'_r(b) \geq P_r(b)$
for all~$b$, we need to verify Condition~(i) only when $P'_{r+1}$ increased
with respect to~$P_{r+1}$.

%
%By induction we have %If $j$ was present in row $r$ in $P$ then by \eqref{eq:newNu1} and \eqref{eq:newNu2}
%\[
%P'_{r+1}(j)= P_{r+1}(j) < P_r(j-1) = P'_r(j-1) %\leq P_{r+1}(j_1-1) \leq P_r(j_1-1) = P_r(j)-1.
%\]

%On the other hand, if $j$ was not present in row $r$ in $P$ then by
%Condition~(i) $P'_{r+1}(j)=P_{r+1}(j)<P_r(j-1)=P'_r(j-1)$, where
%the equality follows by \eqref{eq:newNu1}.
%In both cases we get that $P'_{r+1}(j)<P'_r(j-1)$.

By Lemma~\ref{lem:firstcolP}, we have either row $r+1$ of $P$ is nonempty
and thus $P_{r+1}(j_1)>0$, or else we must have $j_1=r+1$. In the first case
Condition~(i) applies to $P$ and we have $P_{r+1}(j_1)<P_r(j_1-1)$.
By~\eqref{eq:newNu1}, we have
\[
P'_{r+1}(j_1) = P_{r+1}(j_1)+1<P_r(j_1-1)+1 = P'_r(j_1-1).
\]
Finally, suppose $j_1 = r+1$.  Since $j \geq r$, we then have $j=r$, and~$r$ must have
been present in row $r$ in $P$ by Lemma~\ref{lem:firstcolP}. Thus $P_r(r) \geq 1$  and
$$
P'_r(r) \. \geq \. 2 \. >  \. P'_{r+1} \. = \. 1\..
$$
%
%Since $P'_{r+1}(j)<P'_r(j-1)$, $P'_{r+1}(j_1) < P'_r(j_1-1)$ and
%other $P'_{r+1}(i)=P_{r+1}(i)$ then Condition~(i) is verified.
%
%
%
Therefore, Condition~(i) is verified for rows~$r$ and $r+1$ of~$P'$,
as desired.
\end{proof}

\medskip \subsection{Equality between $a(D)$ and $\omega(A_D)$}

\begin{proposition} \label{prop:samestats}
For all excited diagrams $D \in  \ED(\lambda/\mu)$,
$a(D)= \sum_{(i,j) \in \overline{D}} (\lambda'_j-i)$ equals $\omega(A_D)$.
\end{proposition}

First, we show that for the Young diagram of $\mu$ both statistics
$a(\cdot)$ and $\omega(\cdot)$ agree.

\begin{lemma}
For a Young diagram $[\mu] \in \ED(\lambda/\mu)$ we have $a([\mu]) = \omega(A_{\mu})$.
\end{lemma}

\begin{proof}
We proceed by induction on
$|\mu|$ with $\lambda$ fixed. When $\mu=\varnothing$ we have both
$$
a(D_{\varnothing}) \, = \, \sum_{(i,j) \in \lambda} (\lambda'_j-i) \, = \,
\sum_i \binom{\lambda'_i}{2}=\sum_{i}(i-1) \lambda_i = b(\lambda)\..
$$
Now, either directly or by Remark~\ref{rem:minT} for $\mu=\varnothing$,
\[
\omega(A_{\varnothing})=\sum_{(i,j)\in \lambda, i>j} h(i,j) = b(\lambda).
\]
Let $\nu$ be obtained from $\mu$ by adding a cell at position $(a,b)$.
Then
\[
a([\mu]) - a(D_{\nu}) \. = \. \lambda'_b -a\ts.
\]
Next, the array $A_{\nu}$ is obtained from $A_{\mu}$ by moving
the ones in diagonal $\dd_b=\{(i,j) \mid i-j = \mu_b-b\}$ to
diagonal $\dd'_b = \{(i,j) \mid i-j = \mu_b+1-b\}$ and leaving the
rest unchanged. Thus
\begin{equation} \label{eq:diagdiff}
\omega(A_{\mu}) - \omega(A_{\nu}) \. = \. \sum_{u \in \dd_b} h(u) -
\sum_{u \in \dd'_b} h(u)\ts.
\end{equation}
Since $h(i,j) = \lambda_i+\lambda'_j -i-j+1$, then $h(i,j)-h(i,j+1)$
cancels $\lambda_i-i+1$ and $h(i,j)-h(i+1,j)$ cancels the terms
$\lambda'_j-j$.  So by doing horizontal
and vertical cancellations on diagonals $\dd_k$ and $\dd'_k$ in \eqref{eq:diagdiff}
(see Figure~\ref{fig:same_stats}, Left)  we
conclude that either
\[
\sum_{u \in \dd_b} h(u) -
\sum_{u \in \dd'_b} h(u) = \lambda'_b-b - (\lambda'_c-c)
\]
if the diagonals $\dd'_b$ and $\dd_b$ have the same length,  or
\[
\sum_{u \in \dd_b} h(u) -
\sum_{u \in \dd'_b} h(u) = \lambda'_b-b + (\lambda_r-r+1).
\]
otherwise. In both these cases $\lambda'_c-c+b$ and $r-\lambda_r+b-1$ are
equal to~$a$. Thus,
\[
\omega(A_{\mu}) - \omega(A_{\nu})  = \lambda'_b-a = a([\mu]) - a(D_{\nu}).
\]
Then by induction it follows that $\omega(A_{\nu})=a(D_{\nu})$.
\end{proof}

\begin{lemma}
Let $D'\in \ED(\lambda/\mu)$ be obtained from $D \in
\ED(\lambda/\mu)$ with one excited move.
Then $a(D')-a(D) = \omega(A_{D'})-\omega(A_D)$.
\end{lemma}

\begin{proof}
Suppose $D'$ is obtained from $D$ by replacing $(i,j)$ by $(i+1,j+1)$ then
\[
a(D')-a(D) = \lambda'_j-i- (\lambda'_{j+1}-i-1)=\lambda'_j-\lambda'_{j+1}+1,
\]
and since $h_{(s,t)} = \lambda_s -s +\lambda'_t-t+1$ then
\[
\omega(A_{D'})-\omega(A_D) = h_{(i+1,j)}-h_{(i+1,j+1)}=\lambda'_j-\lambda'_{j+1}+1.
\]
We illustrate these differences in Figure~\ref{fig:same_stats}: Right.
\end{proof}

\begin{figure}
\begin{center}
\includegraphics{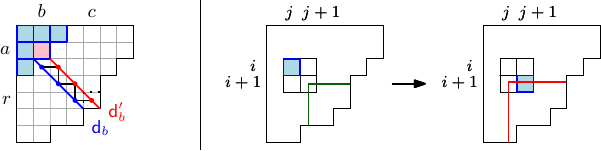}
\caption{The equality of statistics $a(D)$ and $\omega(A_D)$.}
\label{fig:same_stats}
\end{center}
\end{figure}

\bigskip

\section{Skew SSYT with bounded parts} \label{sec:bounded_parts}

Here we consider the generating function of skew SSYT with entries in
$[M]$. The analogous question for straight-shape SSYTs is answered by
Stanley's elegant hook-content formula \cite[\S 7.21]{EC2}.

\begin{theorem}[hook-content formula \cite{St71}]
\[
s_{\lambda}(1,q,q^2,\ldots,q^{M-1}) = q^{b(\lambda)} \prod_{u \in
  [\lambda]} \frac{1-q^{M+c(u)}}{1-q^{h(u)}},
\]
where $c(u)=j-i$ is the {\em content} of the square $u=(i,j)$.
\end{theorem}

In this section we discuss whether there is a hook-content formula for
skew shapes in terms of excited diagrams. We are able to write such
formulas for {\em border strips} but our approach does not extend to
general shapes.  We start by considering the
case of the inverted hook $\lambda/\mu =k^d/(k-1)^{d-1}$, then we
look at the case of {\em border strips} and we end by briefly
discussing the case of general skew shapes.

\subsection{Inverted hooks}

We start studying the border strip $\lambda/\mu = (k^d)/(k-1)^{d-1}$; an inverted hook.
From Example~\ref{ex:skewhook} the complements of excited diagrams of
this shape correspond to lattice paths $\gamma$ from $(d,1)$ to
$(1,k)$.

\begin{proposition} \label{prop:hookcontrect}
\begin{equation} \label{eq:hookcontentrect}
s_{\lambda/\mu}(1,q,\ldots,q^{M-1}) \,= \,\sum_{\ga:(d,1) \to (1,k)}
q^{a(\gamma)} \cdot h_{M-d+1}(1,q^{h(u_1)}, q^{h(u_2)},\ldots, q^{h(u_{k+d-1})})\ts,
\end{equation}
where $u_1,u_2,\ldots,u_{k+d-1}$ are the cells in the path $\ga$.
\end{proposition}

\begin{proof}
By Theorem~\ref{thm:bij}, the image of  a SSYT $T$ of this shape
via Hillman--Grassl is an array $A$ with support on a lattice path
$\ga:(d,1)\to (1,k)$  (i.e. the complement
o an excited diagram) with certain forced nonzero entries. These nonzero
entries are exactly on cells of vertical steps, including outer
corners but not inner corners and excluding $(1,k)$ (see Example~\ref{ex:hkct}).

The maximal entry in $A$ is in the cell $(d,1)$ and the maximal entry
in $T$ is in the cell $(d,k)$. We claim that the latter entry is the
sum of all the entries in the initial array. To see this note that in the steps of the
inverse Hillman--Grassl map $A\mapsto T$, every strip of $1$s added to
the RPP of support in $\lambda/\mu$ starts from a cell in row $d$ and passes through cell $(d,k)$.

Let $\gamma=(u_1,\ldots, u_{k+d-1})$ be a lattice path from $(d,1)$ to
$(1,k)$ along the squares with (top) corners at positions
$(\gamma^1_i,\gamma^2_i)$, so that the cells of the path  are
$u_1=(d,1), u_2=(d-1,1)\ldots (\gamma^1_1,\gamma^1_1)
(\gamma^1_1+1,\gamma^2_1) \ldots (\gamma^1_2,\gamma^2_1)
\ldots... $. The designated nonzero cells on $\gamma$ are the ones
located below these corners:  $(d,1)\ldots (\gamma^1_1-1,1)$ etc. We
notice that we have exactly $d-1$ such entries. If the values in $A$
of the entries in the path are $\alpha_1,\ldots,\alpha_{k+d-1}$, the
maximal entry in $T$ will be $\alpha_1+\cdots+\alpha_{k+d-1}$. The
total weight of the resulting SSYT is then $\prod_{u_i \in \gamma} q^{
  h(u_i)\alpha_i}$. The total contribution of the path $\ga$ over all possible such values $\alpha$ is then
\begin{align*}\label{eq:inverted_hook_paths}
\sum_{\alpha_1+\cdots+\alpha_{k+d-1} \leq M} \prod_{u_i \in \gamma}
  q^{ h(u_i)\alpha_i} \, & = \, \left(\prod_{u \in \ga, \text{ nonzero}} q^{h(u)}\right)  \times h_{M-d+1}(1,q^{h(u_1)}, q^{h(u_2)},\ldots, q^{h(u_{k+d-1})}) \\
&= \, q^{a(\gamma)} \cdot h_{M-d+1}(1,q^{h(u_1)}, q^{h(u_2)},\ldots, q^{h(u_{k+d-1})})\ts,
\end{align*}
as desired.
\end{proof}

\begin{example} \label{ex:hkct}
For the reverse hook shape $(3^3/2^2)$, the six paths (complements of excited
diagrams) with their corresponding nonzero entries of the arrays are
the following:
\begin{center}
\includegraphics{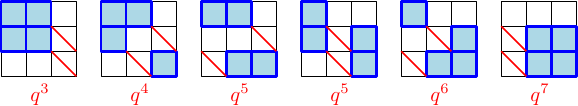}.
\end{center}
Thus, in this case  \eqref{eq:hookcontentrect} gives
\[
s_{3^3/2^2}(1,q,\ldots,q^{M-1}) \, = \, q^3
h_{M-2}(1,q^3,q^2,q^1,q^2,q^3) \. + \. \cdots \. + \. q^7h_{M-2}(1,q^3,q^4,q^5,q^4,q^3)\ts.
\]
Note that in contrast with the principal specialization of $h_k$, the
specializations in \eqref{eq:hookcontentrect} do not necessarily have nice
product formulas. For instance, when $M=3$ the first term in the RHS
above gives
\[
h_1(1,q^3,q^2,q^1,q^2,q^3) \. = \. 1+q^1+2q^2+2q^3 \. = \. (q+1)(2q^2+1)\ts.
\]
\end{example}

\begin{remark}
When we evaluate $q=1$ in \eqref{eq:hookcontentrect}, the hook lengths
involved in the evaluation of the complete symmetric function become
$1$ and so each path $\gamma$ contributes $h_{M-b+1}(1^{a+b}) =
\binom{M+a}{a+b-1}$. Summing this equal contribution over all paths gives
\[
s_{\lambda/\mu}(1^M) \, = \,\sum_{\ga:(d,1)\to(1,k)} \binom{M+k}{k+d-1}
\, = \,\binom{k+d-2}{k-1} \binom{M+k}{k+d-1}\ts.
\]
Since $s_{\lambda/\mu} = s_{k1^{d-1}}$, this is precisely what the
hook-content formula gives for $s_{k1^{d-1}}(1^M)$.
\end{remark}

\subsection{Border strips}

A border strip is a (connected) skew shape $\lambda/\mu$ containing no
$2\times 2$ box.  The inverted hook is an example of a border
strip. Similarly to inverted hooks, complements of excited diagrams of  border
strips correspond to lattice paths $\gamma$ from $(\lambda'_1,1)$ to
$(1,\lambda_1)$ that stay inside $\lambda$.

To state the result, we nee some notation. Let $\lambda/\mu$ be a border strip with corners (these time we
consider the outer corners of $\lambda$) at positions $(x_i,y_i)$
(starting from the bottom left) and divide the diagram $\lambda$ with
the lines $x=x_i$ and $y=y_j$ into rectangular regions $R_{ij}$. A
lattice path $\ga:(\lambda'_1,1) \to (1,\lambda_1)$ inside $\lambda$
may intersect some of these rectangles. Let
$\gamma=\gamma^1,\gamma^2,\ldots$ be the subpaths of $\gamma$, where each $\gamma^\ell$ belongs to a unique
rectangle $R_{i_\ell j_\ell}$. We denote by $g^{\ell}$ a sequence of
nonnegative integers in the cells of $\gamma^{\ell}$, and by
$|g^{\ell}|$ the sum of these entries.

\begin{proposition}
For a border strip $\lambda/\mu$ we have that
\begin{equation} \label{eq:hookcontent_strip}
s_{\lambda/\mu}(1,q,\ldots,q^{M-1}) = \sum_{\gamma: (\lambda'_1,1) \to
(1,\lambda_1), \gamma \subseteq [\lambda]}  q^{a(\gamma)} \sum_{  g^1, g^2, \ldots: \\ \sum_{\ell: i_\ell \geq i, j_\ell \leq i} |g^\ell| + b_\ell \leq M} q^{\sum_{u \in \gamma} g_u h(u)}.
\end{equation}
\end{proposition}

\begin{proof}
Let $T$ be a SSYT of shape $\lambda/\mu$
with entries $\leq M$ and $A=\HG(T)$. By Theorem~\ref{thm:bij}, the
support of $A$
is on a path $\gamma:(\lambda'_1,1) \to (1,\lambda_1)$ inside
of $\lambda$.

By the analogue of Greene's theorem
for $\HG$ (Theorem~\ref{thm:HGdiag} (i)), the maximal entry in $T$ in
each rectangle is the sum of the entries in $A$ within that rectangle,
since the nonzero entries lie on $\gamma$ and so form a single
increasing sequence. Moreover, the forced nonzero entries are on the
vertical steps of~$\gamma$. As in the case of inverted hooks, the
bound $M$ is again involved in the total sum over the path segments in
each rectangle. However, in a border strip the rectangles overlap, and
so would the sums over the path elements.

We divide $\gamma=\gamma^1,\gamma^2,\ldots$ into subpaths, where each $\gamma^\ell$ belongs to a unique
rectangle $R_{i_\ell j_\ell}$. We must have that by monotonicity of $\ga$ $i_1
\leq i_2 \leq \cdots$ and $j_1 \leq j_2 \leq \cdots$, and by
connectivity of $\ga$ that $i_\ell \geq i_{\ell+1} -1$ and $j_\ell \geq
j_{\ell+1}-1$. The entries in $A$ along $\gamma^{\ell}$ are the
sequence $g^\ell$, with sum $M_\ell + b_\ell$, for some
$M_\ell$. By the properties of the Hilman--Grassl bijection, we need to
have forced nonzero elements on the vertical steps of $\gamma$. We can
subtract $1$ from them and consider nonnegative elements summing up to
$M_\ell$. Again by the properties of the bijection, each rectangle
$R_{ii}$ in $A$ has to contain a longest increasing subsequence of
total sum at most $M$ in order to have the entry in the corner of $T$
to be at most $M$. In this case there is only one longest increasing
subsequence in $A$, which is the path $\ga$ itself.
Thus, we have
$$\sum_{\ell: \. i_\ell \geq i \geq j_\ell} \. (M_\ell + b_\ell) \,\leq \, M.
$$
We conclude:
$$s_{\lambda/\mu}(1,q,\ldots,q^{M-1}) \,= \,
\sum_{\gamma} \. q^{\sum_{u \in \gamma, \text{ vertical step }} h(u)} \,\.
\sum \. q^{\sum_{u \in \gamma} \ts g_u \ts h(u)}\ts,
$$
where the last sum is over  all \ts $g^1, g^2, \ldots$ \ts s.t.
$\sum_{\ell: \. i_\ell \geq i \ge j_\ell } |g^\ell| + b_\ell \. \leq \. M$.
Now observe that the sum of the hooks of the forced nonzero entries in $\gamma$ is~$a(\gamma)$,
which implies the result.
%This sum cannot be simplified further in general.
\end{proof}

\begin{remark}
The sums over the sequences $g^i$ in formula~\eqref{eq:hookcontent_strip} cannot be simplified any further, since the restrictions are not over independent pieces. However, one can think of the inequalities as a simple linear program with coefficients 0 or 1, and the entries in the sequence $g^1,\ldots,g^\ell$ as integral points in a polytope defined by these inequalities.
\end{remark}

\begin{corollary}
\[
s_{\lambda/\mu}(1^M) \,= \, \sum_{\gamma} \sum_{M_1,\ldots: \. \sum_{\ell:
    i_\ell \geq i \geq j_\ell } M_\ell + b_\ell \leq M} \. \prod_{\ell}
\binom{M_\ell + a_\ell+b_\ell-2}{M_\ell}\ts.
\]
\end{corollary}

\begin{proof}
We evaluate $q=1$ in \eqref{eq:hookcontent_strip}. If the sum of
entries in $g^\ell$ is $M_\ell$, and the path $\gamma^\ell$ has length
$a_\ell+b_\ell-1$, we have that the number of ways of choosing such
entries is $\binom{M_\ell + a_\ell+b_\ell-2}{M_\ell}$, and so the
result follows.
\end{proof}

\subsection{General skew shapes}
 In the general case, complements of excited diagrams correspond to tuples of
non-intersecting lattice paths (see proof of Lemma~\ref{lem:pleasant_is_subset_of_excited}).
In \cite{MPP2} we use a non-intersecting lattice path approach to
\emph{upgrade} the NHLF and Theorem~\ref{thm:skewSSYT} from border strips to
general skew shapes. However, this approach does not apply for SSYT of
bounded parts. This is because Proposition~\ref{eq:hookcontent_strip} shows that the bound $M$ on the entries of a SSYT $T$ of border
strip shape is encoded via Hillman--Grassl as restricted sums on an
array with support on lattice path $\ga$. Two intersecting paths with restricted sums of elements can be divided into two other intersecting paths with different total sums of elements, which may not satisfy the same restriction. In other words, the usual involution on intersections, that cancels the intersecting paths contribution from the Lindstr\"om--Gessel--Viennot determinant, cannot be applied here as we cannot restrict to the same subset of paths.

\bigskip

\section{Other formulas for the number of standard Young tableaux}\label{sec:compare}

\nin
In this section we give a quick review of several competing formulas for computing $f^{\la/\mu}$.

\subsection{The Jacobi--Trudi identity} \label{ss:compare-jt}
This classical formula (see e.g.~\cite[\S 7.16]{EC2}), allowing an efficient computation
of these numbers.  It generalizes to all Schur functions and thus gives
a natural $q$-analogue for SSYT.  On the negative side, this formula is not
\emph{positive}, nor does it give a $q$-analogue for RPP.

\subsection{The Littlewood--Richardson coefficients}\label{ss:compare-lr}
Equally celebrated is the positive (subtraction-free) formula
\[
f^{\lambda/\mu} \, = \, \sum_{\nu\vdash |\la/\mu|} \, c_{\mu,\nu}^{\lambda} \. f^{\nu}\.,
\]
where $c_{\mu,\nu}^{\lambda}$ are the {\em Littlewood--Richardson {\rm (LR-)}
coefficients}.
% To make this formula seem closer to NHLF, we can write this as
% $$
% f^{\lambda/\mu} \, = \, |\la/\mu|! \.
% \sum_{A \in \LR(\la/\mu,\ast)} \. \prod_{u \in [\text{cont}(A)}] \. \frac{1}{h(u)}\.,
% $$
% where $\LR(\ast,\ast)$ is the set of LR-tableaux and cont$(A)$ is
% the content partition of the LR-tableaux~$A$.
This formula has a natural $q$-analogue for SSYT, but not for RPP.
When LR-coefficients are defined appropriately, this $q$-analogue
does have a bijective proof by a combination of the
Hillman--Grassl bijection and the jeu-de-taquin map;
we omit the details (cf.~\cite{White}).

On the negative side, the LR-coefficients are notoriously hard to compute
both theoretically and practically (see~\cite{Nar}), which makes this
formula difficult to use in many applications.

\subsection{The Okounkov--Olshanski formula} \label{ss:compare-oo}
The following curious formula is of somewhat different nature.
It is also positive, which might not be immediately obvious.

Denote by $\RT(\mu,\ell)$ the set of {\em reverse semistandard tableaux}~$T$
of shape~$\mu$, which are arrays of positive integers of shape~$\mu$,
weakly decreasing in the rows and strictly decreasing in the columns,
and with entries between $1$ and~$\ell$.  The \emph{Okounkov--Olshanski
formula} \eqref{eq:OO} given in~\cite{OO} states:
% TODO label OOF
\begin{equation*} \label{eq:OO} \tag{OOF}
f^{\lambda/\mu} \, = \, \frac{|\lambda/\mu|!}{\prod_{u\in [\lambda]} h(u)} \,
\sum_{T \in \RT(\mu,\ell(\lambda))} \, \prod_{u\in [\mu]} (\lambda_{T(u)} -
c(u))\.,
\end{equation*}
where $c(u)=j-i$ is the content of $u=(i,j)$.  The conditions on tableaux~$T$
imply that the numerators here non-negative.

\begin{example} \label{ex:OOF}
For $\lambda/\mu = (2^31/1^2)$, the reverse semistandard tableaux of shape~$(1^2)$ with
entries $\{1,2,3,4\}$ are
$$
\young(2,1)\,,  \quad \young(3,1)\,, \quad
\young(3,2)\,, \quad \young(4,1)\,, \quad \young(4,2)\,, \quad \young(4,3);
$$
and the contents are $c(0,0)=0$ and $c(1,0)=-1$. Thus, the Okounkov--Olshanski
formula gives:
\[f^{(2^31/1^2)} \, = \,
\frac{5!}{5\cdot 4\cdot 3\cdot 3\cdot 2\cdot 1\cdot 1} \bigl( 2\cdot 3 +
2\cdot 3 + 2\cdot 3 + 1\cdot 3 + 1\cdot 3 + 1\cdot 3 \bigr) \, = \. 9
\]
(cf.~Example~\ref{ex:excited-def}).
Note that the \eqref{eq:OO} is asymmetric. For example, for
$\lambda'/\mu'=(43/2)$, there are two reverse tableaux of shape $(2)$ with
entries $\{1,2\}$.
\end{example}

It is illustrative to compare the NHLF and the OOF for the shape
$\lambda/(1)$ since $f^{\lambda/(1)} = f^{\lambda}$. The
excited diagrams $\ED(\lambda/(1))$ consist of single boxes of the
diagonal $\dd_0$ of $\lambda$, thus the NHLF gives
\[
f^{\lambda/(1)} = \frac{(|\lambda|-1)!}{\prod_{u\in [\lambda]} h(u)}
\,\,\left[\sum_{i} h(i,i)\right].
\]
On the other hand, the reverse tableau $\RT((1),\ell(\lambda))$ are of the form
$T= \young(i)$ for $1\leq i\leq \ell(\lambda)$. For each of these
tableaux $T$ we have $\lambda_{T(1,1)}=\lambda_i$ and $c(1,1)=0$, thus the \eqref{eq:OO} gives
\[
f^{\lambda/(1)} = \frac{(|\lambda|-1)!}{\prod_{u\in [\lambda]} h(u)}
\,\,\left[\sum_{i=1}^{\ell(\lambda)} \lambda_i \right].
\]
Note that in both cases $\sum_{i} h(i,i) = \sum_i \lambda_i =
|\lambda|$, confirming that $f^{\lambda/(1)} = f^{\lambda}$, however
the summands involved in both formulas are different in number and kind.

Chen and Stanley~\cite{CS} found a SSYT \ts $q$-analogue of the \eqref{eq:OO}.
Their proof is algebraic; they also give a bijective proof for shapes $\lambda/(1)$.
It would be very interesting to find a bijective proof of the formula and its
$q$-analogue in full generality.  Note that again, there is no RPP $q$-analogue in
this case. On the positive side, the sizes $|\RT(\mu,\ell)|$ are easy to compute
as the number of bounded SSYT of the (rectangle) complement
shape~$\overline\mu$; we omit the details.

\subsection{Formulas from rules for equivariant Schubert structure constants}

In this section we sketch how there is a formula for $f^{\lambda/\mu}$
for every rule of equivariant Schubert structure constants, a
generalization of the Littlewood--Richardson coefficients.

The {\em equivariant Schubert structure constants}
$C_{\mu,\nu}^{\lambda}({\bf y}):=C_{\mu,\nu}^{\lambda}(y_1,\ldots,y_n)$ are polynomials in
$\mathbb{Z}[y_1,\ldots,y_n]$ of degree $|\mu|+|\nu| - |\lambda|$
defined by the multiplication of {\em equivariant Schubert classes}
$\sigma_{\mu}$ and $\sigma_{\nu}$ in the $T$-equivariant cohomology
ring $H_T(X)$ (see \cite{KT,TY,Knu}). When $|\mu|+|\nu| = |\lambda|$
the degree zero polynomials $C_{\mu,\nu}^{\lambda}({\bf y})$ equal the Littlewood--Richardson
coefficients $c_{\mu,\nu}^{\lambda}$.

The Kostant polynomial $[X_w]|_v = \sigma_w(v)$ from
Section~\ref{sec:algproof} for Grassmannian permutations $w \preceq v$
corresponding to partitions $\mu\subseteq \lambda \subset d \times
(n-d)$ is also equal to $C_{\mu,\lambda}^{\lambda}({\bf y})$,
see~\cite[\S5]{Bil} and~\cite{Knu}.

The proof of the NHLF outlined by Naruse in \cite{Strobl} is based on the
following two identities.

\begin{lemma} \label{lem:clamlamlam}
\begin{equation} \label{eq:clamlamlam}
(-1)^{|\lambda|} \left.C_{\lambda,\lambda}^{\lambda}({\bf y}) \right|_{y_i=i} =
\prod_{u\in [\lambda]} h(u).
\end{equation}
\end{lemma}

\begin{proof}
We use Theorem~\ref{thm:ikna1} for $\mu=\lambda \subseteq d \times (n-d)$, since the only
excited diagram in $\ED(\lambda/\lambda)$ is $[\lambda]$ then
\[
(-1)^{|\lambda|} C_{\lambda,\lambda}^{\lambda}({\bf y}) = \prod_{(i,j)\in
    [\lambda]} (y_{d+j-\lambda'_j} - y_{\lambda_i+d-i+1}).
\]
Evaluating this equation at $y_i=i$ gives the desired formula.
\end{proof}

\begin{lemma}[Naruse \cite{Strobl}, see also \cite{MPP2}] \label{lem:key2NHLF}
\[
(-1)^{|\lambda/\mu|} \,
\left.\frac{C_{\mu,\lambda}^{\lambda}({\bf
      y})}{C_{\lambda,\lambda}^{\lambda}({\bf y})} \,\,\right|_{y_i=i}   \,=\,
\frac{f^{\lambda/\mu}}{|\lambda/\mu|!}.
\]
\end{lemma}

First, we swiftly recover the hook-length formula for $f^{\lambda}$.

\begin{corollary}
Lemma~\ref{lem:key2NHLF} implies the HLF \eqref{eq:hlf}.
\end{corollary}

\begin{proof}
By combining Lemma~\ref{lem:key2NHLF} for $\mu=\varnothing$,
\eqref{eq:clamlamlam}, and  using
$C_{\varnothing,\lambda}^{\lambda}({\bf y})=1$ we obtain the HLF.
\end{proof}

Second, we obtain the NHLF in the following way. The excited diagrams that appear in the NHLF come from the rule
to compute $C_{\mu,\lambda}^{\lambda}({\bf y}) = [X_w]\ts \bigl|_v$ in Theorem~\ref{thm:ikna2}.
Moreover, by Lemma \ref{lem:key2NHLF} any rule to compute
$C_{\mu,\nu}^{\lambda}({\bf y})$ gives a formula for
$f^{\lambda/\mu}$. Below we outline two such rules: the \emph{Knutson--Tao
puzzle rule}~\cite{KT} and the \emph{Thomas--Yong jeu-de-taquin rule}~\cite{TY}.

\medskip

\subsubsection{{\bf Knutson--Tao puzzle rule}} \label{sec:KTpuzzles}

Consider the following eight puzzle pieces, the last one is called the
equivariant piece, the others are called ordinary pieces:
\begin{center}
\includegraphics[scale=0.7]{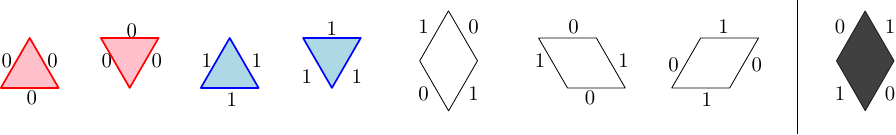}
\end{center}

Given partitions $\lambda,\mu,\nu \subseteq d\times (n-d)$ with $|\lambda|\geq |\mu| + |\nu|$
we consider a tilling of the triangle with edges labelled by the binary
representation of the subsets corresponding to $\nu,
\mu,\lambda$ in $\binom{[n]}{d}$ (clockwise starting from the left edge). To each
equivariant piece in a puzzle we associate coordinates $(i,j)$ coming
from the coordinates on the horizontal edge of the triangle form SW
and SE lines coming from the piece:
\begin{center}
\includegraphics[scale=0.8]{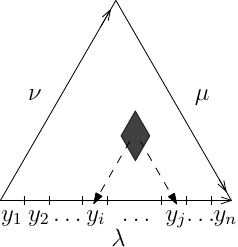}
\end{center}
We denote the piece with its coordinates by $p_{ij}$. The
weight $wt(P)$ of a puzzle $P$ is
\[
wt(P) \, = \, \prod_{p_{ij} \in P;\,\text{eq.}} \, (y_i-y_j),
\]
where the product is over equivariant pieces. Let
$\underset{\lambda}{{}^{\nu}\Delta^{\mu}}$ be the set of puzzles of a
triangle  boundary $\nu,\mu,\lambda$ (clockwise starting from the left
edge of the triangle). Knutson and Tao \cite{KT}
showed that $C^{\lambda}_{\mu,\nu}({\bf y})$ is the weighted sum of puzzles in
$\underset{\lambda}{{}^{\nu}\Delta^{\mu}}$.

\begin{theorem}[Knutson, Tao \cite{KT}]  For all $\la,\mu,\nu$ as above, we have:
\[
C^{\lambda}_{\mu,\nu}({\bf y}) \, = \, \sum_{P \in \, \underset{\lambda}{{}^{\nu}\Delta^{\mu}}} \. wt(P)\ts,
\]
where the sum is over puzzles of a triangle with boundary $\nu,\mu,\lambda$.
\end{theorem}

\begin{corollary}  For all skew shapes $\la/\mu$ as above, we have:
\begin{equation*} \tag{KTF}
f^{\lambda/\mu} \, = \, \frac{|\lambda/\mu|!}{\prod_{u\in [\lambda]} h(u)} \.
\sum_{P \in \, \underset{\lambda}{{}^{\lambda}\Delta^{\mu}}} \,\prod_{p_{ij} \in P;\,\text{eq.}} (j-i)\ts.
\end{equation*}
\end{corollary}

Knutson and Tao also showed that there is a unique puzzle $P_{\lambda}$ with boundary
$\underset{\lambda}{{}^{\lambda}\Delta^{\lambda}}$.  This gives us the following interesting
version of the HLF:

\begin{corollary} For all partitions~$\la$ with $\ell(\la)+\la_1=n$, we  have:
$$\prod_{p_{ij} \in P_{\lambda};\,\text{eq.}} \. (j-i) \, = \, \prod_{u\in [\lambda]} \.h(u)\ts.$$
\end{corollary}

\begin{example}
For $\lambda=(2^31)$ the puzzle $P_{2^31}$ with boundary
$\underset{2^31}{{}^{2^31}\Delta^{2^31}}$~ is:

\begin{center}
,
\end{center}
and $\prod_{p_{ij}\in P(2^31); \,\text{eq.}}
(j-i) = \prod_{u\in [2^31]} h(u) = 5\cdot 4\cdot 3^2\cdot 2$. For the skew shape $\lambda/\mu = (2^31/1^2)$ there are six puzzles with boundary  \,
$\underset{2^31}{{}^{2^31}\Delta^{1^2}}$~:
\smallskip
\begin{center}
%
.
\end{center}
\medskip

\noindent
Thus,
\[
f^{(2^31/1^2)} \, = \, \frac{5!}{5\cdot 4\cdot 3\cdot 3\cdot 2\cdot 1\cdot 1}
\bigl( 2\cdot 3 + 2\cdot 3 + 2\cdot 3 + 1\cdot 3 + 1\cdot 3 + 1\cdot 3 \bigl)\. = \.9\ts.
\]
This agrees term by term with the \eqref{eq:OO} (cf.~Example~\ref{ex:OOF}) and is
different from NHLF (cf.~Example~\ref{ex:excited-def}).  In full generality,
the connection is conjectured in~\cite{MPP3}.  Thus, both advantages and disadvantages
of~\eqref{eq:OO} possibly apply in this case as well.
\end{example}

\medskip

\subsubsection{{\bf Thomas--Yong jeu-de-taquin rule}}
Let $n=|\la|$. Consider all skew tableaux $T$ of shape $\lambda/\mu$ with labels
$1,2,\ldots,n$ where each label is either inside a box alone or on a
horizontal edge, not necessarily alone. The labels increase along
columns including the edge labels and along rows only for the
cells. Let $\TYYT(\lambda/\mu, n)$ be the set of these
tableaux. Denote by $T_{\lambda}$ be the {\em row superstandard tableau}
of shape $\lambda$ whose entries are
$1,2,\ldots,\lambda_1$ in the first row,
$\lambda_1+1,\lambda_1+2,\ldots,\lambda_1+\lambda_2$ in the second
row, etc.

Next we perform jeu-de-taquin on each of these tableau where an edge
label can move to an empty box above it, and no label slides to a
horizontal edge. In this jeu-de-taquin
procedure each edge label~$r$ starts right below a box $u_r$ and ends
in a box at row $i_r$. We associate a weight to each labelled edge~$r$
given by $y_{c(u_r)+\ell(\lambda)} - y_{\lambda_{i_r} -i_r +
  \ell(\lambda)+1}$. Denote by $\EqSYT(\lambda,\mu)$ the set of
  tableaux $T \in \TYYT(\lambda/\mu,n)$ that rectify to
$T_{\lambda}$.  Define the weight of each such $T$ by
\[
wt(T) \, = \, \prod_{r=1}^n \. \bigl(y_{c(u_r)+\ell(\lambda)} \ts - \ts y_{\lambda_{i_r} -i_r +\ell(\lambda)+1}\bigr).
\]

\smallskip

\begin{theorem}[Thomas, Yong \cite{TY}]
\[
C_{\mu,\lambda}^{\lambda}({\bf y}) \. = \. \sum_{T\in \EqSYT(\la,\mu)} \. wt(T)\..
\]
\end{theorem}

Specializing $y_i$ as in Lemma~\ref{lem:key2NHLF}, we get the following enumerative formula.

\begin{corollary}\label{cor:thomas_yong}
\begin{equation*} \label{eq:TYF} \tag{TYF}
f^{\lambda/\mu} \, = \, \frac{|\lambda/\mu|!}{\prod_{u\in [\lambda]} h(u)}
\, \sum_{T\in \EqSYT(\la,\mu)}  \, \prod_{r=1}^n \, \bigl(\lambda_{i_r}-i_r+1-c(u_r)\bigr)\ts.
\end{equation*}
% where the the product is over
% edge labels $r$ in~$T$.
\end{corollary}

Note an important disadvantage of~\eqref{eq:TYF} when compared to LR-coefficients
and other formulas:  the set of tableaux $\EqSYT(\lambda,\mu)$ does not have an easy description.
In fact, it would be interesting to see if it can be presented as the
number of integer points in some polytope, a result which famously holds in all
other cases.

\begin{example}
\ytableausetup{boxsize=0.15in}
Consider the case when $\lambda/\mu=(2^2/1)$. There are two tableaux
of shape $\lambda/\mu$ that rectify to the superstandard tableaux
$\ytableaushort{12,34}$ of weight $\lambda$:
\[
\ytableaushort{{\parbox{0.1in}{\vspace{0.1in} $\quad$ \\ 1}}2,{$\,$
    3}4} \qquad \text{and} \qquad  \ytableaushort{{$\quad$ }2,{ 1 \parbox{0.1in}{\vspace{0.1in} $\quad$  \\ 3}}4}\,\,,
\]
where the first tableau has weight $(2-1+1 - (0))=2$ corresponding to edge
label~$1$, and the second tableau with weight $(2-2+1 - (-1))=2$
corresponding to edge label~$3$.  By Corollary~\ref{cor:thomas_yong}, we have
$$f^{(2^2/1)} \, = \, \frac{3!}{3\cdot 2\cdot 2\cdot 1} \. (2+2) \. = \.2\ts.$$

Comparing with the terms from the NHLF, we have 2 excited diagrams
\includegraphics[scale=0.8]{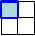}  which contributes a weight 3
(hook length of the blue square) and \includegraphics[scale=0.8]{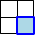} which contributes weight~1, so
$$
f^{(2^2/1)} \, = \,  \frac{3!}{3\cdot 2\cdot 2\cdot 1} \. (3+1)\..
$$

As this example illustrates, the Thomas--Young formula~\eqref{eq:TYF}
and the NHLF have different terms,
and thus neither equivalent nor easily comparable.
\end{example}

\subsection{The Naruse hook-length formula}\label{ss:compare-nhlf}
In lieu of a summary, the NHLF has both SSYT and RPP $q$-analogue. The
RPP $q$-analogue~\eqref{eq:skew-RPP} is proved
fully bijectively.  For the SSYT $q$-analogue, we do not have a description of the (restricted) inverse map $\GH = \HG^{-1}$
to give a fully bijective proof of~\eqref{eq:skewschur}. Instead we prove that the (restricted)
Hillman--Grassl map is bijective in this case in part via an algebraic argument.
We believe that map~$\GH$ can in fact be given an explicit description, but perhaps
the resulting bijective proof would be more involved (cf.~\cite{NPS}).
% One can question whether this combination \emph{constitutes} a ``bijective proof''.
% There is no doubt, however, that it implies the \emph{existence} of a ``bijective proof'',
% since the inverse map must have a ``description'' even if it is more involved
% (cf.~\cite{NPS}).  Whether this means we \emph{have} a ``bijective proof''
% is an interesting philosophical question we set aside.
The NHLF has
a combinatorial proof (via the RPP $q$-analogue and combinatorics of excited and
pleasant diagrams), but no direct bijective proof.  It is also a summation over
a set $\ED(\la/\mu)$ which is easy to compute (Corollary~\ref{cor:GV}).
As a bonus it has common generalization with Stanley and Gansner's trace formulas
(see~$\S$\ref{ss:intro-trace}).
%More curiously, it gives explicit combinatorial formulas for the thick strip cases
%(see $\S$\ref{ss:intro-enum} and~$\S$\ref{ss:finrem-foulkes}).

\bigskip

\section{Final remarks} \label{sec:finrem}

\subsection{} \label{ss:finrem-future}
This paper is the first in a series dedicated to the study of the NHLF  and contains most of
the arXiv preprint~\cite{MPP1}, except for the enumerative applications.
The latter are expanded and further generalized in~\cite{MPP2},
including a distinct proof of the NHLF~\eqref{eq:Naruse} using the
Lindstr\"om--Gessel--Viennot lemma.
Generalization to multivariate formulas, product formulas for the
number of SYT of certain skew shapes, and connections to lozenge tilings
can be found in~\cite{MPP3}.  A generalization to
Grothendieck polynomials including the most unusual extension
of the HLF to the number of increasing tableaux will appear
in~\cite{MPP4}. Asymptotic applications of
the NHLF and other formulas for $f^{\lambda/\mu}$ can be found in \cite{MPP5}.
Let us mention that while these next papers in this series rely
on the current work, they are largely independent from each other.

\subsection{} \label{ss:finrem-hist-survey}
There is a very large literature on the number of SYT of both straight
and skew shapes.  We refer to a recent comprehensive survey~\cite{HEC} of this
fruitful subject.  Similarly, there is a large literature on enumeration of
plane partitions, both using bijective and algebraic arguments.  We refer to
an interesting historical overview~\cite{Kra-survey} which begins with
MacMahon's theorem and ends with recent work on ASMs and perfect matchings.

\subsection{}  \label{ss:finrem-GNW}
%Probabilistic proof of Naruse's formula}
%
As we mention in the introduction, there are many proofs of the HLF,
some of which give rise to generalizations and pave interesting
connections to other areas (see e.g.\
\cite{Ban,CKP,GNW,Kratt-invol,NPS,P1,Rem,Ver}).  Unfortunately,
none of them easily adapt to skew shapes.  Ideally, one would want
to give a NPS-style bijective proof of the NHLF
(Theorem~\ref{thm:IN}).
In 2017, Konvanlinka \cite{Ko} gave a combinatorial proof of the
NHLF via a {\em bumping algorithm} (see also \cite[\SS 10.2]{MPP2}).

Recall that Stanley's Theorem~\ref{thm:HG} is a special case of more general
{\em Stanley's hook-content formula} for $s_{\lambda}(1,q,\ldots,q^{M})$
(see e.g.~\cite[\S 7.21]{EC2}).
Krattenthaler was able to combine the Hillman--Grassl
correspondence with the jeu-de-taquin and the NPS correspondences
to obtain bijective proofs of the hook-content formula~\cite{Kratt,Kra-another}.
Is there a NHLF-style hook-content formula for
$s_{\lambda/\mu}(1,q,\ldots,q^{M})$? See Section~\ref{sec:bounded_parts} for a version
for border strips and a discussion for general skew shapes.

In a different direction, the hook-length formula for $f^{\lambda}$ has a
celebrated probabilistic proof~\cite{GNW}.  If an NPS-style proof is too
much to hope for, perhaps a GNW-style proof of the NHLF would be more natural
and as a bonus would give a simple way to sample from $\SYT(\la/\mu)$
(as would the NPS-style proof, cf.~\cite{SSG}).  Such algorithm would be
theoretical and computational interest.  Note that for general posets~$\cP$
on~$n$ elements, there is a $O(n^3\log n)$ time MCMC algorithm for
perfect sampling of linear extensions of~$\cP$~\cite{Huber}.

\subsection{} \label{ss:finrem-hist}
The excited diagrams were introduced independently in \cite{IkNa09} by
Ikeda--Naruse and in~\cite{VK,VK2} by Kreiman in the context of
equivariant cohomology theory of Schubert varieties (see
also~\cite{GK, IkNa12}).
For skew shapes coming from {\em vexillary permutations},
they also appear in terms of {\em pipe dreams} or {\em rc-graphs} in the work of
Knutson, Miller and Yong~\cite[\S 5]{KMY}, who  used these objects to
give a formula for {\em double Schubert polynomials} of such permutations.

\subsection{}  \label{ss:finrem-Lam}
As we mention in the previous section, RPP typically do not arise in the
context of symmetric functions.  A notable exception is the recent
work by Lam and Pylyavskyy~\cite{LaPy}, who defined a symmetric function
$g_{\lambda/\mu}({\bf x})$ in terms of RPP of shape $\lambda/\mu$, and
have a LR-rule \cite{Ga}.
However, these functions are not homogeneous and the specialization
$g_{\lambda/\mu}(1,q,q^2,\ldots)$ is different than our RPP \ts $q$-analogue.

\subsection{}  \label{ss:finrem-complexity}
%Complexity of computing the number of pleasant diagrams}
%
By Corollary~\ref{cor:GV}, the number of excited
diagrams of $\lambda/\mu$ can be computed with a determinant of binomials. Thus
$|\ED(\lambda/\mu)|$ can be computed in polynomial time.
This raises a question
whether $|\PD(\lambda/\mu)|$ can be computed efficiently
(see Section~\ref{sec:pleasant}).  Perhaps,
Theorem~\ref{thm:num_pleasant} can be applied in the general case.

%\subsection{}  \label{ss:finrem-cat}
%The curious Catalan determinant in Corollary~\ref{cor:cat-det} is both
%similar and related\footnote{The connection was found by T.~Amdeberhan (personal communication).}
%to another Catalan determinant
%in~\cite[proof of Lemma~1.1]{AL}.
%In fact, both determinants are special cases of more general counting results,
%and both can be proved by the the Lindstr\"om--Gessel--Viennot lemma.

% \subsection{} \label{ss:finrem-combin-proof}
%{Combinatorial proof of Theorem~\ref{thm:skewSSYT}}
%
% In Section~\ref{sec:HGSSYT} we proved Theorem~\ref{thm:bij}: that the
% (restricted) Hillman--Grassl map is a bijection between SSYT of skew
% shape and excited arrays. The proof uses Theorem~\ref{thm:skewSSYT},
% proved algebraically and Theorem~\ref{thm:bij}, proved
% combinatorially. A combinatorial proof of the converse of the latter:
% i.e. showing that $\HG(SYT(\lambda/\mu))
% \subseteq \bigcup_{D\in \ED(\lambda/\mu)} \mathcal{A}^*_D$ would yield
% a combinatorial proof of Theorem~\ref{thm:skewSSYT}. One first
% step is given a SSYT (SYT) $T$ of shape
% $\lambda/\mu$, determine what is its associated excited diagram of the array
% $\HG(T)$?

 \subsection{}  \label{ss:finrem-shifted}
%{Naruse's formula for skew shifted shapes}
%
Along with Theorem~\ref{thm:IN}, Naruse also announced two formulas for the number
$g^{\lambda/\mu}$ of standard tableaux of skew shifted shape~$\lambda/\mu$,
in terms of {\em type $B$} and {\em type $D$} excited diagrams,
respectively. In fact, as one of the reviewers pointed out, the derivation of
the NHLF is valid for all {\em minuscule types}.
\nin
It would be of interest to find both $q$-analogues of these formulas,
as in theorems~\ref{thm:skewSSYT} and~\ref{thm:skewRPP}.
Let us mention that while some arguments translate to the
shifted case without difficulty (see e.g.~\cite{Kratt-invol,Sag-hook}),
in other cases this is a major challenge (see e.g.~\cite{Fis}).

\vskip.6cm

\subsection*{Acknowledgements}
We are grateful to Per Alexandersson, Dan Betea, Sara Billey, Brian
Chan, Leonid
Petrov, Robert Proctor, Eric Rains, Luis Serrano, Richard Stanley,
Nathan Williams, Matthew Willis, 
and Alex Yong and the referees
for useful comments and help with the references.
% We are also thankful to
%Tewodros Amdeberhan, Brendon Rhoades and Emily Leven for discussions on the
%curious Catalan determinants.
The puzzles in
Section~\ref{sec:KTpuzzles} were done using Sage and its algebraic combinatorics features
developed by the Sage-Combinat community \cite{sage-combinat}.
The first author was partially supported by an AMS-Simons travel
grant. The second and third authors were partially supported by the~NSF.

\end{document}